\documentclass[a4paper,11pt]{article}

\usepackage{mathrsfs}
\usepackage{amsmath}
\usepackage{bm}
\usepackage{amssymb}
\usepackage{makeidx}
\makeindex
\usepackage[dvipsnames]{xcolor}
\usepackage[colorlinks=true, linkcolor=blue, citecolor=blue, urlcolor=blue]{hyperref}
\usepackage{titlesec}
\usepackage{amsthm}
\usepackage{setspace}
\usepackage{pifont}
\usepackage{esint}
\usepackage{graphicx}
\usepackage{float}
\usepackage{enumerate}
\usepackage{array}
\numberwithin{equation}{section}
\usepackage[backend=biber,style=alphabetic,sorting=nyt]{biblatex}
\DeclareFieldFormat[article]{title}{#1}
\renewbibmacro{in:}{}
\DeclareFieldFormat{pages}{#1}
\DeclareFieldFormat[article]
{journaltitle}{\mkbibemph{#1}\addcomma}
\renewbibmacro*{volume+number+eid}{%
	\printfield{volume}%
	\setunit{\addspace}%
	\printfield{eid}%
}
\renewbibmacro*{eprint}{%
	\iffieldundef{eprint}
	{}
	{%
		\setunit{\addcomma\space}
		\printtext{%
			\printfield[eprint:arxiv]{eprint}}%
	}%
}

\titleformat{\section}[block]{\bfseries\Large}{\thesection}{1em}{}

\newtheoremstyle{mystyle}
{}{}{\upshape}{}{\bfseries}{.}{.5em}{}

\theoremstyle{mystyle}

\newtheorem{theorem}{Theorem}[section]
\newtheorem{proposition}[theorem]{Proposition}
\newtheorem{lemma}[theorem]{Lemma}
\newtheorem{coro}[theorem]{Corollary}
\newtheorem{definition}[theorem]{Definition}

\newtheorem{remark}[theorem]{Remark}

\newcommand{\R}{\mathbb{R}}
\newcommand{\Z}{\mathbb{Z}}
\newcommand{\N}{\mathbb{N}}

\newcommand{\cP}{\mathcal{P}}

\newcommand{\cF}{\mathcal{F}}
\newcommand{\cS}{\mathcal{S}}

\begin{document} 
	
	\title{\large
The Fractional-Logarithmic Laplacian:\\
Potentials, Regularity, and Critical Compact Embeddings}
	\author{\small Rui Chen}
	\date{}
	\maketitle
	\thispagestyle{empty}
	\pagenumbering{arabic}
	
\noindent\textbf{Abstract:} We develop potential-theoretic and \(L^p\)-regularity results for the
fractional--logarithmic Laplacian \((-\Delta)^{s+\ln}\) and its inhomogeneous
counterpart \((\lambda I-\Delta)^{s+\ln}\), \(\lambda>1\). These operators lead
to logarithmic analogues of the classical Riesz and Bessel potentials. For the
associated logarithmic Bessel kernel \(K_{s+\ln}^{\lambda}\), we obtain
representation formulas and sharp pointwise asymptotics at both the origin and
infinity, including explicit leading constants.

A key ingredient is a measure-level bridge between the homogeneous and
inhomogeneous symbols. This allows us to pass between the equations $(\lambda I-\Delta)^{s+\ln}u=f$ and $(-\Delta)^{s+\ln}u=f,$
and yields global \(L^p\) estimates, well-posedness for distributional
solutions, and a natural scale of logarithmic Bessel spaces
\(\mathcal L^p_{s+\ln,\lambda}\). We also discuss the dependence of these spaces
on \(\lambda\), their relation to the classical Bessel spaces and with the logarithmic Bessel potential spaces in \cite{opic2000bessel}.

As applications, we prove endpoint embeddings and critical compactness results.
On the critical line \(n=2sp\), we obtain embeddings with a logarithmic modulus
of continuity, local compactness on bounded domains, and global compactness in
the radial class. In the subcritical case \(n>2sp\), we prove compactness at the
critical Sobolev exponent $p^*=\frac{np}{n-2sp},$
recovering compactness at the borderline Lebesgue threshold, a phenomenon absent
from the classical Sobolev and Bessel scales.

		\medskip
        
	

	\tableofcontents
	\thispagestyle{empty}

	\setcounter{page}{1}
	
\section{Introduction and Main Results}

In this paper we develop a systematic potential-theoretic and $L^p$ regularity theory for two closely related pseudo-differential operators: the fractional--logarithmic Laplacian
\[
(-\Delta)^{s+\ln},
\qquad
\text{with Fourier symbol }(4\pi^2|\xi|^2)^s\ln(4\pi^2|\xi|^2),
\]
and its inhomogeneous counterpart
\[
(\lambda I-\Delta)^{s+\ln},
\qquad
\text{with Fourier symbol }(\lambda+4\pi^2|\xi|^2)^s\ln(\lambda+4\pi^2|\xi|^2),
\quad \lambda>1.
\]
Our point of view is that these two operators should be treated simultaneously as the homogeneous and inhomogeneous models of a single logarithmic smoothness scale. Their inverses then give rise to logarithmic analogues of the classical Riesz and Bessel potentials, and this provides the natural functional framework for the analysis developed below.

The operator $(-\Delta)^{s+\ln}$ has only very recently begun to be studied as a genuine nonlocal operator in its own right. In the range $s\in(0,1)$, the recent work~\cite{chen2026fractional} establishes several equivalent realizations of $(-\Delta)^{s+\ln}$, develops the associated energy spaces on $\mathbb R^n$ and on bounded domains, and investigates qualitative spectral properties, including a Dirichlet eigenvalue problem and a Weyl-type asymptotic law with a logarithmic correction. More broadly, $(-\Delta)^{s+\ln}$ can be viewed as a new nonlocal operator interpolating between fractional and logarithmic regimes, and it is natural to expect that it will play an increasingly important role in analysis and PDE.

This expectation is motivated by the remarkable success of related nonlocal operators. For the fractional Laplacian $(-\Delta)^s$, there is by now a vast literature on regularity, eigenvalue problems, semilinear and critical equations, free boundary problems, conformal geometry, and nonlocal minimal surfaces; see, for example, \cite{RosOtonSerra2014,caffarelli2008regularity,silvestre2007regularity,RosOtonWeidner2026,caffarelli2007extension,RosOtonSerra2015,CaffarelliRosOtonSerra2017,chang2011fractional,caffarelli2010nonlocal,chen2026poincar} and the references therein. For the logarithmic Laplacian, following the pointwise representation obtained in \cite{chen2019dirichlet}, an active theory has developed around boundary value problems, qualitative properties of solutions, and related nonlocal phenomena; see, for instance, \cite{chen2019dirichlet,jarohs2020poisson,chenveron2024cauchy,frank2020classification,ch2025logarithmic,chen2025logarithmic,gerontogiannis2025ahlfors}.

On the inhomogeneous side, the logarithmic Bessel potential
\[
(\lambda+4\pi^2|\xi|^2)^{-s}\ln^{-1}(\lambda+4\pi^2|\xi|^2),
\qquad \lambda>1,\ s>0,
\]
has a longer history and appears naturally in the theory of spaces of logarithmic smoothness. Two complementary viewpoints have been especially influential.

\smallskip
\noindent
\emph{(i) Logarithmic refinement of the base integrability scale.}
One keeps the classical Bessel potential operator but replaces Lebesgue spaces by finer Lorentz--Zygmund type spaces; see, e.g., \cite{1996sharpness,1997embeddings,1998norms,edmunds2000optimality,cobos2015besov}.

\smallskip
\noindent
\emph{(ii) Logarithmic refinement of the smoothness scale.}
One inserts slowly varying logarithmic factors into the Bessel multiplier, thereby obtaining spaces of logarithmic smoothness that sharpen critical Sobolev and H\"older thresholds; see, e.g., \cite{opic2000bessel,dominguez2023function}.

\smallskip
Our approach is closer to~(ii), but differs from much of the earlier literature in two essential respects. First, we treat the logarithmic Bessel potential and the fractional--logarithmic Laplacian as the inhomogeneous and homogeneous members of a single pseudo-differential framework, rather than as isolated logarithmic refinements of known embedding scales. Second, and more importantly, we identify an explicit structural bridge between the inhomogeneous and homogeneous symbols. This bridge is one of the main analytic mechanisms of the paper and allows us to transfer information between logarithmic Bessel and logarithmic Riesz potentials.

From the potential-theoretic viewpoint, the operators
\[
(-\Delta)^{-(s+\ln)}
\qquad\text{and}\qquad
(\lambda I-\Delta)^{-(s+\ln)},\quad \lambda>1,
\]
should therefore be regarded as the logarithmic counterparts of the classical Riesz and Bessel potentials. Indeed, whereas the standard Riesz and Bessel potentials correspond to the multipliers
\[
(4\pi^2|\xi|^2)^{-s}
\qquad\text{and}\qquad
(\lambda+4\pi^2|\xi|^2)^{-s},
\]
their logarithmic analogues are obtained by inserting the slowly varying factor $\ln(\cdot)$ into the same functional calculus, leading to the multipliers
\[
(4\pi^2|\xi|^2)^{-s}\bigl[\ln(4\pi^2|\xi|^2)\bigr]^{-1}
\qquad\text{and}\qquad
(\lambda+4\pi^2|\xi|^2)^{-s}\bigl[\ln(\lambda+4\pi^2|\xi|^2)\bigr]^{-1}.
\]

A central theme of the paper is that this logarithmic correction has a decisive effect at critical thresholds. Roughly speaking, it weakens the singularity of the classical kernel just enough to restore strong endpoint integrability. This gain places the analysis in a robust Young-type convolution framework and ultimately yields compactness phenomena at critical exponents that are absent in the classical Sobolev and Bessel scales. One of the main consequences is the recovery of compactness at the critical Sobolev exponent
\[
p^*=\frac{np}{n-2sp}
\]
in the subcritical regime $n>2sp$, both locally on bounded domains and globally under radial symmetry.

\medskip

To investigate the inhomogeneous logarithmic potential, we introduce the associated \emph{logarithmic Bessel kernel} $K_{s+\ln}^{\lambda}(x)$ (see Definition~\ref{kernel1}). Our first objective is to determine its sharp asymptotic behaviour both as $|x|\to0$ and as $|x|\to\infty$. As a benchmark, recall that the classical Bessel kernel $G_s$, namely the convolution kernel of $(I-\Delta)^{-s}$, has the well-known near-field trichotomy and an exponentially decaying far-field profile; see Section~\ref{bessel poten} for background. More precisely, for every $s>0$,
\begin{equation}\label{eq:Gs-infty-asympt}
\lim_{|x|\to\infty} |x|^{\frac{n+1}{2}-s}\,e^{|x|}\,G_s(x)
=\frac{1}{2^{\frac{n-1}{2}+s}\,\pi^{\frac{n-1}{2}}\,\Gamma(s)},
\end{equation}
see \cite[Lemma~4.15]{abatangelo2025gentle}, whereas near the origin one has, for all
$x\in B_5\setminus\{0\}$,
\begin{equation}\label{eq:Gs-zero-trichotomy}
|G_s(x)|
\le
\begin{cases}
\dfrac{C}{|x|^{\,n-2s}}, & n>2s,\\[2mm]
C\bigl(1+|\ln|x||\bigr), & n=2s,\\[1mm]
C, & n<2s,
\end{cases}
\end{equation}
for some constant $C=C(n,s)>0$; see \cite[Lemma~4.17]{abatangelo2025gentle}. In the classical subcritical case $n>2s$, this singularity is exactly too strong to yield strong endpoint integrability at $L^{\frac{n}{n-2s}}$, which is one of the basic analytic obstructions behind critical Sobolev embeddings.

The logarithmic kernel $K_{s+\ln}^{\lambda}$ turns out to be qualitatively different. Near the origin, the logarithmic correction weakens the leading singularity by an additional slowly varying factor; at infinity, it preserves an exponential profile with an exact leading term. This combination is precisely what drives the endpoint estimates in the later sections.

To obtain a workable representation of $K_{s+\ln}^{\lambda}$, we introduce a family of shifted Bessel kernels. For $\alpha>0$ and $\lambda>0$, let $G^\lambda_{\alpha}$ denote the tempered distribution whose Fourier transform is
\begin{equation}\label{eq:def-shifted-bessel}
\widehat{G^\lambda_{\alpha}}(\xi)
=(\lambda+4\pi^2|\xi|^2)^{-\alpha},
\qquad \xi\in\mathbb R^n.
\end{equation}
Arguing exactly as in the classical case (see \cite[Chapter~V]{stein1970singular}) and using $\widehat{p_t}(\xi)=e^{-4\pi^2 t|\xi|^2}$, one obtains the heat kernel representation
\begin{equation}\label{eq:G-lambda-heat}
G^\lambda_{\alpha}(x)
=\frac{1}{\Gamma(\alpha)}\int_{0}^{\infty} e^{-\lambda t}\,t^{\alpha-1}\,p_t(x)\,dt,
\quad x\in\mathbb R^n,
\end{equation}
where
\begin{equation}\label{gaussian}
p_t(x)
=(4\pi t)^{-n/2}e^{-\frac{|x|^2}{4t}},
\qquad t>0,\ x\in\mathbb R^n.
\end{equation}
This representation shows that $G^\lambda_\alpha$ is a Gamma-type mixture of Gaussians and leads to a useful Laplace-type asymptotic as $\alpha\to\infty$.

\begin{proposition}\label{prop:Galp-laplace-asymp}
Let $\lambda>0$. For every fixed $x\in \mathbb{R}^n\setminus \left\{0\right\}$,
\[
G^\lambda_\alpha(x)
=\lambda^{-\alpha}\,p_{\alpha/\lambda}(x)\,\Bigl(1+O(\alpha^{-1})\Bigr),
\qquad \alpha\to\infty,
\]
where $p_t$ is the Gaussian kernel given in \eqref{gaussian}. Equivalently,
\[
G^\lambda_\alpha(x)
=\lambda^{-\alpha}\,(4\pi\alpha/\lambda)^{-n/2}
\exp\!\Bigl(-\frac{\lambda|x|^2}{4\alpha}\Bigr)
\Bigl(1+O(\alpha^{-1})\Bigr),
\qquad \alpha\to\infty.
\]
Moreover, the $O(\alpha^{-1})$ error is uniform for $x$ in compact subsets of
$\mathbb R^n$.
\end{proposition}

As an immediate consequence, for every $s>0$ and $\lambda>1$ the improper integral
$\int_s^\infty G^\lambda_{\alpha}(x)\,d\alpha$ converges for each fixed $x\in\mathbb R^n$. Since $G^\lambda_\alpha\ge0$ and $\|G^\lambda_\alpha\|_{L^1(\mathbb R^n)}=\lambda^{-\alpha}$, Tonelli's theorem yields
\[
\int_{\mathbb R^n}\int_s^\infty G^\lambda_\alpha(x)\,d\alpha\,dx
=\int_s^\infty \lambda^{-\alpha}\,d\alpha
=\frac{\lambda^{-s}}{\ln\lambda}<\infty.
\]
Hence $x\mapsto \int_s^\infty G^\lambda_{\alpha}(x)\,d\alpha$ defines a nonnegative
$L^1(\mathbb R^n)$-function. The next result identifies it with the logarithmic Bessel kernel.

\begin{proposition}\label{lem:log-bessel-kernel}
Let $s>0$ and $\lambda>1$. Then, for every $x\in \mathbb{R}^n\setminus \left\{0\right\}$,
\begin{equation}\label{eq:K-log-alpha}
K_{s+\ln}^{\lambda}(x)=\int_s^\infty G^\lambda_{\alpha}(x)\,d\alpha\in L^1(\mathbb{R}^n).
\end{equation}
Equivalently,
\begin{equation}\label{eq:K-log-heat}
K_{s+\ln}^{\lambda}(x)
=\int_s^\infty \frac{1}{\Gamma(\alpha)}
\int_0^\infty e^{-\lambda t}\,t^{\alpha-1}\,p_t(x)\,dt\,d\alpha,
\end{equation}
where $p_t$ is the Gaussian kernel from \eqref{gaussian}.
\end{proposition}

Our first main result determines the sharp local asymptotics of $K_{s+\ln}^{\lambda}$ at the origin. In contrast with the classical Bessel kernel, whose singularity is of pure power type (or logarithmic in the borderline case), the logarithmic perturbation produces a genuinely new regime: the leading singularity is weakened by an additional logarithmic factor. This is precisely the mechanism behind the endpoint mapping properties used later in the paper.

\begin{proposition}\label{prop:kernel_asymptotics}
Let $n\ge1$, $\lambda>1$, and $s>0$. The asymptotic behavior of the logarithmic
Bessel kernel $K_{s+\ln}^{\lambda}(x)$ at the origin falls into three regimes.

\medskip
\noindent\textnormal{(i)} \textbf{Singular regime ($n>2s$).}
The kernel has a logarithmically moderated Riesz-type singularity:
\[
\lim_{|x|\to0}
K_{s+\ln}^{\lambda}(x)\,
\frac{\ln\!\bigl(\frac{1}{|x|^2}\bigr)}{|x|^{2s-n}}
=\frac{\Gamma(\frac n2-s)}{\pi^{n/2}\,2^{2s}\,\Gamma(s)}.
\]
In particular, the limit constant coincides with the normalization constant of the
classical Riesz potential kernel of order $2s$.

\medskip
\noindent\textnormal{(ii)} \textbf{Critical regime ($n=2s$).}
The kernel exhibits a double-logarithmic blow-up:
\[
\lim_{|x|\to0}
K_{\frac n2+\ln}^{\lambda}(x)\,
\frac{1}{\ln\!\Bigl(\ln\!\bigl(\tfrac{1}{|x|}\bigr)\Bigr)}
=\frac{1}{(4\pi)^{\frac n2}\,\Gamma(\frac n2)}.
\]

\medskip
\noindent\textnormal{(iii)} \textbf{Continuous regime ($n<2s$).}
The kernel extends continuously to the origin and remains bounded:
\[
\lim_{|x|\to0} K_{s+\ln}^{\lambda}(x)
=\frac{1}{(4\pi)^{\frac n2}}
\int_0^\infty
\frac{\Gamma\!\bigl(s+p-\frac n2\bigr)}{\Gamma(s+p)}\,
\lambda^{\frac n2-s-p}\,dp
<\infty.
\]
\end{proposition}

We next turn to the far-field regime. This part requires a different
asymptotic mechanism from the near-field analysis. In the near-field case, the
singularity is produced by the small-time behaviour of the heat kernel, whereas
in the far-field regime the dominant contribution comes from a time scale which
moves to infinity together with \(|x|\). The main difficulty is therefore to
identify the correct moving scale and to prove that all remaining time regions
give lower-order contributions. We overcome this by combining the large-time
asymptotics of the auxiliary function 
\[
H_s(t):=\int_0^\infty \frac{t^p}{\Gamma(s+p)}\,dp
=
e^t t^{1-s}\left(1+O(t^{-1})\right),
\qquad t\to\infty
\]
with a one-dimensional Laplace
method. This reduces the problem to the analysis of a non-degenerate minimum of
the associated phase function and yields the far-field asymptotic below.

\begin{proposition}\label{lem:Klog-asymp-infty-final}
Let $s>0$ and $\lambda>1$. Then, as $r:=|x|\to\infty$, the kernel
$K_{s+\ln}^{\lambda}(x)$ satisfies
\begin{equation}\label{eq:Klog-asymp-final}
\lim_{r\to\infty}
\Bigl(r^{\frac{n-1}{2}}e^{\sqrt{\lambda-1}\,r}\Bigr)\,
K_{s+\ln}^{\lambda}(x)
=\frac{(\lambda-1)^{\frac{n-3}{4}}}{2^{\frac{n+1}{2}}\pi^{\frac{n-1}{2}}}.
\end{equation}
\end{proposition}

A striking feature of \eqref{eq:Klog-asymp-final} is that both the decay rate and the exact leading constant are independent of $s$. Combined with Proposition~\ref{prop:kernel_asymptotics}, this yields a decisive global integrability gain. In the classical case, the Riesz and Bessel kernels behave like $|x|^{2s-n}$ near the origin and therefore fail at the endpoint exponent $L^{\frac{n}{n-2s}}$ when $n>2s$. By contrast, Proposition~\ref{prop:kernel_asymptotics}\textnormal{(i)} shows that the logarithmic correction weakens the singularity just enough to restore strong endpoint integrability:
\[
K_{s+\ln}^{\lambda}\in L^{\frac{n}{n-2s}}_{\mathrm{loc}}(\mathbb R^n),
\qquad n>2s,
\]
while Proposition~\ref{lem:Klog-asymp-infty-final} provides exponential decay at infinity. Consequently,
\[
K_{s+\ln}^{\lambda}\in L^{q}(\mathbb R^n)
\qquad\text{for every }q\in\Bigl[1,\frac{n}{n-2s}\Bigr]
\]
in the subcritical regime $n>2s$. This strong endpoint integrability is one of the central analytic inputs of the paper.

\medskip

We now introduce the homogeneous counterpart of the logarithmic Bessel potential, namely the \emph{logarithmic Riesz potential}, and explain the functional setting in which both the potential and the operator $(-\Delta)^{s+\ln}$ are naturally defined. Let $s>0$. The logarithmic Riesz potential of order $s$ is the Fourier multiplier with symbol
\[
\widehat{K_{s+\ln}}(\xi)
:=
\frac{1}{(4\pi^2|\xi|^2)^s\,\ln(4\pi^2|\xi|^2)},
\qquad \xi\in\mathbb R^n\setminus\{0\}.
\]
A basic difficulty is that neither $\widehat{K_{s+\ln}}$ nor the symbol
\[
m_{s+\ln}(\xi):=(4\pi^2|\xi|^2)^s\,\ln(4\pi^2|\xi|^2)
\]
acts continuously on the full Schwartz space $\mathcal S(\mathbb R^n)$, because of the logarithmic singularity at $\xi=0$. In particular, the maps
$\psi\mapsto \widehat{K_{s+\ln}}\psi$ and $\psi\mapsto m_{s+\ln}\psi$ do not in general preserve $\mathcal S(\mathbb R^n)$, so a direct definition on $\mathcal S'(\mathbb R^n)$ would be ambiguous up to polynomials.

To avoid this difficulty, we work in the standard homogeneous Schwartz subspaces
$\mathcal S_\infty(\mathbb R^n)$ and $\mathcal S_0(\mathbb R^n)$; see
\cite[Section~2.3.4]{grafakos2008classical} and \cite[Section~5.1.2]{Triebel1983Theory}. A fundamental fact is that the Fourier transform restricts to a linear homeomorphism
\[
\mathcal F:\mathcal S_\infty(\mathbb R^n)\to\mathcal S_0(\mathbb R^n),
\qquad
\mathcal F^{-1}:\mathcal S_0(\mathbb R^n)\to\mathcal S_\infty(\mathbb R^n).
\]
This allows us to define the Fourier transform on $\mathcal S_\infty'$ and $\mathcal S_0'$ by duality; see Section~\ref{moment} for details. Within this homogeneous framework, the fractional--logarithmic Laplacian is defined unambiguously as a Fourier multiplier on $\mathcal S_\infty'$.

\begin{definition}\label{def:frac-log-lap}
Let $s>0$ and let $u\in\mathcal S_\infty'(\mathbb R^n)$. Define
$(-\Delta)^{s+\ln}u\in\mathcal S_\infty'(\mathbb R^n)$ by prescribing its Fourier
transform in $\mathcal S_0'(\mathbb R^n)$:
\[
\mathcal F\bigl((-\Delta)^{s+\ln}u\bigr)
:= m_{s+\ln}\,\widehat u
\quad\text{in }\mathcal S_0'(\mathbb R^n),
\]
that is, for every $\psi\in\mathcal S_0(\mathbb R^n)$,
\[
\big\langle \mathcal F\bigl((-\Delta)^{s+\ln}u\bigr),\psi\big\rangle
:=\big\langle \widehat u,\ m_{s+\ln}\psi\big\rangle.
\]
Equivalently,
\[
(-\Delta)^{s+\ln}u
=\mathcal F^{-1}\!\bigl(m_{s+\ln}\,\widehat u\bigr)
\qquad\text{in }\mathcal S_\infty'(\mathbb R^n).
\]
\end{definition}

When $s\in(0,1)$, this definition agrees with the fractional--logarithmic Laplacian studied in \cite{chen2026fractional}. More precisely, for $u\in C_c^2(\mathbb R^n)$ one has the singular-integral representation
\[
(-\Delta)^{s+\ln}u(x)
=
c_{n,s}\,\mathcal L_1u(x)
+
b_{n,s}\,(-\Delta)^s u(x),
\qquad x\in\mathbb R^n,
\]
where
\[
c_{n,s}:=2^{2s}\pi^{-\frac n2}s\,\frac{\Gamma\!\left(\frac{n+2s}{2}\right)}{\Gamma(1-s)},
\qquad
b_{n,s}:=\frac{d}{ds}c_{n,s}
=
\ln 4+\frac1s+\psi(1-s)+\psi\!\left(\frac{n+2s}{2}\right),
\]
and
\begin{equation}\label{eq:L1-def}
\mathcal L_1u(x)
:=
\mathrm{p.v.}\int_{\mathbb R^n}\frac{u(x)-u(y)}{|x-y|^{n+2s}}
\bigl(-2\ln|x-y|\bigr)\,dy.
\end{equation}
In particular,
\[
(-\Delta)^{s+\ln}u
=
\left.\frac{d}{dt}(-\Delta)^t u\right|_{t=s},
\qquad u\in C_c^2(\mathbb R^n),
\]
so the operator may be interpreted as the derivative of the fractional Laplacian with respect to the order.

\medskip

Having introduced both the homogeneous and inhomogeneous logarithmic potentials, we next construct an explicit structural bridge between them. In the classical theory, the difference between massive and massless symbols can often be absorbed into a regularizing convolution kernel; see, for example, \cite[Section~3.2, Lemma~2]{stein1970singular}. In the present logarithmic setting, however, the additional $\ln$-factor destroys the relevant algebraic simplifications, so a different mechanism is needed. Our first step is a general criterion ensuring that a bounded multiplier with controlled derivatives and a mild logarithmic singularity at the origin is the Fourier transform of an $L^1$-kernel.

\begin{proposition}\label{erjinzhi}
Let $n\ge 1$ and let $f\in L^\infty(\mathbb R^n)\cap C^N(\mathbb R^n\setminus\{0\})$,
where $N>n$ is an integer. Assume that there exist constants $C>0$, $\delta_1>0$,
$\delta_2>0$ such that for all multi-indices $|\alpha|\le N$,
\begin{align}
\label{eq:HF_assump}
|\partial^\alpha f(\xi)|
&\le C\,|\xi|^{-\delta_1-|\alpha|},
\qquad &&|\xi|\ge 1,\\
\label{eq:LF_assump}
|\partial^\alpha f(\xi)|
&\le C\,|\xi|^{\delta_2-|\alpha|}\bigl(1+|\ln|\xi||\bigr),
\qquad &&0<|\xi|\le 2.
\end{align}
Then there exists $g\in L^1(\mathbb R^n)$ such that $\widehat g(\xi)=f(\xi)$ for a.e. $\xi\in\mathbb R^n$.
\end{proposition}

Its proof is based on a dyadic Littlewood--Paley decomposition. The high-frequency decay guarantees summability on positive dyadic scales, while the low-frequency gain compensates for the logarithmic singularity near the origin.

Applying Proposition~\ref{erjinzhi} to a suitable transition multiplier yields the desired bridge between the inhomogeneous symbol
\[
(\lambda+4\pi^2|\xi|^2)^s\ln(\lambda+4\pi^2|\xi|^2)
\]
and its homogeneous counterpart
\[
(4\pi^2|\xi|^2)^s\ln(4\pi^2|\xi|^2).
\]

\begin{proposition}\label{prop:bridge-measures}
Let $n\ge1$, $\lambda>1$, and $s>0$. Define
\[
m(\xi)
:=\frac{(4\pi^2|\xi|^2)^s\,\ln (4\pi^2|\xi|^2)}
{(\lambda+4\pi^2|\xi|^2)^s\,\ln(\lambda+4\pi^2|\xi|^2)},
\qquad \xi\in\mathbb R^n,
\]
and extend $m$ continuously at $\xi=0$ by setting $m(0)=0$.
Then:

\noindent\textnormal{(i)} There exists $g_{s,\lambda}\in L^1(\mathbb R^n)$ such that
\begin{equation}\label{eq:m-splitting}
m(\xi)=1+\widehat g_{s,\lambda}(\xi),
\qquad \xi\in\mathbb R^n.
\end{equation}

\noindent\textnormal{(ii)} Consequently, there exist finite measures
$\mu_{s,\lambda}$ and $w_{s,\lambda}$ on $\mathbb R^n$ such that
\begin{equation}\label{eq:measure-decomp}
(\lambda + 4\pi^2|\xi|^2)^s \ln (\lambda + 4\pi^2|\xi|^2)
=\widehat{\mu_{s,\lambda}}(\xi)
+\widehat{w_{s,\lambda}}(\xi)\,(4\pi^2|\xi|^2)^s\,\ln (4\pi^2|\xi|^2),
\qquad \xi\in\mathbb R^n,
\end{equation}
and moreover $\widehat{\mu_{s,\lambda}},\widehat{w_{s,\lambda}}\in\mathcal O_M(\mathbb R^n)$,
where $\mathcal O_M(\mathbb R^n)$ is defined in \eqref{omom}.
\end{proposition}

This bridge is the key tool that allows us to compare regularity for the homogeneous equation with the inhomogeneous logarithmic Bessel scale.

\medskip

We now introduce the main spaces of the paper. Let $s>0$, $\lambda>1$, and
$1\le p\le\infty$. The \emph{logarithmic Bessel space} is defined by
\[
\mathcal{L}^{p}_{s+\ln,\lambda}(\mathbb R^n)
:=
\Bigl\{u\in\mathcal S'(\mathbb R^n):
(\lambda I-\Delta)^{s+\ln}u\in L^p(\mathbb R^n)\Bigr\},
\]
with norm
\[
\|u\|_{\mathcal{L}^{p}_{s+\ln,\lambda}(\mathbb R^n)}
:=
\bigl\|(\lambda I-\Delta)^{s+\ln}u\bigr\|_{L^p(\mathbb R^n)}.
\]
We also recall the classical Bessel potential space
\[
\mathcal{L}_s^p(\mathbb R^n)
:=
\Bigl\{u\in\mathcal S'(\mathbb R^n):
\mathcal F^{-1}\bigl((1+4\pi^2|\xi|^2)^{s/2}\widehat u(\xi)\bigr)\in L^p(\mathbb R^n)\Bigr\},
\]
equipped with the norm
\[
\|u\|_{\mathcal{L}_s^p(\mathbb R^n)}
:=
\Bigl\|
\mathcal F^{-1}\bigl((1+4\pi^2|\xi|^2)^{s/2}\widehat u\bigr)
\Bigr\|_{L^p(\mathbb R^n)}.
\]
Thus $(\lambda I-\Delta)^{s+\ln}$ plays the role of an inhomogeneous logarithmic differential operator of order $2s$, and $\mathcal{L}^{p}_{s+\ln,\lambda}$ is the corresponding logarithmic Bessel potential space.

Our first $L^p$ result gives well-posedness and representation for the inhomogeneous equation.

\begin{theorem}\label{thm:Lp-log-Bessel}
Let $s>0$, $\lambda>1$, and $1\le p\le\infty$, and let $f\in L^p(\mathbb R^n)$.
Suppose that $u\in\mathcal S'(\mathbb R^n)$ is a distributional solution of
\begin{equation}\label{eq:lambda-log-eq}
(\lambda I-\Delta)^{s+\ln}u=f
\qquad\text{in }\mathbb R^n.
\end{equation}
Then $u\in\mathcal L^{p}_{s+\ln,\lambda}(\mathbb R^n)$ and
\[
\|u\|_{\mathcal L^{p}_{s+\ln,\lambda}(\mathbb R^n)}=\|f\|_{L^p(\mathbb R^n)}.
\]
Moreover, the solution is uniquely determined and is given by the logarithmic
Bessel potential
\[
u=(\lambda I-\Delta)^{-(s+\ln)}f
=
K^{\lambda}_{s+\ln}*f.
\]
\end{theorem}

The next theorem concerns the homogeneous equation driven by $(-\Delta)^{s+\ln}$.
Using Proposition~\ref{prop:bridge-measures}, we compare the homogeneous operator with its inhomogeneous counterpart and obtain an a priori estimate in the logarithmic Bessel scale.

\begin{theorem}\label{thm:Lp-log-Riesz}
Let $s>0$, $\lambda>1$, and $1\le p\le\infty$. Assume that $u,f\in L^p(\mathbb R^n)$ and that
\begin{equation}\label{eq:log-Riesz-eq}
(-\Delta)^{s+\ln}u=f
\qquad\text{in }\mathbb R^n
\end{equation}
in the sense of tempered distributions. Then $u\in\mathcal L^{p}_{s+\ln,\lambda}(\mathbb R^n)$
and there exists a constant $C=C(n,s,p,\lambda)>0$ such that
\[
\|u\|_{\mathcal L^{p}_{s+\ln,\lambda}(\mathbb R^n)}
\le
C\bigl(\|f\|_{L^p(\mathbb R^n)}+\|u\|_{L^p(\mathbb R^n)}\bigr).
\]
\end{theorem}

We next study the spaces $\mathcal L^{p}_{s+\ln,\lambda}$ themselves and compare them both with the classical Bessel scale and with the logarithmic Bessel potential spaces introduced by Opic and Trebels. A natural route would be to use the Mikhlin multiplier theorem, which is sufficient in the reflexive range $1<p<\infty$, but does not cover the endpoints $p=1$ and $p=\infty$. Our dyadic $L^1$-kernel criterion in Proposition~\ref{erjinzhi} allows us to work uniformly for the full range $1\le p\le\infty$.

\begin{proposition}\label{prop:lambda-equivalence-wiener}
Let $s>0$, $1\le p\le\infty$, and let $1<\lambda_1<\lambda_2$.
Then
\[
\mathcal{L}^{p}_{s+\ln,\lambda_1}(\mathbb R^n)
=
\mathcal{L}^{p}_{s+\ln,\lambda_2}(\mathbb R^n),
\]
and there exists $C=C(n,p,s,\lambda_1,\lambda_2)\ge 1$ such that for all
$u\in\mathcal S(\mathbb R^n)$,
\[
C^{-1}\|u\|_{\mathcal{L}^{p}_{s+\ln,\lambda_2}}
\le
\|u\|_{\mathcal{L}^{p}_{s+\ln,\lambda_1}}
\le
C\|u\|_{\mathcal{L}^{p}_{s+\ln,\lambda_2}}.
\]
\end{proposition}

\begin{proposition}\label{prop:lambda-inclusion}
Let $s>0$, $\lambda>1$, and $1\le p\le\infty$. Then
\[
\mathcal{L}^{p}_{s+\ln,\lambda}(\mathbb R^n)
\subsetneq
\mathcal{L}^{p}_{s+\ln,1}(\mathbb R^n).
\]
\end{proposition}

\begin{proposition}\label{prop:Hsln-vs-Hs}
Let $s>0$, $\lambda>1$, and $1\le p\le\infty$.
Then for every $\varepsilon>0$ the following continuous embeddings hold:
\[
\mathcal{L}^{p}_{2s+\varepsilon}(\mathbb R^n)
\hookrightarrow
\mathcal{L}^{p}_{s+\ln,\lambda}(\mathbb R^n)
\hookrightarrow
\mathcal{L}^{p}_{2s}(\mathbb R^n)
\hookrightarrow
L^p(\mathbb R^n).
\]
Moreover, there exists a constant $C=C(n,p,s,\lambda,\varepsilon)\ge1$ such that for all
$u\in\mathcal S(\mathbb R^n)$,
\[
\|u\|_{\mathcal{L}^{p}_{2s}}
\le
C\,\|u\|_{\mathcal{L}^{p}_{s+\ln,\lambda}}
\le
C\,\|u\|_{\mathcal{L}^{p}_{2s+\varepsilon}}.
\]
\end{proposition}

We also identify our scale with the logarithmic Bessel potential spaces from \cite{opic2000bessel}. For $\sigma>0$ and $\alpha\in\mathbb R$, Opic and Trebels define spaces $L^p_{\sigma,\alpha}$ by logarithmic modifications of the inhomogeneous Bessel multiplier. In the case relevant here, namely $\sigma=2s$ and $\alpha=1$, one has
\[
L^p_{2s,1}(\mathbb R^n)
:=
\Bigl\{u\in\mathcal S'(\mathbb R^n):
\mathcal F^{-1}\!\Bigl((1+4\pi^2|\xi|^2)^{s}\bigl(1+\ln(1+4\pi^2|\xi|^2)\bigr)\widehat u\Bigr)
\in L^p(\mathbb R^n)\Bigr\},
\]
equipped with the corresponding $L^p$-norm.

\begin{proposition}\label{prop:space-equivalence-literature}
Let $s>0$, $\lambda>1$, and $1\le p\le\infty$. Then
\[
\mathcal{L}^{p}_{s+\ln,\lambda}(\mathbb R^n) = L^p_{2s,1}(\mathbb R^n),
\]
and their norms are equivalent. More precisely, there exists a constant
$C=C(n,p,s,\lambda)\ge 1$ such that for all $u\in\mathcal S(\mathbb R^n)$,
\[
C^{-1}\|u\|_{L^{p}_{2s,1}}
\le
\|u\|_{\mathcal{L}^{p}_{s+\ln,\lambda}}
\le
C\|u\|_{L^{p}_{2s,1}}.
\]
\end{proposition}

\medskip

We now come to the embedding and compactness theory, which is one of the main applications of the preceding analysis. For the classical Bessel scale, the continuous embedding theory is standard; see, for instance, \cite[Theorem~6.2.4]{grafakos2009modern}. Sharp local compact embeddings on bounded domains and global compactness under radial symmetry are discussed in \cite[Theorem~1.1]{bellido2025compact} and \cite[Section~3.4, Theorem~4 and Corollary~2]{sickel2000radial}. For logarithmic Bessel spaces in the sense of \cite{opic2000bessel}, further compactness results can be found in \cite{edmunds2006non,edmunds2005compact}. For convenience, we summarize the relevant classical facts in the following statement:

Let $n\ge1$, $s\in(0,\frac{1}{2})$, $1<p<\infty$, and let $\Omega\subset\mathbb R^n$ be a bounded
Lipschitz domain. Let $\mathcal L^{p}_{2s,\mathrm{rad}}(\mathbb R^n)$ denote the closed subspace of
radially symmetric functions in $\mathcal L^{p}_{2s}(\mathbb R^n)$, and let
$C_0(\mathbb R^n)$ be the space of continuous functions on $\mathbb R^n$ vanishing
at infinity, equipped with the supremum norm. Set
\[
p^*=\frac{np}{n-2sp}\quad\text{if }2sp<n,
\qquad
\mu^*=2s-\frac{n}{p}\quad\text{if }2sp>n.
\]
\begin{itemize}
\item[(i)] \textbf{Subcritical case ($2sp<n$).}
$\mathcal L^{p}_{2s}(\mathbb R^n)\hookrightarrow L^q(\mathbb{R}^n)$ for $p\le q\le p^*$, and
$\mathcal L^{p}_{2s}(\mathbb R^n)\hookrightarrow\hookrightarrow L^q(\Omega)$ for $1\le q<p^*$.
If $n\ge2$, then $\mathcal L^{p}_{2s,\mathrm{rad}}(\mathbb R^n)\hookrightarrow\hookrightarrow L^q(\mathbb R^n)$
for $p<q<p^*$.

\item[(ii)] \textbf{Critical case ($2sp=n$).}
$\mathcal L^{p}_{2s}(\mathbb R^n)\hookrightarrow L^q(\mathbb{R}^n)$ for $p\le q<\infty$ and
$\mathcal L^{p}_{2s}(\mathbb R^n)\hookrightarrow\hookrightarrow L^q(\Omega)$ for every $1\le q<\infty$.
If $n\ge2$, then $\mathcal L^{p}_{2s,\mathrm{rad}}(\mathbb R^n)\hookrightarrow\hookrightarrow L^q(\mathbb R^n)$
for every $p<q<\infty$.

\item[(iii)] \textbf{Supercritical case ($2sp>n$).}
$\mathcal L^{p}_{2s}(\mathbb R^n)\hookrightarrow C_0(\mathbb{R}^n)$ and
$\mathcal L^{p}_{2s}(\mathbb R^n)\hookrightarrow\hookrightarrow C^{0,\mu}(\overline\Omega)$ for every
$0<\mu<\mu^*$. If $n\ge2$, then
$\mathcal L^{p}_{2s,\mathrm{rad}}(\mathbb R^n)\hookrightarrow\hookrightarrow C_0(\mathbb R^n)$, and
$\mathcal L^{p}_{2s,\mathrm{rad}}(\mathbb R^n)\hookrightarrow\hookrightarrow L^q(\mathbb R^n)$ for all
$p<q\le\infty$.
\end{itemize}

By contrast, in our logarithmic scale one gains new endpoint information. As a consequence of Proposition~\ref{prop:space-equivalence-literature}, the critical embeddings for $\mathcal L^{p}_{s+\ln,\lambda}$ follow from the logarithmic spaces of \cite{opic2000bessel}. At the borderline index $2s=\frac{n}{p}$, the classical Bessel space $\mathcal L^p_{2s}(\mathbb R^n)$ does not in general embed into $L^\infty(\mathbb R^n)$ and hence does not yield continuity. The additional logarithmic smoothness encoded in $\mathcal L^{p}_{s+\ln,\lambda}$ pushes the critical case beyond the $L^\infty$ threshold and yields continuity with an explicit logarithmic modulus.

\begin{theorem}\label{cor:critical_embeddings}
Let $n\ge 1,s>0,1< p<\infty$ and $\lambda>1$. If $s=\frac12\bigl(\frac{n}{p}+m\bigr)$ for an integer $m\ge 0$, then there is a
continuous embedding
\[
\mathcal{L}^{p}_{s+\ln,\lambda}(\mathbb R^n)\hookrightarrow C_0^m(\mathbb R^n).
\]
Moreover, for any $u\in\mathcal{L}^{p}_{s+\ln,\lambda}(\mathbb R^n)$ and any multi-index
$\gamma$ with $|\gamma|=m$, the weak derivative $\partial^\gamma u$ belongs to
$C_0(\mathbb R^n)$ and satisfies the logarithmic modulus of continuity
\[
\|\partial^\gamma u(\cdot+h)-\partial^\gamma u(\cdot)\|_{L^\infty}
\le
C\,\|u\|_{\mathcal L^{p}_{s+\ln,\lambda}}\,|\ln|h||^{-\frac1p},
\]
for $|h|$ sufficiently small, where $C=C(n,p,s,\lambda)>0$.
\end{theorem}

By Remark~\ref{rmk:blowup-sequence-p1}, the endpoint case $p=1$ cannot be included in Theorem~\ref{cor:critical_embeddings}. Indeed, when $2s=\frac{n}{p}=n$, the convolution operator
$f\mapsto K_{\frac n2+\ln}^{\lambda}*f$ is not bounded from $L^1(\mathbb R^n)$ to $L^\infty(\mathbb R^n)$, so no embedding into $C_0(\mathbb R^n)$, or even into $L^\infty(\mathbb R^n)$, can hold. In this case one can only expect the same type of $L^q$ conclusions as in the classical Bessel scale, namely embeddings of the form
\[
\mathcal L^{1}_{s+\ln,\lambda}(\mathbb R^n)\hookrightarrow L^q(\mathbb R^n),
\qquad 1\le q<\infty.
\]

We next address compactness. Beyond the critical continuous embeddings above, the strong endpoint integrability of the logarithmic Bessel kernel leads to genuinely compact embeddings, both locally on bounded domains and globally under radial symmetry. This is one of the main topological consequences of the logarithmic correction: it restores compactness at borderline exponents relevant to nonlinear and variational problems.

Let $\mathcal{L}^{p}_{s+\ln,\lambda,\mathrm{rad}}(\mathbb R^n)$ denote the closed subspace
of radially symmetric functions in $\mathcal{L}^{p}_{s+\ln,\lambda}(\mathbb R^n)$.

\begin{theorem}\label{thm:compact_critical}
Let $n\ge 1,s>0,1<p<\infty,$ and $\lambda>1.$
\begin{itemize}
\item[(i)] \textbf{Local compactness.}
Let $\Omega\subset\mathbb R^n$ be a bounded
Lipschitz domain and assume that $n=2sp.$
Then the restriction to $\Omega$ yields the compact embedding
\[
\mathcal{L}^{p}_{s+\ln,\lambda}(\mathbb R^n)\hookrightarrow\hookrightarrow C(\overline{\Omega}).
\]
Consequently, the following embeddings are compact for all $1\le q\le\infty$:
\[
\mathcal{L}^{p}_{s+\ln,\lambda}(\mathbb R^n)\hookrightarrow\hookrightarrow L^q(\Omega).
\]

\item[(ii)] \textbf{Global radial compactness.}
Assume $n\ge2$ and $n=2sp$. Then the embedding into the space of continuous
functions vanishing at infinity is compact:
\[
\mathcal{L}^{p}_{s+\ln,\lambda,\mathrm{rad}}(\mathbb R^n)\hookrightarrow\hookrightarrow C_0(\mathbb R^n).
\]
Consequently, the following embeddings are compact for all $p<q\le\infty$:
\[
\mathcal{L}^{p}_{s+\ln,\lambda,\mathrm{rad}}(\mathbb R^n)\hookrightarrow\hookrightarrow L^q(\mathbb R^n).
\]
\end{itemize}
\end{theorem}

The proof combines the critical continuity theorem above with the identification of $\mathcal L^{p}_{s+\ln,\lambda}$ and the logarithmic Bessel potential spaces of \cite{opic2000bessel}, which yields uniform equicontinuity. We emphasize that this theorem holds for all $s>0$ under the critical relation $s=\frac{n}{2p}$; in particular, no restriction to fractional orders $s\in(0,\frac12)$ is needed.

Finally, in the classical Sobolev and Bessel scales it is well known that the embedding into the critical space $L^{p^*}$ is not compact in the subcritical regime $n>2sp$. In contrast, the logarithmic modification weakens the kernel singularity exactly at the endpoint and places the critical convolution estimate in a strong Young-type framework. This gain is sufficient to recover compactness at the exponent $p^*$, both locally on bounded domains and globally under radial symmetry.

\begin{theorem}\label{thm:compact_subcritical}
Let $n\ge 1$, $\lambda>1$, $1<p<\infty$, and assume the subcritical regime
$n>2sp$. Set $p^*=\frac{np}{n-2sp}$.
\begin{itemize}
\item[(i)] \textbf{Local compactness.}
For any bounded
Lipschitz domain $\Omega\subset\mathbb R^n$, the restriction to $\Omega$ yields the
compact embedding
\[
\mathcal{L}^{p}_{s+\ln,\lambda}(\mathbb R^n)\hookrightarrow\hookrightarrow L^{p^*}(\Omega).
\]
Consequently, the following embeddings are compact for all $1\le q\le p^*$:
\[
\mathcal{L}^{p}_{s+\ln,\lambda}(\mathbb R^n)\hookrightarrow\hookrightarrow L^{q}(\Omega).
\]

\item[(ii)] \textbf{Global radial compactness.}
Assume $n\ge2$ and let $\mathcal{L}^{p}_{s+\ln,\lambda,\mathrm{rad}}(\mathbb R^n)$ denote the closed
subspace of radially symmetric functions. Then we have the global compact embedding
\[
\mathcal{L}^{p}_{s+\ln,\lambda,\mathrm{rad}}(\mathbb R^n)\hookrightarrow\hookrightarrow L^{p^*}(\mathbb R^n).
\]
Consequently, the following embeddings are compact for all $p<q\le p^*$:
\[
\mathcal{L}^{p}_{s+\ln,\lambda,\mathrm{rad}}(\mathbb R^n)\hookrightarrow\hookrightarrow L^{q}(\mathbb R^n).
\]
\end{itemize}
\end{theorem}

The paper is organized as follows. In Section~2, we collect the preliminary material used throughout the paper, including the moment-vanishing and flat-at-zero Schwartz subspaces, basic facts on convolutions of measures with distributions, and a brief review of classical Bessel potentials and Bessel spaces. Section~3 is devoted to logarithmic potential theory: we introduce the logarithmic Bessel kernel and the corresponding logarithmic Bessel spaces, derive explicit representations, and establish sharp asymptotics both at the origin and at infinity; we then define the logarithmic Riesz potential and prove the measure-level bridge linking the homogeneous and inhomogeneous symbols. In Section~4, we develop the $L^p$ regularity theory for global solutions, study the associated logarithmic Bessel spaces (including their dependence on the parameter $\lambda$ and their equivalence with the logarithmic spaces from the literature), and finally establish the embedding and compactness theory, with particular emphasis on the critical and subcritical compactness results and on global radial compactness.
	
		\section{Preliminaries}
	Throughout this paper, we work on $\R^n$ with $n\ge1$, set $\mathbb{N}_0 = \{0, 1, 2, \dots\}$ and $\mathbb{N} = \{ 1, 2, \dots\}$. For any dimension $n \in \mathbb{N}$, the set of multi-indices is denoted by
\[
\mathbb{N}_0^n = \mathbb{N}_0 \times \mathbb{N}_0 \times \dots \times \mathbb{N}_0.
\]
An element $\gamma \in \mathbb{N}_0^n$ is an $n$-tuple of non-negative integers $\gamma = (\gamma_1, \dots, \gamma_n)$. 

We write $\cS(\R^n)$ for the Schwartz space, $\cS'(\R^n)$ for the space of tempered distributions, see \cite[Section 2.2]{grafakos2008classical}, and we use the Fourier transform
\[
\widehat{f}(\xi)
:=\cF f(\xi)
:=\int_{\R^n}f(x)e^{-2\pi i x\cdot\xi}\,dx,
\quad \xi\in\R^n,
\]
with inverse $\cF^{-1}$ defined by
\[
\widehat{f}^{-1}(x)
:=\cF^{-1} \widehat{f}(x)
:=\int_{\R^n}f(\xi)e^{2\pi i x\cdot\xi}\,d\xi,
\quad x\in\R^n.
\]


\subsection{Moment Vanishing and Flat-at-Zero Schwartz Subspaces}\label{moment}

We introduce two closed subspaces of the Schwartz space that will be used to
quotient out polynomial ambiguities and to localize away from the origin in the
Fourier variable, see \cite[Section 2.3.4]{grafakos2008classical} and \cite[Section 5.1.2]{Triebel1983Theory}.

Define
\[
\cS_\infty(\R^n)
:=\Bigl\{\varphi\in\cS(\R^n): \int_{\R^n} x^\gamma \varphi(x)\,dx=0,\ \forall\,\gamma\in\N_0^n\Bigr\}.
\]
Equivalently, $\cS_\infty$ is the intersection of the kernels of the continuous
linear functionals $\varphi\mapsto \int_{\R^n}x^\gamma\varphi(x)\,dx$, hence it is a
closed subspace of $\cS(\R^n)$ endowed with the induced subspace topology.
Its continuous dual is denoted by $\cS_\infty'(\R^n)$.

Define
\[
\cS_0(\R^n)
:=\Bigl\{\psi\in\cS(\R^n): \partial^\gamma\psi(0)=0,\ \forall\,\gamma\in\N_0^n\Bigr\}.
\]
Again, $\cS_0$ is closed in $\cS$ as an intersection of kernels of continuous
linear maps $\psi\mapsto \partial^\gamma\psi(0)$ and we write $\cS_0'(\R^n)$ for its
continuous dual space.

For $\varphi\in\cS(\R^n)$ and $\gamma\in\N_0^n$, differentiation under the integral sign gives
\[\partial^\gamma \widehat{\varphi}(\xi)
=\int_{\R^n}(-2\pi i x)^\gamma \varphi(x)\,e^{-2\pi i x\cdot\xi}\,dx,
\qquad \xi\in\R^n,\]
and in particular
\[\partial^\gamma \widehat{\varphi}(0)=(-2\pi i)^{|\gamma|}\int_{\R^n}x^\gamma \varphi(x)\,dx.\]
Hence
\[\varphi\in\cS_\infty(\R^n)
\iff
\widehat{\varphi}\in\cS_0(\R^n),\]
so that the Fourier transform restricts to a  linear homeomorphism
\[
\cF:\cS_\infty(\R^n)\xrightarrow{}\cS_0(\R^n),
\qquad
\cF^{-1}:\cS_0(\R^n)\xrightarrow{}\cS_\infty(\R^n).
\]

Using the above identification, we define the Fourier transform on $\cS_\infty'$
and $\cS_0'$ by duality. For $u\in\cS_\infty'(\R^n)$ we set $\widehat{u}\in\cS_0'(\R^n)$ via
\[\langle \widehat{u},\psi\rangle := \langle u,\cF^{-1}\psi\rangle,
\qquad \forall\,\psi\in\cS_0(\R^n),\]
and for $v\in\cS_0'(\R^n)$ we define $\cF^{-1}v\in\cS_\infty'(\R^n)$ by
\[\langle \cF^{-1}v,\varphi\rangle := \langle v,\cF\varphi\rangle,
\qquad \forall\,\varphi\in\cS_\infty(\R^n).\]
Then $\cF:\cS_\infty'(\R^n)\to\cS_0'(\R^n)$ is also a linear homeomorphism with inverse $\cF^{-1}$.

Let $\cP$ be the vector space of all polynomials on $\R^n$, viewed as tempered
distributions, and let
\[
\mathcal{E}
:=\mathrm{span}\{\partial^\gamma\delta_0:\gamma\in\N_0^n\}
\subset \cS'(\R^n)
\]
be the space of tempered distributions supported at $\{0\}$.
The annihilators of $\cS_\infty$ and $\cS_0$ in $\cS'$ are precisely $\cP$ and $\mathcal{E}$,
respectively, and therefore there are canonical identifications, see \cite[Proposition 2.3.25]{grafakos2008classical}
\[
\cS_\infty'(\R^n)\cong \cS'(\R^n)/\cP,
\qquad
\cS_0'(\R^n)\cong \cS'(\R^n)/\mathcal{E},
\]
where the quotient map is given by restriction to $\cS_\infty$ or $\cS_0$.
Moreover, the Fourier transform exchanges $\cP$ and $\mathcal{E}$.
Indeed, for each multi-index $\gamma$,
\[\cF(x^\gamma)=\Bigl(\frac{i}{2\pi}\Bigr)^{|\gamma|}\partial^\gamma\delta_0,
\qquad
\cF^{-1}(\partial^\gamma\delta_0)=( -2\pi i)^{|\gamma|}x^\gamma,\]
so $\cF(\cP)=\mathcal{E}$ and $\cF^{-1}(\mathcal{E})=\cP$. Consequently, $\cF$ induces an isomorphism of
quotient spaces
\[
\cS'(\R^n)/\cP \ \cong\ \cS'(\R^n)/\mathcal{E},
\]
and hence one may freely identify $\cS_\infty'(\R^n)$ and $\cS_0'(\R^n)$ with the
same quotient of $\cS'(\R^n)$, up to this canonical Fourier correspondence.


    \subsection{Convolution of Finite Borel Measures with Distributions}

Any finite Borel measure $\mu$ on $\R^n$ defines a tempered distribution (still denoted by $\mu$) via
\[
\langle \mu,f\rangle := \int_{\R^n} f(x)\,d\mu(x),\qquad f\in\cS(\R^n).
\]
Indeed, if $f_k\to f$ in $\cS(\R^n)$, then in particular $\|f_k-f\|_{L^\infty}\to0$, and hence
\[
|\langle \mu,f_k-f\rangle|
\le |\mu|(\R^n)\|f_k-f\|_{L^\infty}\longrightarrow 0,
\]
where $|\mu|(\R^n)$ is the total variation of $\mu$.
Thus $\mu\in\cS'(\R^n)$. Note that a finite measure need not have compact support. We denote by $\mathcal O_M(\R^n)$ the space of \emph{Schwartz multipliers}, i.e.
\begin{equation}\label{omom}
    \mathcal O_M(\R^n)
:=\Bigl\{m\in C^\infty(\R^n):\ \forall \alpha\in\N_0^n, \exists\,C_\alpha>0,\ N_\alpha\in\mathbb N_0
\text{ such that }|\partial^\alpha m(\xi)|\le C_\alpha(1+|\xi|)^{N_\alpha}\Bigr\}.
\end{equation}
A basic fact (see, e.g., \cite[Section~2.3.2]{grafakos2008classical}) is that
$m\in\mathcal O_M$ if and only if the multiplication map
\[
M_m:\cS(\R^n)\to\cS(\R^n),\qquad \varphi\mapsto m\varphi,
\]
is continuous. Consequently, for $m\in\mathcal O_M$ and $T\in\cS'(\R^n)$ the product
$mT\in\cS'(\R^n)$ is well-defined by
\[
\langle mT,\varphi\rangle:=\langle T,m\varphi\rangle,\qquad \forall\,\varphi\in\cS(\R^n),
\]
and the map $T\mapsto mT$ is continuous on $\cS'(\R^n)$.

\medskip
Following \cite[Section~2.3.3]{grafakos2008classical}, for $t\in\R^n$ we denote by $\tau_t$ the translation
operator on $\cS(\R^n)$,
\[
(\tau_t f)(x):=f(x-t),\qquad x\in\R^n,
\]
and for $u\in\cS'(\R^n)$ we define its translation by duality,
\[
\langle \tau_t u,f\rangle := \langle u,\tau_{-t}f\rangle,\qquad f\in\cS(\R^n).
\]

Let $\mu$ be a finite Borel measure on $\R^n$ and assume that
\begin{equation}\label{eq:mu-OM}
\widehat{\mu}\in\mathcal O_M(\R^n),
\qquad
\widehat{\mu}(\xi):=\int_{\R^n}e^{-2\pi i x\cdot\xi}\,d\mu(x).
\end{equation}
For instance, \eqref{eq:mu-OM} holds if $\mu$ is compactly supported, or more generally if for every $k\in\N$
$$\int_{\R^n}|x|^k\,d|\mu|(x)<\infty.$$

Under \eqref{eq:mu-OM}, multiplication by $\widehat\mu$ is continuous on $\cS(\R^n)$, hence it restricts to a
continuous map $\cS_0(\R^n)\to\cS_0(\R^n)$. Indeed, if $\psi\in\cS_0$, then $\widehat\mu\,\psi\in\cS$ and for every
multi-index $\alpha$,
\[
\partial^\alpha(\widehat\mu\,\psi)(0)
=\sum_{\beta\le\alpha}\binom{\alpha}{\beta}\,(\partial^\beta\widehat\mu)(0)\,(\partial^{\alpha-\beta}\psi)(0)=0,
\]
since $\partial^{\gamma}\psi(0)=0$ for all $\gamma$. Thus $\widehat\mu\,\psi\in\cS_0$. Thus, the map $T\mapsto \widehat\mu T$ is continuous on $\cS_0'(\R^n)$.

Let $u\in\cS_\infty'(\R^n)$. Using the Fourier isomorphism $\cF:\cS_\infty'(\R^n)\to\cS_0'(\R^n)$, we define
$\mu*u\in\cS_\infty'(\R^n)$ by prescribing its Fourier transform:
\begin{equation}\label{eq:Fourier-measure-conv}
\cF(\mu*u):=\widehat\mu\,\cF u \in\cS_0'(\R^n),
\end{equation}
where the product on the right-hand side is understood by duality on $\cS_0$:
\[
\langle \widehat\mu\,\cF u,\psi\rangle := \langle \cF u,\widehat\mu\,\psi\rangle,
\qquad \forall\,\psi\in\cS_0(\R^n).
\]
This defines a continuous linear operator $u\mapsto \mu*u$ on $\cS_\infty'(\R^n)$.

Let $1\le p\le\infty$ and assume $u\in L^p(\R^n)$. Then $u$ defines an element of $\cS_\infty'(\R^n)$ by
restriction, hence $\mu*u$ is well-defined by \eqref{eq:Fourier-measure-conv}. Moreover, if $\mu$ is a finite
Borel measure, then the classical convolution
\[
(\mu*u)(x):=\int_{\R^n}u(x-t)\,d\mu(t)
\]
is defined for a.e.\ $x\in\R^n$ and satisfies the Young estimate
\[
\|\mu*u\|_{L^p(\R^n)}\le |\mu|(\R^n)\,\|u\|_{L^p(\R^n)}.
\]
In this case, the distribution $\mu*u\in\cS_\infty'(\R^n)$ defined by \eqref{eq:Fourier-measure-conv}
coincides with the tempered distribution induced by the function $x\mapsto\int u(x-t)\,d\mu(t)$; namely,
for every $\varphi\in\cS_\infty(\R^n)$,
\[
\langle \mu*u,\varphi\rangle=\int_{\R^n}\int_{\R^n}u(x-t)\varphi(x)dx\,d\mu(t).
\]


\subsection{Bessel Potential and Bessel Spaces}\label{bessel poten}

It is well known (see, e.g., Triebel \cite{triebel2006theory}) that the Bessel potential space $\mathcal{L}_s^p(\R^n)$ is a Banach space. For $1\le p\le \infty$ and $s\ge 0$, we have the continuous embedding $\mathcal{L}_s^p(\R^n)\hookrightarrow L^{p}(\R^n).$
We mention that this inclusion fails for $s<0$; in fact, the embedding relation is reversed, yielding $L^p(\R^n)\hookrightarrow \mathcal{L}_s^p(\R^n)$ for $s<0$. Furthermore, for $1\le p<\infty$ and $s\in \mathbb{R}$, the Schwartz space $\cS(\R^n)$ is dense in $\mathcal{L}_s^p(\R^n)$.

\begin{definition}
For $s>0$, the Bessel kernel $G_s$ is defined as the tempered distribution whose 
Fourier transform is
\[
\widehat G_s(\xi)=(1+4\pi^2|\xi|^2)^{-s},\quad \xi\in\R^n.
\]
\end{definition}

By \cite[Chapter V]{stein1970singular},
 the Bessel kernel $G_s$ has the integral
representation
\[ G_s(x)
:=\cF^{-1}(\widehat G_s)(x)
=
\frac{1}{(4\pi)^{\frac n2}\Gamma(s)}
\int_0^\infty \frac{e^{-\frac{|x|^2}{4t}-t}}{t^{\frac{n}{2}+1-s}}dt,\quad x\in \mathbb{R}^n.\]
Equivalently, in terms of the Gaussian heat kernel $p_t(x)$ defined in \eqref{gaussian}, we can write
\[G_s(x)
=
\frac{1}{\Gamma(s)}
\int_0^\infty e^{-t}\,t^{s-1}\,p_t(x)\,dt,
\quad x\in\R^n.\]

Moreover, $G_s$ is integrable on $\R^n$.  
Indeed, using the identity
\[
\int_{\R^n}p_t(x)\,dx=1\qquad\text{for every }t>0,
\]
and applying Fubini's theorem, we obtain
\begin{equation}\label{gssg}
    \|G_s\|_{L^1(\R^n)}
=
\int_{\R^n}G_s(x)\,dx
=
\frac{1}{\Gamma(s)}
\int_0^\infty e^{-t}t^{s-1}
\Bigl(\int_{\R^n}p_t(x)\,dx\Bigr)\,dt
=1.
\end{equation}

By (\ref{gssg}), we know that convolution with $G_s$ defines a bounded operator on $L^p(\R^n)$
for every $1\le p\le\infty$, the Bessel potential operator
\[
J_s f := G_s * f
\]
is an isomorphism of $L^p(\R^n)$ onto the Bessel potential space $\mathcal{L}_{2s}^p(\R^n)$. Let $s\in \mathbb{R}$. For $u\in\cS(\R^n),$ we define
\[
(I-\Delta)^{s}u
:=\cF^{-1}\bigl((1+4\pi^2|\xi|^2)^{s}\widehat u(\xi)\bigr).
\]
Since $(1+4\pi^2|\xi|^2)^{s}$ is a smooth function with at most polynomial growth,
the multiplier $(1+4\pi^2|\xi|^2)^{s}$ maps $\cS(\R^n)$ continuously into itself;
hence $(I-\Delta)^{s}:\cS(\R^n)\to\cS(\R^n)$ is a continuous linear operator.

The operator $(I-\Delta)^{s}$ extends uniquely by duality to a
continuous linear map on $\cS'(\R^n)$, still denoted by $(I-\Delta)^{s}$, via
\[
\langle (I-\Delta)^{s}u,\varphi\rangle
:=\langle u,(I-\Delta)^{s}\varphi\rangle,
\qquad u\in\cS'(\R^n),\ \varphi\in\cS(\R^n).
\]

By construction we have for $s>0$, in the sense of tempered distributions,
\[
\cF\bigl((I-\Delta)^{\pm s}u\bigr)(\xi)
=(1+4\pi^2|\xi|^2)^{\pm s}\widehat u(\xi),
\qquad u\in\cS'(\R^n),
\]
and thus,
\[(I-\Delta)^{s}\circ (I-\Delta)^{-s}=(I-\Delta)^{-s}\circ (I-\Delta)^{s}=\mathrm{Id}\quad \text{in}\quad \cS'(\R^n).\]

Hence, the operator $(I-\Delta)^{-s},s>0$ is an isometric isomorphism from $L^p(\R^n)$ onto
$\mathcal{L}_{2s}^p(\R^n)$: 
\[
\|(I-\Delta)^{-s} f\|_{\mathcal{L}_{2s}^p}
=\|f\|_{L^p},
\]
and for every $u\in \mathcal{L}_{2s}^p(\R^n)$ we have $u=(I-\Delta)^{-s} f$ with $f:=(I-\Delta)^s u\in L^p(\R^n)$.

We will repeatedly convolve $f\in L^p(\mathbb R^n)$ with an integrable kernel
$g\in L^1(\mathbb R^n)$ whose Fourier transform satisfies $\widehat g\in\mathcal O_M(\mathbb R^n)$.
While Young's inequality ensures $g*f\in L^p\subset\mathcal S'$, we also need the
distributional Fourier identity $\widehat{g*f}=\widehat g\,\widehat f$.
The following lemma provides this fact via a standard mollification argument.

\begin{lemma}\label{lem:Fourier-conv-L1-OM}
Let $g\in L^1(\R^n)$ and $f\in L^p(\R^n)$ for some $1\le p\le\infty$. Assume moreover that
\[
\widetilde g:=\widehat g \in \mathcal O_M(\R^n).
\]
where $\mathcal O_M(\R^n)$ is defined in \eqref{omom}. Then $g*f$ is an $L^p$-function and
\[
\widehat{g*f}=\widehat g\,\widehat f \qquad\text{in }\cS'(\R^n).
\]
\end{lemma}

\begin{proof}
Young's inequality yields $g*f\in L^p(\R^n)\subset \cS'(\R^n)$.

If $g\in\cS$, by \cite[Proposition 2.3.22]{grafakos2008classical}, the standard convolution identity gives
\[
\widehat{g*f}=\widehat g\,\widehat f \qquad\text{in }\cS'(\R^n).
\]

Let $\rho\in\cS(\R^n)$ satisfy $\int_{\R^n}\rho \,dx=1$, and define $\rho_t(x)=t^{-n}\rho(x/t)$ and
$g_t:=g*\rho_t$ for $t>0$. Then $g_t\in\cS(\R^n)$ and $g_t\to g$ in $L^1(\R^n)$ as $t\to0$.
Hence, by Young's inequality,
\[
\|(g_t-g)*f\|_{L^p}\le \|g_t-g\|_{L^1}\,\|f\|_{L^p}\xrightarrow[]{t\to0}0,
\]
so $g_t*f\to g*f$ in $L^p$, thus also in $\cS'$. Since $\cF:\cS'\to\cS'$ is continuous,
\begin{equation}\label{eq:Ft-conv-limit}
\widehat{g_t*f}\longrightarrow \widehat{g*f}\qquad\text{in }\cS'.
\end{equation}

On the other hand, 
\begin{equation}\label{eq:Ft-conv-prod}
\widehat{g_t*f}=\widehat{g_t}\,\widehat f\qquad\text{in }\cS'.
\end{equation}
Moreover,
\[
\widehat{g_t}=\widehat g\,\widehat{\rho_t}=\widetilde g\cdot \widehat\rho(t\,\cdot),
\]
where $\widehat\rho\in\cS(\R^n)$ and $\widehat\rho(0)=\int\rho=1$.

Now fix $\varphi\in\cS(\R^n)$. Let $\psi = \widetilde g \varphi$. Since $\widetilde g\in\mathcal O_M$ and $\varphi \in \cS$, it follows that $\psi \in \cS(\R^n)$. We claim that
\[
\widehat{g_t}\,\varphi - \widetilde g\,\varphi = \psi(\cdot)\big(\widehat\rho(t\,\cdot)-1\big) \longrightarrow 0 \qquad\text{in }\cS(\R^n).
\]
To show this, we need to prove that for any $N \ge 0$ and any multi-index $\alpha$,
\[
\sup_{\xi\in\R^n} (1+|\xi|)^N \big|\partial^\alpha\big[\psi(\xi)\big(\widehat\rho(t\xi)-1\big)\big]\big| \xrightarrow[]{t\to0} 0.
\]
By the Leibniz rule, the derivatives are a linear combination of terms of the form
\[
\partial^{\alpha-\beta}\psi(\xi) \cdot t^{|\beta|}(\partial^\beta \widehat\rho)(t\xi), \qquad \beta \le \alpha.
\]
For $\beta = 0$, the term is $(\partial^\alpha \psi(\xi)) (\widehat\rho(t\xi)-1)$. Since $\partial^\alpha \psi \in \cS$, it decays faster than any polynomial. By the Dominated Convergence Theorem, the weighted supremum goes to $0$ as $t\to 0$. 
For $|\beta| \ge 1$, since $\partial^{\alpha-\beta}\psi \in \cS$, we have $\sup_{\xi} (1+|\xi|)^N |\partial^{\alpha-\beta}\psi(\xi)| \le C_{N,\alpha,\beta} < \infty$. Furthermore, $\partial^\beta \widehat\rho$ is uniformly bounded. Thus, this term is bounded by $C t^{|\beta|}$, which clearly goes to $0$ as $t \to 0$. 
This proves that
\[
\widehat{g_t}\,\varphi \longrightarrow \widetilde g\,\varphi \qquad\text{in }\cS(\R^n).
\]

Therefore, using $\widehat f\in\cS'(\R^n)$ and the continuity of $\widehat f$ on $\cS$,
\[
\langle \widehat{g_t}\,\widehat f-\widetilde g\,\widehat f,\varphi\rangle
=\langle \widehat f,(\widehat{g_t}-\widetilde g)\varphi\rangle
\longrightarrow 0,
\]
which proves
\begin{equation}\label{eq:prod-limit}
\widehat{g_t}\,\widehat f \longrightarrow \widetilde g\,\widehat f\qquad\text{in }\cS'.
\end{equation}
Finally, combining \eqref{eq:Ft-conv-limit}, \eqref{eq:Ft-conv-prod}, and \eqref{eq:prod-limit}
yields $\widehat{g*f}=\widetilde g\,\widehat f$ in $\cS'$, as claimed.
\end{proof}

We conclude this subsection by recalling the following classical identification:

\begin{lemma}\label{lem:I-Delta-Js}
Let $s>0$ and $1\le p\le \infty.$
Then, for every $f\in L^p(\R^n)$,
\[
J_s f = (I-\Delta)^{-s}f\quad \text{in}\quad L^p(\R^n),
\]
and
\[(I-\Delta)^{s}(G_s*f)=f\quad \text{in}\quad \cS'(\R^n).\]
\end{lemma}

\begin{proof}
Let $f\in L^p(\R^n)$ and recall that $L^p(\R^n)\subset\cS'(\R^n)$ for every $1\le p\le \infty,$ see \cite[Examples 2.3.5]{grafakos2008classical}.
Fix $\varphi\in\cS(\R^n)$ and set $\psi:=\cF^{-1}\varphi\in\cS(\R^n)$, so that
$\varphi=\cF[\psi]$. Using the definition of the Fourier transform on $\cS'$, we compute
\[
\langle J_s f,\varphi\rangle
=\langle G_s*f,\cF[\psi]\rangle
=\langle \widehat{G_s*f},\psi\rangle.
\]
Since $\widehat G_s\in \mathcal O_M(\R^n)$, by Lemma \ref{lem:Fourier-conv-L1-OM}, we have
\[
\widehat{G_s*f}=\widehat G_s\,\widehat f \qquad\text{in }\cS'(\R^n).
\]
Therefore,
\[
\langle J_s f,\varphi\rangle
=\langle \widehat G_s\,\widehat f,\psi\rangle.
\]
On the other hand, by definition of $(I-\Delta)^{-s}$ as a Fourier multiplier on $\cS'$,
\[
\cF({(I-\Delta)^{-s}f})=(1+4\pi^2|\xi|^2)^{-s}\,\widehat f
=\widehat G_s\,\widehat f,
\]
hence,
\[
\langle J_s f,\varphi\rangle=\langle (I-\Delta)^{-s}f,\varphi\rangle
\qquad \forall\,\varphi\in\cS(\R^n),
\]
so $J_s f=(I-\Delta)^{-s}f$ in $\cS'(\R^n)$. Since both $J_s f$ and $(I-\Delta)^{-s}f$ belong to $L^p(\R^n)$ and hence
\[
J_s f=(I-\Delta)^{-s}f \quad \text{in } L^p(\R^n).
\]

Then the second claim follows from the fact
\[(I-\Delta)^{s}\circ (I-\Delta)^{-s}=(I-\Delta)^{-s}\circ (I-\Delta)^{s}=\mathrm{Id}\quad \text{in}\quad \cS'(\R^n),\]
hence, we complete the proof.
\end{proof}

Consequently, one has
\begin{align*}
    \mathcal{L}_s^p(\R^n)
=&\Bigl\{u\in\cS'(\R^n):\cF^{-1}\bigl((1+4\pi^2|\xi|^2)^{s/2}\widehat u(\xi)\bigr)\in L^p(\R^n)\Bigr\}\\=&\Bigl\{u\in\cS'(\R^n):(I-\Delta)^{\frac{s}{2}}u\in L^p(\R^n)\Bigr\}\\=&\Bigl\{u\in\cS'(\R^n):u=G_{\frac{s}{2}}*f,f\in L^p(\R^n)\Bigr\}\\=&\Bigl\{u\in L^p(\R^n):u=G_{\frac{s}{2}}*f,f\in L^p(\R^n)\Bigr\}
=J_{\frac{s}{2}}(L^p(\R^n))
\end{align*}
with equivalent norms.


\section{Logarithmic Bessel and Riesz Potentials}

In this section we introduce the two potential kernels that underpin our whole-space
regularity theory: the \emph{logarithmic Bessel potential} (inhomogeneous) and the
\emph{logarithmic Riesz potential} (homogeneous). We first define the logarithmic
Bessel kernel $K_{s+\ln}^{\lambda}$ through its Fourier multiplier and derive
useful pointwise representations, including an $\alpha$--integral formula in terms
of shifted Bessel kernels and a Fourier--Bessel (Hankel-type) representation. These
formulas enable us to establish sharp asymptotics of $K_{s+\ln}^{\lambda}$ both near
the origin and at infinity, together with explicit leading constants. 

We then
introduce the logarithmic Riesz potential and the homogeneous functional framework
based on the subspaces $\mathcal S_\infty$ and $\mathcal S_0$, which removes
polynomial ambiguities and isolates the low-frequency singularity. Finally, we
connect the homogeneous and inhomogeneous theories via a dyadic Littlewood--Paley
construction: we show that the transition multiplier generates an $L^1$-kernel, and
we obtain a finite-measure decomposition that expresses the inhomogeneous symbol as
a regular term plus a weighted homogeneous fractional--logarithmic component. This
bridge will be the key tool for transferring mapping properties and establishing
$L^p$ regularity for equations driven by $(-\Delta)^{s+\ln}$.

Throughout this section we fix a parameter $\lambda>1$, so that
\[
\ln(\lambda+4\pi^2|\xi|^2)\ge \ln\lambda>0,
\qquad \forall\,\xi\in\mathbb R^n.
\]

\subsection{Logarithmic Bessel Kernel and Spaces}

Let $s>0$ and $\lambda>1$. For $\varphi\in\cS(\R^n)$ we define 
\[
(\lambda I-\Delta)^{s+\ln}\varphi
:=
\cF^{-1}\!\Bigl((\lambda+4\pi^2|\xi|^2)^s\ln(\lambda+4\pi^2|\xi|^2)\,\widehat\varphi(\xi)\Bigr).
\]
Since $(\lambda+4\pi^2|\xi|^2)^s\ln(\lambda+4\pi^2|\xi|^2)$ is a smooth function with at
most polynomial growth, this defines a continuous linear map
\[
(\lambda I-\Delta)^{s+\ln}:\cS(\R^n)\to\cS(\R^n).
\]
By duality we extend it to a continuous linear operator on tempered
distributions by
\[
\langle(\lambda I-\Delta)^{s+\ln}u,\varphi\rangle
:=\langle u,(\lambda I-\Delta)^{s+\ln}\varphi\rangle,
\quad u\in\cS'(\R^n),\ \varphi\in\cS(\R^n).
\]
In particular,
\[
\cF\bigl((\lambda I-\Delta)^{s+\ln}u\bigr)(\xi)
=(\lambda+4\pi^2|\xi|^2)^s\ln(\lambda+4\pi^2|\xi|^2)\,\widehat u(\xi)
\quad\text{in }\cS'(\R^n).
\]

\begin{lemma}
\label{lem:log-bessel-fcalc}
Let $s>0$ and $\lambda>1$. For every $\varphi\in\cS(\R^n)$ one has
\[
(\lambda I-\Delta)^{s+\ln}\varphi
=
(\lambda I-\Delta)^s\circ\ln(\lambda I-\Delta)\,\varphi
=
\ln(\lambda I-\Delta)\circ(\lambda I-\Delta)^s\varphi,
\]
where $\ln(\lambda I-\Delta)$ is understood as the Fourier multiplier with
symbol $\ln(\lambda+4\pi^2|\xi|^2)$.
\end{lemma}

\begin{proof}
Fix $\varphi\in\cS(\R^n)$. By definition of Fourier multipliers,
\[
\cF\!\left((\lambda I-\Delta)^s\circ\ln(\lambda I-\Delta)\,\varphi\right)(\xi)
=
(\lambda+4\pi^2|\xi|^2)^s\,\ln(\lambda+4\pi^2|\xi|^2)\,\widehat\varphi(\xi).
\]
Taking inverse Fourier transform gives
\[
(\lambda I-\Delta)^s\circ\ln(\lambda I-\Delta)\,\varphi
=
\cF^{-1}\!\Bigl((\lambda+4\pi^2|\xi|^2)^s\ln(\lambda+4\pi^2|\xi|^2)\,\widehat\varphi(\xi)\Bigr)
=
(\lambda I-\Delta)^{s+\ln}\varphi,
\]
which proves the first identity. The second identity and the commutation
property follow similarly, since on $\cS(\R^n)$ all these operators are
Fourier multipliers and hence their symbols multiply pointwise.
\end{proof}

Motivated by the Bessel potential spaces $\mathcal{L}_s^p(\R^n),$ we now define spaces associated with $(\lambda I-\Delta)^{s+\ln}$.

\begin{definition}
Let $s>0$, $\lambda>1$ and $1\le p\le\infty$.  
The logarithmic Bessel space $\mathcal{L}^{p}_{s+\ln,\lambda}(\R^n)$ is
\[
\mathcal{L}^{p}_{s+\ln,\lambda}(\R^n)
:=
\Bigl\{u\in\cS'(\R^n)\;:\;
(\lambda I-\Delta)^{s+\ln}u\in L^p(\R^n)\Bigr\},
\]
equipped with the norm
\[
\|u\|_{\mathcal{L}^{p}_{s+\ln,\lambda}(\R^n)}
:=
\bigl\|(\lambda I-\Delta)^{s+\ln}u\bigr\|_{L^p(\R^n)}.
\]
\end{definition}

\begin{remark}
(i) One could equivalently define the norm by taking the $L^p$–norm of the
Fourier multiplier,
\[
\|u\|_{\mathcal{L}^{p}_{s+\ln,\lambda}(\R^n)}
=
\Bigl\|
\cF^{-1}\Bigl((\lambda+4\pi^2|\xi|^2)^s\ln(\lambda+4\pi^2|\xi|^2)\,\widehat u\Bigr)
\Bigr\|_{L^p(\R^n)}.
\]

(ii) $(\lambda I-\Delta)^{s+\ln}:\mathcal{L}^{p}_{s+\ln,\lambda}(\R^n)\rightarrow L^p(\R^n)$ is isomorphism.
\end{remark}

In analogy with the classical Bessel potential operator $(I-\Delta)^{-s}$,
we now introduce the inverse of the logarithmic Bessel operator as a
pseudo–differential operator with symbol
$$(\lambda+4\pi^2|\xi|^2)^{-s}\bigl[\ln(\lambda+4\pi^2|\xi|^2)\bigr]^{-1}.$$

\begin{definition}\label{kernel1}
Let $s>0$ and $\lambda>1$. The logarithmic Bessel kernel is defined as the tempered distribution whose Fourier transform is
\[
\widehat{K_{s+\ln}^{\lambda}}(\xi)
:=
\frac{1}{(\lambda+4\pi^2|\xi|^2)^s\ln(\lambda+4\pi^2|\xi|^2)},
\quad \xi\in\R^n.
\]
We define the logarithmic Bessel potential operator
$(\lambda I-\Delta)^{-(s+\ln)}$ on $\cS(\R^n)$ by
\[
(\lambda I-\Delta)^{-(s+\ln)}\varphi
:=\cF^{-1}\!\bigl(\widehat{K_{s+\ln}^{\lambda}}(\xi)\,\widehat\varphi(\xi)\bigr),
\qquad \varphi\in\cS(\R^n),
\]
and extend it to $\cS'(\R^n)$ by duality. 
\end{definition}

\begin{remark}
    $(\lambda I-\Delta)^{s+\ln} \circ (\lambda I-\Delta)^{-(s+\ln)} =(\lambda I-\Delta)^{-(s+\ln)} \circ (\lambda I-\Delta)^{s+\ln} =\mathrm{Id}\:\:\text{in}$ $\:\:\cS'(\R^n).$
\end{remark}

Next, we emphasize a crucial structural observation regarding the
integral representation \eqref{eq:G-lambda-heat}. The weight
$\Gamma(\alpha)^{-1}e^{-\lambda t}t^{\alpha-1}$ in \eqref{eq:G-lambda-heat} is,
up to the normalizing factor $\lambda^{-\alpha}$, the density of a Gamma random
variable. More precisely, let $T_\alpha\sim\mathrm{Gamma}(\alpha,\lambda)$ (shape
$\alpha$, rate $\lambda$), i.e.\ with density
\[
f_\alpha(t)=\frac{\lambda^\alpha}{\Gamma(\alpha)}\,t^{\alpha-1}e^{-\lambda t},
\qquad t>0.
\]
Then
\[
\mathbb E[T_\alpha]=\frac{\alpha}{\lambda}=:\mu_\alpha,
\qquad
\mathrm{Var}(T_\alpha)=\frac{\alpha}{\lambda^2}=:\sigma_\alpha^2,
\]
and \eqref{eq:G-lambda-heat} can be rewritten as the mixture formula
\begin{equation}\label{eq:G-lambda-mixture}
G^\lambda_\alpha(x)
=\frac{1}{\Gamma(\alpha)}\int_0^\infty e^{-\lambda t}t^{\alpha-1}p_t(x)\,dt
=\frac{1}{\lambda^\alpha}\int_0^\infty p_t(x)\,f_\alpha(t)\,dt
=\frac{1}{\lambda^\alpha}\,\mathbb E\!\big[p_{T_\alpha}(x)\big],\,x\in \mathbb{R}^n\setminus \left\{0\right\}.
\end{equation}
This probabilistic viewpoint will be exploited below: the parameter $\alpha$
controls the concentration of the mixing measure around its mean
$\mu_\alpha=\alpha/\lambda$, while $\lambda$ fixes the exponential tail, and
both features translate into quantitative information on the spatial behaviour
of $G^\lambda_\alpha$.

\medskip
In particular, since $T_\alpha$ concentrates at $\mu_\alpha$ as $\alpha\to\infty$,
the mixture \eqref{eq:G-lambda-mixture} becomes asymptotically deterministic.
The next proposition makes this precise and can be viewed as a Laplace-type
asymptotic for the Gamma mixture: one Taylor-expands the smooth map
$t\mapsto p_t(x)$ around $t=\mu_\alpha$ and then averages with respect to the
Gamma law, using the explicit moments
$\mathbb E[T_\alpha-\mu_\alpha]=0$ and $\mathbb E[(T_\alpha-\mu_\alpha)^2]=\sigma_\alpha^2$.

\begin{proof}[\textbf{Proof of Proposition \ref{prop:Galp-laplace-asymp}.}]
Let $T_\alpha\sim \mathrm{Gamma}(\alpha,\lambda)$, i.e.\ with density $f_\alpha(t)$
and then the heat-kernel representation yields
\[
G^\lambda_\alpha(x)=\frac{1}{\lambda^\alpha}\,\mathbb{E}\big[p_{T_\alpha}(x)\big].
\]

Fix $x\in \mathbb{R}^n\setminus \left\{0\right\}$. We expand $t\mapsto p_t(x)$ at $t=\mu_\alpha$:
\begin{equation}\label{error-term}
    p_{T_\alpha}(x)=p_{\mu_\alpha}(x)+p'_{\mu_\alpha}(x)(T_\alpha-\mu_\alpha)
+\frac12 p''_{\mu_\alpha}(x)(T_\alpha-\mu_\alpha)^2+R_\alpha,
\end{equation}
and by Taylor's theorem, there exists some random variable $\theta_\alpha$ between $T_\alpha$ and $\mu_\alpha$,
\[
R_\alpha=\frac{1}{6}\,p'''_{\theta_\alpha}(x)\,(T_\alpha-\mu_\alpha)^3.
\]
Taking expectations and using $\mathbb{E}[T_\alpha-\mu_\alpha]=0$ gives
\begin{equation}\label{eq:EpT-expansion}
\mathbb{E}\big[p_{T_\alpha}(x)\big]
=p_{\mu_\alpha}(x)+\frac12 p''_{\mu_\alpha}(x)\,\sigma_\alpha^2+\mathbb{E}[R_\alpha].
\end{equation}

We now estimate the derivatives of $p_t(x)=(4\pi t)^{-n/2}e^{-|x|^2/(4t)}$ for large $t$.
A direct computation shows that for $k=1,2,3$,
\[
\partial_t^k p_t(x)=p_t(x)\,t^{-k}\,Q_k\!\left(\frac{|x|^2}{t}\right),
\]
where $Q_k$ is a polynomial in one variable of degree at most $k$.

Fix $x\in\R^n$ and set $t_*:=\frac{|x|^2}{2n}$. Since $p_t(x)$ attains its maximum at $t=t_*$
and is strictly decreasing for $t\ge t_*$, there exists
\[
\alpha_x:=\max\Big\{1,\frac{\lambda|x|^2}{n}\Big\}
\quad\text{such that}\quad
\frac{\mu_\alpha}{2}=\frac{\alpha}{2\lambda}\ge t_*
\quad\forall\,\alpha\ge \alpha_x.
\]
Consequently, for all $\alpha\ge \alpha_x$ we have
\[
\sup_{t\ge \mu_\alpha/2} p_t(x)=p_{\mu_\alpha/2}(x),
\]
it follows that for $k=1,2,3$,
\[
\sup_{t\ge \mu_\alpha/2}|\partial_t^k p_t(x)|
\le  C_{k,x}\,\mu_\alpha^{-k}\,p_{\mu_\alpha/2}(x),
\quad \alpha\ge \alpha_x,
\]
where $C_{k,x}>0$ depends only on $n,k$ and $|x|$.
Moreover, since
\[
\frac{p_{\mu_\alpha/2}(x)}{p_{\mu_\alpha}(x)}
=2^{n/2}\exp\!\Bigl(-\frac{|x|^2}{4\mu_\alpha}\Bigr)\le 2^{n/2},
\]
we may further bound
\begin{equation}\label{eq:pk-der-bound}
    \sup_{t\ge \mu_\alpha/2}|\partial_t^k p_t(x)|
\le \widetilde C_{k,x}\,\mu_\alpha^{-k}\,p_{\mu_\alpha}(x),
\qquad k=1,2,3,\ \alpha\ge \alpha_x,
\end{equation}
for some $\widetilde C_{k,x}>0$ and $\widetilde C_{k,x}$ is locally uniformly bounded in $x$. Hence, using \eqref{eq:pk-der-bound} with $k=2$,
\[
\frac12\,|p''_{\mu_\alpha}(x)|\,\sigma_\alpha^2
\le C_x\,p_{\mu_\alpha}(x)\,\mu_\alpha^{-2}\,\sigma_\alpha^2
= C_x\,p_{\mu_\alpha}(x)\,\alpha^{-1}.
\]

On the event \(\{T_\alpha\ge \mu_\alpha/2\}\), Taylor's formula gives
\[
|R_\alpha|
\le
\frac16
\sup_{t\ge \mu_\alpha/2}|p'''_t(x)|
\,|T_\alpha-\mu_\alpha|^3.
\]
Therefore, using \eqref{eq:pk-der-bound} with \(k=3\) and
$\mathbb E\big[|T_\alpha-\mu_\alpha|^3\big]
=
O(\sigma_\alpha^3),$ we get
\[
\begin{aligned}
\mathbb E\big[|R_\alpha|;\,T_\alpha\ge \mu_\alpha/2\big]
\le
C_x p_{\mu_\alpha}(x)\mu_\alpha^{-3}\sigma_\alpha^3
=C_x p_{\mu_\alpha}(x)\alpha^{-3/2}.
\end{aligned}
\]

It remains to consider the lower-tail event
\(\{T_\alpha<\mu_\alpha/2\}\). 
For every \(\theta>0\), we have
\[
\mathbb P\Big(T_\alpha<\frac{\mu_\alpha}{2}\Big)
=
\mathbb P\Big(e^{-\theta T_\alpha}>e^{-\theta\mu_\alpha/2}\Big)
\le
e^{\theta\mu_\alpha/2}\mathbb E[e^{-\theta T_\alpha}].
\]
Since \(T_\alpha\sim\mathrm{Gamma}(\alpha,\lambda)\),
\[
\mathbb E[e^{-\theta T_\alpha}]
=
\left(\frac{\lambda}{\lambda+\theta}\right)^\alpha .
\]
Choosing \(\theta=\lambda\), and using \(\mu_\alpha=\alpha/\lambda\), we obtain
\[
\mathbb P\Big(T_\alpha<\frac{\mu_\alpha}{2}\Big)
\le
e^{\alpha/2}2^{-\alpha}.
\]

By the expansion \eqref{error-term} and the boundedness of the heat kernel
\(t\mapsto p_t(x)\) on \((0,\infty)\) for fixed \(x\neq0\), we obtain
\[
|R_\alpha|\le C_x
\qquad\text{on } \{T_\alpha<\mu_\alpha/2\}.
\]
Therefore, using the Gamma lower-tail estimate,
we get
\[
\mathbb E\big[|R_\alpha|;\,T_\alpha<\mu_\alpha/2\big]
\le
C_x\mathbb P\left(T_\alpha<\frac{\mu_\alpha}{2}\right)\le C_{x}e^{\alpha/2}2^{-\alpha}.
\]
Combining the two estimates yields
\[
|\mathbb E[R_\alpha]|
=
O\!\left(p_{\mu_\alpha}(x)\alpha^{-1}\right).
\]
Plugging these bounds into \eqref{eq:EpT-expansion} yields
\[
\mathbb{E}\big[p_{T_\alpha}(x)\big]
=p_{\mu_\alpha}(x)\bigl(1+O(\alpha^{-1})\bigr),
\qquad \alpha\to\infty.
\]
Finally, since $G^\lambda_\alpha(x)=\lambda^{-\alpha}\mathbb{E}[p_{T_\alpha}(x)]$ and
$\mu_\alpha=\alpha/\lambda$, we obtain
\[
G^\lambda_\alpha(x)
=\lambda^{-\alpha}\,p_{\alpha/\lambda}(x)\bigl(1+O(\alpha^{-1})\bigr),
\]
as claimed. Moreover, the above estimates can be made uniform for $x$ in compact
subsets of $\R^n$, so the $O(\alpha^{-1})$ remainder is uniform on every compact set.
\end{proof}

In the next proposition we show that its value is the logarithmic Bessel
kernel $K_{s+\ln}^{\lambda}$.

\begin{proof}[\textbf{Proof of Proposition \ref{lem:log-bessel-kernel}.}]
Recall that $G^\lambda_\alpha$ is defined by
\[
\widehat{G^\lambda_\alpha}(\xi)=(\lambda+4\pi^2|\xi|^2)^{-\alpha},
\quad \xi\in\R^n,
\]
and that $K_{s+\ln}^{\lambda}$ is the tempered distribution with Fourier transform
\[
\widehat{K_{s+\ln}^{\lambda}}(\xi)=\frac{(\lambda+4\pi^2|\xi|^2)^{-s}}{\ln(\lambda+4\pi^2|\xi|^2)},
\quad \xi\in\R^n.
\]

For every $\xi\in\R^n$, since $\lambda>1$ we have $\ln(\lambda+4\pi^2|\xi|^2)>0$ and thus
\[
\int_s^\infty (\lambda+4\pi^2|\xi|^2)^{-\alpha}\,d\alpha
=\left[-\frac{(\lambda+4\pi^2|\xi|^2)^{-\alpha}}{\ln(\lambda+4\pi^2|\xi|^2)}\right]_{\alpha=s}^{\infty}
=\frac{(\lambda+4\pi^2|\xi|^2)^{-s}}{\ln(\lambda+4\pi^2|\xi|^2)}.
\]
Hence,
\[\int_s^\infty \widehat{G^\lambda_\alpha}(\xi)\,d\alpha=\widehat{K_{s+\ln}^{\lambda}}(\xi)
\qquad\text{for all }\xi\in\R^n.\]

Let $\varphi\in\cS(\R^n)$ and set $\psi:=\cF^{-1}\varphi\in\cS(\R^n)$, so that $\varphi=\cF[\psi]$.
Using the definition of the Fourier transform on tempered distributions,
\begin{equation}\label{eq:hat-identity}
    \langle K_{s+\ln}^{\lambda},\varphi\rangle
=\langle K_{s+\ln}^{\lambda},\cF[\psi]\rangle
=\langle \widehat{K_{s+\ln}^{\lambda}},\psi\rangle.
\end{equation}
Consider the nonnegative function
\[
F(\alpha,\xi):=(\lambda+4\pi^2|\xi|^2)^{-\alpha}\,|\psi(\xi)|
\quad\text{on }[s,\infty)\times\R^n.
\]
Since $\lambda>1$, we have $a(\xi):=\lambda+4\pi^2|\xi|^2>1$ for all $\xi$, hence
\[
\int_s^\infty a(\xi)^{-\alpha}\,d\alpha=\frac{a(\xi)^{-s}}{\ln a(\xi)}.
\]
Moreover, $\ln a(\xi)\ge \ln\lambda>0$, so
\[
\int_s^\infty F(\alpha,\xi)\,d\alpha
=\frac{(\lambda+4\pi^2|\xi|^2)^{-s}}{\ln(\lambda+4\pi^2|\xi|^2)}\,|\psi(\xi)|
\le \frac1{\ln\lambda}\,(\lambda+4\pi^2|\xi|^2)^{-s}\,|\psi(\xi)|.
\]
Since $\psi\in\cS(\R^n)$ and $(\lambda+4\pi^2|\xi|^2)^{-s}$ has at most polynomial decay,
the right-hand side is integrable in $\xi$. Therefore,
\[
\int_{\R^n}\int_s^\infty F(\alpha,\xi)\,d\alpha\,d\xi<\infty,
\]
and Tonelli's theorem applies,
\begin{align*}
\int_s^\infty \langle \widehat{G^\lambda_\alpha},\psi\rangle\,d\alpha
&=\int_s^\infty\int_{\R^n}(\lambda+4\pi^2|\xi|^2)^{-\alpha}\psi(\xi)\,d\xi\,d\alpha\\
&=\int_{\R^n}\int_s^\infty(\lambda+4\pi^2|\xi|^2)^{-\alpha}\,d\alpha\ \psi(\xi)\,d\xi\\
&=\left\langle \int_s^\infty \widehat{G^\lambda_\alpha}\,d\alpha,\ \psi\right\rangle.
\end{align*}
Combining this with \eqref{eq:hat-identity} gives
\begin{equation}\label{eq:pairing-K}
\langle K_{s+\ln}^{\lambda},\varphi\rangle
=\Big\langle \int_s^\infty \widehat{G^\lambda_\alpha}\,d\alpha,\ \psi\Big\rangle
=\int_s^\infty \langle \widehat{G^\lambda_\alpha},\psi\rangle\,d\alpha=\int_s^\infty \langle G^\lambda_\alpha,\varphi\rangle\,d\alpha,
\end{equation}
\medskip

For every $\varphi\in\cS(\R^n)$,
\begin{equation}\label{eq:abs-bound-Galpha}
|\langle G^\lambda_\alpha,\varphi\rangle|
=\Big|\int_{\R^n}G^\lambda_\alpha(x)\varphi(x)\,dx\Big|
\le \|\varphi\|_{L^\infty}\int_{\R^n}G^\lambda_\alpha(x)\,dx
=\|\varphi\|_{L^\infty}\lambda^{-\alpha}.
\end{equation}
We define
\[
T=\int_s^\infty G^\lambda_\alpha\,d\alpha\quad\text{in }\quad\cS'(\R^n),
\]
i.e. the integral is understood in the weak (Pettis) sense, see \cite{talagrand1984pettis,brooks1969representations}. We know this integral exists in $\cS'(\R^n)$ since
\[
\int_s^\infty |\langle G^\lambda_\alpha,\varphi\rangle|\,d\alpha
\le \|\varphi\|_{L^\infty}\int_s^\infty \lambda^{-\alpha}\,d\alpha
=\frac{\lambda^{-s}}{\ln\lambda}\,\|\varphi\|_{L^\infty}<\infty.
\]
Since $(\cS'(\mathbb{R}^n))'=\cS(\mathbb{R}^n)$, we obtain
\[
\langle T,\varphi\rangle=\int_s^\infty \langle G^\lambda_\alpha,\varphi\rangle\,d\alpha.
\]
Consequently, \eqref{eq:pairing-K} becomes
\[
\langle K_{s+\ln}^{\lambda},\varphi\rangle
=\int_s^\infty \langle G^\lambda_\alpha,\varphi\rangle\,d\alpha
=\Big\langle \int_s^\infty G^\lambda_\alpha\,d\alpha,\ \varphi\Big\rangle.
\]
Since $\varphi\in\cS(\R^n)$ was arbitrary, we conclude that
\[
K_{s+\ln}^{\lambda}=\int_s^\infty G^\lambda_\alpha\,d\alpha \qquad\text{in }\cS'(\R^n),
\]
which yields \eqref{eq:K-log-alpha} since
\[\int_s^\infty G^\lambda_{\alpha}(x)\,d\alpha\in L^1(\mathbb{R}^n),\]
thus, we complete the proof.
\end{proof}

In analogy with the one-parameter family of Bessel potential operators
$J_s f := G_s * f$, we introduce the logarithmic one-parameter family by
convolution with the logarithmic Bessel kernel $K_{s+\ln}^{\lambda}$.

\begin{definition}
Let $s>0$ and $\lambda>1$.  
For $f\in\cS(\R^n)$ we define
\[
J_{s+\ln}^{\lambda} f
:= K_{s+\ln}^{\lambda}*f.
\]
\end{definition}

Since $K_{s+\ln}^{\lambda}\in L^1(\R^n)$ by Proposition~\ref{lem:log-bessel-kernel},
Young's inequality implies that $J_{s+\ln}^{\lambda}$ extends uniquely to a
bounded linear operator
\[
J_{s+\ln}^{\lambda}:L^p(\R^n)\to L^p(\R^n),
\qquad 1\le p\le\infty.
\]

\begin{remark}\label{jsgs}
Let $s>0$, $\lambda>1$ and $1\le p\le\infty$. Then for every
$f\in L^p(\R^n)$ one has
\[J_{s+\ln}^{\lambda} f=(\lambda I-\Delta)^{-(s+\ln)}f\quad \text{in}\:\: L^p(\R^n).\]
and
\[
J_{s+\ln}^{\lambda}\bigl((\lambda I-\Delta)^{s+\ln} f\bigr)
=(\lambda I-\Delta)^{s+\ln}\bigl(J_{s+\ln}^{\lambda} f\bigr)=f
\quad\text{in }\cS'(\R^n).
\]
\end{remark}

As a consequence, for $1\le p\le \infty$ we may characterize the logarithmic Bessel
space $\mathcal{L}^{p}_{s+\ln,\lambda}(\R^n)$ either via the operator
$(\lambda I-\Delta)^{s+\ln}$ or via 
$J_{s+\ln}^{\lambda}$.

\begin{coro}
Let $s>0$, $\lambda>1$ and $1\le p\le \infty$. Then
\[
\mathcal{L}^{p}_{s+\ln,\lambda}(\R^n)
=
\bigl\{u\in\cS'(\R^n):(\lambda I-\Delta)^{s+\ln}u\in L^p(\R^n)\bigr\}
=
J_{s+\ln}^{\lambda}(L^p(\R^n)).
\]
\end{coro}

We begin with a key asymptotic lemma for the auxiliary function \(H_s\), which will be used in the near-origin asymptotics and far-field analysis. 

\begin{lemma}\label{lem:Hs-large-time}
Let \(s>0\) and define
\[
H_s(t):=\int_0^\infty \frac{t^p}{\Gamma(s+p)}\,dp,
\qquad t>0.
\]
Then
\[
H_s(t)
=
e^t t^{1-s}\left(1+O(t^{-1})\right),
\qquad t\to\infty .
\]
\end{lemma}

\begin{proof}
We write
\[
H_s(t)
=
t\int_0^\infty
\frac{t^{tq}}{\Gamma(s+tq)}\,dq .
\]
The main contribution will come from a neighborhood of \(q=1\). We first
derive the asymptotic form of the integrand there. By Stirling's formula,
uniformly for \(q\) in any compact subinterval of \((0,\infty)\),
\[
\Gamma(s+tq)
=
\sqrt{2\pi}\,(s+tq)^{s+tq-\frac12}e^{-(s+tq)}
\left(1+O(t^{-1})\right),
\qquad t\to\infty.
\]
Hence
\[
\frac{t^{tq}}{\Gamma(s+tq)}
=
\frac{1}{\sqrt{2\pi}}\,
t^{\frac12-s}
q^{\frac12-s}
\exp\{tq(1-\ln q)\}
\left(1+O(t^{-1})\right),
\]
again uniformly for \(q\) in compact subintervals of \((0,\infty)\). Indeed,
the factor
\[
e^s\left(1+\frac{s}{tq}\right)^{-s-tq+\frac12}
\]
is equal to \(1+O(t^{-1})\) uniformly on such compact intervals.

Set
\[
\phi(q):=q(1-\ln q),
\qquad q>0.
\]
Then
\[
\phi'(q)=-\ln q,
\qquad
\phi''(q)=-\frac1q.
\]
Thus \(\phi\) has a unique maximum at \(q=1\), and $\phi(1)=1,\phi''(1)=-1$.
Let \(\delta\in(0,1)\) be fixed. On the interval
\[
U_\delta:=(1-\delta,1+\delta),
\]
we therefore have
\[
\frac{t^{tq}}{\Gamma(s+tq)}
=
\frac{1}{\sqrt{2\pi}}\,
t^{\frac12-s}
q^{\frac12-s}
e^{t\phi(q)}
\left(1+O(t^{-1})\right).
\]
Consequently,
\[
t\int_{U_\delta}
\frac{t^{tq}}{\Gamma(s+tq)}\,dq
=
\frac{t^{\frac32-s}}{\sqrt{2\pi}}
\int_{U_\delta}
q^{\frac12-s}e^{t\phi(q)}
\,dq
\left(1+O(t^{-1})\right).
\]
Applying the standard one-dimensional Laplace method, see \cite[Chapter 2]{murray2012asymptotic}, we get
\[
\int_{U_\delta}
q^{\frac12-s}e^{t\phi(q)}
\,dq
=
e^t
\sqrt{\frac{2\pi}{t}}
\left(1+O(t^{-1})\right).
\]
Therefore
\[
t\int_{U_\delta}
\frac{t^{tq}}{\Gamma(s+tq)}\,dq
=
e^t t^{1-s}
\left(1+O(t^{-1})\right).
\]

Since \(q=1\) is the unique maximum point of \(\phi\), and since
\([\varepsilon,R]\setminus U_\delta\) is compact and does not contain \(1\),
there exists \(c>0\) such that
\[
\phi(q)\le 1-c,
\qquad q\in[\varepsilon,R]\setminus U_\delta.
\]
Using the uniform Stirling expansion on this compact set, we obtain
\[
t\int_{[\varepsilon,R]\setminus U_\delta}
\frac{t^{tq}}{\Gamma(s+tq)}\,dq
\le
C t^{\frac32-s}e^{(1-c)t}.
\]
Hence this contribution is exponentially smaller than the principal term
\(e^t t^{1-s}\), since
\[
\frac{t^{\frac32-s}e^{(1-c)t}}{e^t t^{1-s}}
=
t^{\frac12}e^{-ct}
\to0 .
\]

For the endpoint \(q\to0^+\),  choosing
\(\varepsilon>0\) sufficiently small, Stirling's formula for \(p\ge1\), together
with the elementary boundedness of \(1/\Gamma(s+p)\) on bounded \(p\)-intervals,
gives
\[
t\int_0^\varepsilon
\frac{t^{tq}}{\Gamma(s+tq)}\,dq
=
\int_0^{\varepsilon t}
\frac{t^p}{\Gamma(s+p)}\,dp
=
O\left(e^{(1-c)t}\right)
\]
for some \(0<c<1\). Similarly, since \(\phi(q)\to-\infty\) as \(q\to\infty\),
the tail \(q\ge R\) satisfies
\[
t\int_R^\infty
\frac{t^{tq}}{\Gamma(s+tq)}\,dq
=
O\left(e^{(1-c)t}\right)
\]
for \(R>1+\delta\) sufficiently large. Hence
\[
t\int_{(0,\infty)\setminus U_\delta}
\frac{t^{tq}}{\Gamma(s+tq)}\,dq
=
O\left(e^{(1-c)t}\right)
=
o\left(e^t t^{1-s-1}\right).
\]
Combining this with the contribution from \(U_\delta\), we obtain
\[
H_s(t)
=
e^t t^{1-s}
\left(1+O(t^{-1})\right),
\qquad t\to\infty,
\]
which complete the proof.
\end{proof}

We now derive the sharp near-origin asymptotics of $K_{s+\ln}^{\lambda}$ by starting
from the integral representation in Proposition~\ref{lem:log-bessel-kernel} and
analyzing the contribution of small spatial scales.

\begin{proof}[\textbf{Proof of Proposition \ref{prop:kernel_asymptotics}.}]
(i) Assume $n>2s$. By Proposition \ref{lem:log-bessel-kernel}, we have
\[
K_{s+\ln}^{\lambda}(x) = \int_0^\infty G_{s+p}^\lambda(x) \,dp.
\]
Using the standard heat kernel representation,  the Bessel kernel admits the representation:
\[
G_\alpha^\lambda(x) = \frac{1}{\Gamma(\alpha)(4\pi)^{\frac{n}{2}}} \int_0^\infty t^{\alpha-1-\frac{n}{2}} e^{-\lambda t - \frac{|x|^2}{4t}} \,dt.
\]
Let $r = |x| > 0$. We split the $p$-integral into a local part and a tail part using a small fixed constant $0 < \delta < \frac{n}{2}-s$:
\[
K_{s+\ln}^{\lambda}(x) = \int_0^\delta G_{s+p}^\lambda(x) \,dp + \int_\delta^\infty G_{s+p}^\lambda(x) \,dp =: I_1 + I_2.
\]
\medskip
\noindent\textbf{Analysis of the main term $I_1$:} 
Substituting $t = r^2 u$ into the heat kernel yields
\[
G_{s+p}^\lambda(x) = \frac{r^{2s+2p-n}}{(4\pi)^{\frac{n}{2}}} \Phi(p, r), \quad \text{where } \Phi(p, r) := \frac{1}{\Gamma(s+p)} \int_0^\infty u^{s+p-1-\frac{n}{2}} e^{-\lambda r^2 u - \frac{1}{4u}} \,du.
\]
Thus, $$I_1 = \frac{r^{2s-n}}{(4\pi)^{\frac{n}{2}}} \int_0^\delta r^{2p} \Phi(p, r) \,dp.$$ 
Let $L = -2\ln r$. As $r \to 0$, $L \to \infty$. We evaluate the integral using the substitution $q = pL$:
\[
\int_0^\delta e^{-pL} \Phi(p, r) \,dp = \frac{1}{L} \int_0^{\delta L} e^{-q} \Phi\Big(\frac{q}{L}, r\Big) \,dq.
\]
Because $s+\delta < \frac{n}{2}$, we can uniformly bound $e^{-\lambda r^2 u} \le 1$, yielding
\[
\Phi(p, r) \le \Phi(p, 0) = \frac{4^{\frac{n}{2}-s-p} \Gamma(\frac{n}{2}-s-p)}{\Gamma(s+p)}.
\]
Since $p \mapsto \Phi(p, 0)$ is strictly positive and continuous on the compact interval $[0, \delta]$, it is bounded by some constant $M < \infty$. Consequently, the integrand is uniformly dominated by $M e^{-q} \in L^1(0, \infty)$. By the Lebesgue Dominated Convergence Theorem, 
\[
\lim_{r \to 0} \int_0^{\delta L} e^{-q} \Phi\Big(\frac{q}{L}, r\Big) \,dq = \int_0^\infty e^{-q} \Phi(0, 0) \,dq = \Phi(0, 0) = \frac{4^{\frac{n}{2}-s} \Gamma(\frac{n}{2}-s)}{\Gamma(s)}.
\]
Therefore, the local part behaves exactly as
\[
\lim_{|x|\rightarrow 0}I_1\frac{\ln \frac{1}{|x|^2}}{|x|^{2s-n}} = \frac{4^{\frac{n}{2}-s} \Gamma(\frac{n}{2}-s)}{(4\pi)^{\frac{n}{2}}\Gamma(s)}=\frac{\Gamma(\frac{n}{2}-s)}{\pi^{n/2}2^{2s}\Gamma(s)}.
\]

\medskip
\noindent\textbf{Analysis of the tail error $I_2$:}
We estimate $I_2$ directly from the heat kernel representation. By Tonelli's theorem, we exchange the $p$ and $t$ integrals:
\[
I_2 = \int_\delta^\infty G_{s+p}^\lambda(x) \,dp = \frac{1}{(4\pi)^{\frac{n}{2}}} \int_0^\infty \left( \int_\delta^\infty \frac{t^{p}}{\Gamma(s+p)} \,dp \right) t^{s-1-\frac{n}{2}} e^{-\lambda t - \frac{r^2}{4t}} \,dt.
\]
Let $$H(t) := \int_\delta^\infty \frac{t^p}{\Gamma(s+p)} \,dp.$$ 
We bound $H(t)$ in two regimes:

1) For $0 < t \le 1$, since $p \ge \delta$, we have $t^p \le t^\delta$. Thus $H(t) \le t^\delta \int_\delta^\infty \frac{1}{\Gamma(s+p)} \,dp \le C_1 t^\delta$.

2) For $t > 1$, by Lemma \ref{lem:Hs-large-time}, $H(t) \le  C_2 t^{1-s} e^t$.

Splitting the $t$-integral at $t=1$, we obtain
\[
I_2 \lesssim \int_0^1 t^{s+\delta-1-\frac{n}{2}} e^{-\frac{r^2}{4t}} \,dt + \int_1^\infty t^{-\frac{n}{2}} e^{-(\lambda-1)t} \,dt.
\]
Because $\lambda > 1$, the second integral converges to a finite constant. For the first integral, substituting $t = r^2 u$ bounds it by $$r^{2s+2\delta-n} \int_0^\infty u^{s+\delta-1-\frac{n}{2}} e^{-\frac{1}{4u}} \,du = O(r^{2s+2\delta-n}).$$
Hence, $I_2 = O(r^{2s+2\delta-n}) + O(1)$. Since $\delta > 0$ and $n>2s$, the ratio
\[
\frac{r^{2s+2\delta-n}+1}{r^{2s-n}/(-\ln r)} \to 0 \quad \text{as } r \to 0.
\]
Thus, we complete the proof.

\medskip
(ii) Assume $n=2s$. We exchange the order of integration:
\[
K_{\frac{n}{2}+\ln}^{\lambda}(x) = \frac{1}{(4\pi)^{\frac{n}{2}}} \int_0^\infty t^{-1} e^{-\lambda t - \frac{|x|^2}{4t}} H(t) \,dt, \quad \text{where } H(t) := \int_0^\infty \frac{t^p}{\Gamma(\frac{n}{2}+p)} \,dp.
\]
For $t\in(0,1/2]$, set $L:=-\ln t>0$ and write
\[
H(t)=\int_{0}^{\infty} \frac{t^{p}}{\Gamma(\frac n2+p)}\,dp
=\int_{0}^{\infty} e^{-Lp}\,f(p)\,dp,
\qquad
f(p):=\frac{1}{\Gamma(\frac n2+p)}.
\]
Since $1/\Gamma$ is entire, $f$ is $C^\infty$.
We split the integral into $[0,1]$ and $[1,\infty)$.

On $[0,1]$, by Taylor's theorem there exists $C>0$ such that
\[
f(p)=f(0)+f'(0)p+R(p),\qquad |R(p)|\le Cp^{2}\quad \text{for all }p\in[0,1].
\]
Therefore,
\[
\int_{0}^{1} e^{-Lp}f(p)\,dp
=f(0)\int_{0}^{1} e^{-Lp}\,dp
+f'(0)\int_{0}^{1} p e^{-Lp}\,dp
+O\!\left(\int_{0}^{1} p^{2}e^{-Lp}\,dp\right)
=\frac{f(0)}{L}+O\!\left(\frac{1}{L^{2}}\right),
\]
where we used
\[
\int_{0}^{1} e^{-Lp}\,dp=\frac{1-e^{-L}}{L}=\frac{1}{L}+O(e^{-L}),\quad
\int_{0}^{1} p e^{-Lp}\,dp=O\!\left(\frac{1}{L^{2}}\right),\quad
\int_{0}^{1} p^{2} e^{-Lp}\,dp=O\!\left(\frac{1}{L^{3}}\right).
\]

For the tail, since \(0<t\le \frac12\) and \(p\ge1\), we have
\(t^p\le t\). Moreover,
\[
\int_1^\infty \frac{1}{\Gamma(\frac n2+p)}\,dp<\infty,
\]
Hence
\[
0\le
\int_1^\infty
\frac{t^p}{\Gamma(\frac n2+p)}\,dp
\le
t
\int_1^\infty
\frac{1}{\Gamma(\frac n2+p)}\,dp
\le C t
=
O(e^{-L}),
\]
which is $o(L^{-2})$ as $t\rightarrow 0$. Combining the two parts, we conclude that, as $t\rightarrow 0$,
\begin{equation}\label{eq:asym}
    H(t)=\frac{1}{\Gamma(\frac n2)(-\ln t)}+O\!\left(\frac{1}{(-\ln t)^{2}}\right).
\end{equation}

Fix $\delta\in(0,1/2)$. We split the $t$-integral into a local part
$(0,\delta]$ and a tail part $(\delta,\infty)$.
For the tail, using the bound $H(t)\le C_\delta t^{1-\frac{n}{2}} e^{t}$ for $t\ge\delta$, which has been proved in Lemma \ref{lem:Hs-large-time}, we obtain 
\[
\int_\delta^\infty t^{-1} e^{-\lambda t-\frac{r^2}{4t}} H(t)\,dt
\le C_\delta \int_\delta^\infty t^{-\frac{n}{2}} e^{-(\lambda-1)t}\,dt
=O(1),
\quad r>0.
\]
On the local interval $(0,\delta]$, we use the asymptotic expansion \eqref{eq:asym},
which yields
\[
K_{\frac n2+\ln}^{\lambda}(x)
=\frac{1}{(4\pi)^{\frac n2}\Gamma(\frac n2)}
\int_0^\delta t^{-1} e^{-\lambda t-\frac{r^2}{4t}}\frac{1}{-\ln t}\,dt
+O(1),
\quad r>0.
\]
Since $e^{-\lambda t}=1+O(t)$ as $t\downarrow 0$,
\[
\int_0^\delta t^{-1} e^{-\frac{r^2}{4t}}\frac{O(t)}{-\ln t}\,dt
\le C\int_0^\delta e^{-\frac{r^2}{4t}}\frac{1}{-\ln t}\,dt
\le C\int_0^\delta \frac{1}{-\ln t}\,dt<\infty,
\]
which can be absorbed into the $O(1)$ term. 
Therefore it remains to analyze
\[
I(r):=\int_0^\delta t^{-1} e^{-\frac{r^2}{4t}}\frac{1}{-\ln t}\,dt,
\qquad 0<\delta<1.
\]
Scaling \(t=r^2u\), we obtain
\[
I(r)
=
\int_0^{\delta/r^2}
\frac{e^{-\frac1{4u}}}
{u(-2\ln r-\ln u)}
\,du.
\]

We split the integral into \((0,1)\) and \((1,\delta/r^2)\). For \(u\in(0,1)\),
we have
\[
\int_0^1
\frac{e^{-\frac1{4u}}}
{u(-2\ln r-\ln u)}
\,du
\le
\frac1{-2\ln r}
\int_0^1 u^{-1}e^{-\frac1{4u}}\,du
=
O\!\left(\frac1{-\ln r}\right).
\]

For \(u\in(1,\delta/r^2)\), we write
\[
e^{-\frac1{4u}}
=
1-\left(1-e^{-\frac1{4u}}\right).
\]
The error term is bounded. Indeed, since $1-e^{-\frac1{4u}}\le \frac1{4u},$
and
\[
-2\ln r-\ln u
=
-\ln(r^2u)
\ge
-\ln\delta>0
\qquad
\text{for }1<u<\frac{\delta}{r^2},
\]
we get
\[
\int_1^{\delta/r^2}
\frac{1-e^{-\frac1{4u}}}
{u(-2\ln r-\ln u)}
\,du
\le
C_\delta
\int_1^\infty \frac{du}{u^2}
=
O(1).
\]
Hence
\[
I(r)
=
\int_1^{\delta/r^2}
\frac{du}{u(-2\ln r-\ln u)}
+
O(1).
\]

The remaining integral can be computed explicitly:
\[
\int_1^{\delta/r^2}
\frac{du}{u(-2\ln r-\ln u)}
=
\ln(-2\ln r)-\ln(-\ln\delta)
=
\ln(-\ln r)+O(1).
\]
Thus
\[
I(r)=\ln(-\ln r)+O(1),
\qquad r\downarrow0.
\]

Reintroducing the prefactor, we achieve the precise double-logarithmic asymptotic behavior:
\[
K_{\frac{n}{2}+\ln}^{\lambda}(x) \sim \frac{1}{(4\pi)^{\frac{n}{2}} \Gamma(\frac{n}{2})} \ln(-\ln |x|)  \quad \text{as } |x| \to 0.
\]

\medskip
(iii) Assume $n<2s$. By Tonelli's Theorem:
\[
K_{s+\ln}^{\lambda}(x) = \frac{1}{(4\pi)^{\frac{n}{2}}} \int_0^\infty \int_0^\infty \frac{t^{s+p-1-\frac{n}{2}}}{\Gamma(s+p)} e^{-\lambda t} e^{-\frac{|x|^2}{4t}} \, dp \, dt.
\]
As $|x| \to 0$, the function $e^{-\frac{|x|^2}{4t}}$ increases monotonically to $1$ for every $t > 0$. By the Monotone Convergence Theorem, we can pass the limit inside the integrals:
\[
\lim_{|x| \to 0} K_{s+\ln}^{\lambda}(x) = \frac{1}{(4\pi)^{\frac{n}{2}}} \int_0^\infty \left( \int_0^\infty \frac{t^{s+p-1-\frac{n}{2}}}{\Gamma(s+p)} e^{-\lambda t} \,dt \right) dp.
\]
Evaluating the inner Gamma integral yields $\lambda^{-(s+p-\frac{n}{2})} \Gamma(s+p-\frac{n}{2})$. Thus,
\[
K_{s+\ln}^{\lambda}(0) = \frac{1}{(4\pi)^{\frac{n}{2}}} \int_0^\infty \frac{\Gamma(s+p-\frac{n}{2})}{\Gamma(s+p)} \lambda^{\frac{n}{2}-s-p} \,dp.
\]
Since $\frac{\Gamma(p+A)}{\Gamma(p+B)} \sim p^{A-B}$ as $p \to \infty$, the integrand decays exponentially as $\lambda^{-p}$ since $\lambda>1$. Therefore, the limit is a strictly finite positive constant, confirming continuity at the origin.

This completes the rigorous characterization for all three regimes.
\end{proof}

Next, we prove the far-field asymptotic.

\begin{proof}[\textbf{Proof of Proposition \ref{lem:Klog-asymp-infty-final}.}]
Set $r:=|x|$ and $\mu:=\lambda-1>0$. Define
\[
H_s(t):=\int_0^\infty \frac{t^p}{\Gamma(s+p)}\,dp,
\qquad t>0.
\]
By Lemma \ref{lem:Hs-large-time}, we have
\[
H_s(t)
=
e^t t^{1-s}\left(1+O(t^{-1})\right),
\qquad t\to\infty.
\]
We shall also use the rough bound
\[
H_s(t)
\le
C_s e^t(1+t)^M,
\qquad t>0,
\]
for some constants \(C_s>0\) and \(M>0\).

Starting from the integral representation, we make the change of variables $t=r\tau.$
Then 
\begin{align}
K_{s+\ln}^{\lambda}(x)
&=
\frac{1}{(4\pi)^{\frac n2}}
r^{s-\frac n2}
\int_0^\infty
H_s(r\tau)\,
\tau^{s-1-\frac n2}
\exp\left(-\lambda r\tau-\frac{r}{4\tau}\right)
\,d\tau .
\label{eq:K-after-scaling}
\end{align}

The phase function is
\[
\phi(\tau)
:=
\mu\tau+\frac{1}{4\tau},
\qquad \tau>0.
\]
Then
\[
\phi'(\tau)
=
\mu-\frac{1}{4\tau^2},\quad \phi''(\tau)
=
\frac{1}{2\tau^3}>0.
\]
Hence \(\phi\) has a unique global minimum point $\tau_0
=
\frac{1}{2\sqrt{\mu}}.$ At this point, $$\phi(\tau_0)
=
\sqrt{\mu},\quad \phi''(\tau_0)
=
4\mu^{3/2}$$
Thus the leading contribution comes from $t=r\tau_0
=
\frac{r}{2\sqrt{\mu}},$
which is a large-time region \(t\sim r\).

Let \(U\Subset(0,\infty)\) be a fixed open interval containing \(\tau_0\), the large-time asymptotic of \(H_s\) gives, uniformly for \(\tau\in U\),
\[
H_s(r\tau)
=
e^{r\tau}(r\tau)^{1-s}
\left(1+O(r^{-1})\right),
\qquad r\to\infty.
\]
Therefore the contribution from \(U\), denoted by \(K_U(r)\), satisfies
\begin{align}
K_U(r)
&=
\frac{1}{(4\pi)^{\frac n2}}
r^{s-\frac n2}
\int_U
H_s(r\tau)\,
\tau^{s-1-\frac n2}
\exp\left(-\lambda r\tau-\frac{r}{4\tau}\right)
\,d\tau
\nonumber\\
&=
\frac{1}{(4\pi)^{\frac n2}}
r^{s-\frac n2}
\int_U
e^{r\tau}(r\tau)^{1-s}
\tau^{s-1-\frac n2}
\exp\left(-\lambda r\tau-\frac{r}{4\tau}\right)
\left(1+O(r^{-1})\right)
\,d\tau
\nonumber\\
&=
\frac{1}{(4\pi)^{\frac n2}}
r^{1-\frac n2}
\int_U
\tau^{-\frac n2}
e^{-r\phi(\tau)}
\left(1+O(r^{-1})\right)
\,d\tau .
\label{eq:local-main}
\end{align}

Next we show that the contribution outside \(U\) is exponentially smaller. Since \(\phi\) has a unique strict global minimum at \(\tau_0\), after choosing \(U\) sufficiently small, there exists \(\eta>0\) such that
\[
\phi(\tau)
\ge
\sqrt{\mu}+\eta,
\qquad
\tau\in (0,\infty)\setminus U.
\]
Using the rough bound for \(H_s\), we obtain
\[
H_s(r\tau)e^{-\lambda r\tau}
\le
C_s(1+r\tau)^M e^{-\mu r\tau}.
\]
Consequently,
\[
\left|
K_{(0,\infty)\setminus U}(r)
\right|
\le
C r^{s-\frac n2}
\int_{(0,\infty)\setminus U}
(1+r\tau)^M
\tau^{s-1-\frac n2}
e^{-r\phi(\tau)}
\,d\tau .
\]
By the standard localization estimate for Laplace integrals, this gives
\[
K_{(0,\infty)\setminus U}(r)
=
O\left(r^{M_1}e^{-(\sqrt{\mu}+\eta)r}\right)
\]
for some \(M_1>0\). In particular,
\[
K_{(0,\infty)\setminus U}(r)
=
o\left(
r^{-\frac{n-1}{2}}e^{-\sqrt{\mu}r}
\right).
\]
Thus the leading asymptotic is determined by \(K_U(r)\).

Applying the standard Laplace method to (see \cite[Chapter 2]{murray2012asymptotic})
\[
\int_U
\tau^{-\frac n2}
e^{-r\phi(\tau)}
\,d\tau,
\]
we obtain
\[
\int_U
\tau^{-\frac n2}
e^{-r\phi(\tau)}
\,d\tau
\sim
\tau_0^{-\frac n2}
e^{-r\phi(\tau_0)}
\sqrt{\frac{2\pi}{r\phi''(\tau_0)}}.
\]
Since
\[
\tau_0
=
\frac{1}{2\sqrt{\mu}},
\qquad
\phi(\tau_0)
=
\sqrt{\mu},
\qquad
\phi''(\tau_0)
=
4\mu^{3/2},
\]
we have
\[
\tau_0^{-\frac n2}
=
2^{\frac n2}\mu^{\frac n4},
\]
and
\[
\sqrt{\frac{2\pi}{r\phi''(\tau_0)}}
=
\sqrt{\frac{2\pi}{4\mu^{3/2}r}}
=
\sqrt{\frac{\pi}{2}}\,
\mu^{-\frac34}r^{-\frac12}.
\]
Hence
\[
\int_U
\tau^{-\frac n2}
e^{-r\phi(\tau)}
\,d\tau
\sim
2^{\frac{n-1}{2}}
\sqrt{\pi}\,
\mu^{\frac{n-3}{4}}
r^{-\frac12}
e^{-\sqrt{\mu}r}.
\]
Substituting this into \eqref{eq:local-main}, we get
\begin{align*}
K_U(r)
&\sim
\frac{1}{(4\pi)^{\frac n2}}
r^{1-\frac n2}
\cdot
2^{\frac{n-1}{2}}
\sqrt{\pi}\,
\mu^{\frac{n-3}{4}}
r^{-\frac12}
e^{-\sqrt{\mu}r}
\\
&=
\frac{
\mu^{\frac{n-3}{4}}
}{
2^{\frac{n+1}{2}}\pi^{\frac{n-1}{2}}
}
r^{-\frac{n-1}{2}}
e^{-\sqrt{\mu}r}.
\end{align*}
Since the contribution outside \(U\) is exponentially smaller, we obtain
\[
K_{s+\ln}^{\lambda}(x)
\sim
\frac{
\mu^{\frac{n-3}{4}}
}{
2^{\frac{n+1}{2}}\pi^{\frac{n-1}{2}}
}
r^{-\frac{n-1}{2}}
e^{-\sqrt{\mu}r}.
\]
Finally, recalling that
\[
\mu=\lambda-1,
\qquad
r=|x|,
\]
we conclude that
\[
K_{s+\ln}^{\lambda}(x)
\sim
\frac{
(\lambda-1)^{\frac{n-3}{4}}
}{
2^{\frac{n+1}{2}}\pi^{\frac{n-1}{2}}
}
|x|^{-\frac{n-1}{2}}
e^{-\sqrt{\lambda-1}\,|x|}.
\]
This completes the proof.
\end{proof}

We highlight a striking analytical feature of the logarithmic modifier. In the classical fractional framework, the Riesz potential operator $I_{2s}f := |\cdot|^{2s-n} * f$ maps $L^p(\R^n)$ continuously into the critical Sobolev space $L^{p^*}(\R^n)$. This boundedness relies heavily on the deep Hardy-Littlewood-Sobolev (HLS) inequality, precisely because the Riesz kernel $|\cdot|^{2s-n}$ fails to belong to the Lebesgue space $L^r(\R^n)$ for the matching exponent $r = \frac{n}{n-2s}$ (it only resides in the weak Lorentz space $L^{r,\infty}(\R^n)$). 

However, the effect of the logarithmic term in our operator strictly pushes the singular kernel into the standard $L^r(\R^n)$ space. Consequently, the highly non-trivial HLS machinery can be completely bypassed and replaced by the elementary Young's convolution inequality. The following lemma rigorously establishes this critical integrability.

\begin{lemma}
\label{lem:kernel_Lr}
Let $n>2s$ and $\lambda > 1$. The logarithmic fractional kernel $K_{s+\ln}^\lambda$ is globally integrable to the critical power $r = \frac{n}{n-2s}$. Consequently, for any $1 < p < \frac{n}{2s}$ and the critical Sobolev exponent $p^* = \frac{np}{n-2sp}$, the convolution operator $f \mapsto K_{s+\ln}^\lambda * f$ maps $L^p(\R^n)$ continuously into $L^{p^*}(\R^n)$ with the strong norm estimate:
\[
\|K_{s+\ln}^\lambda * f\|_{L^{p^*}(\R^n)} \le \|K_{s+\ln}^\lambda\|_{L^r(\R^n)} \|f\|_{L^p(\R^n)}.
\]
\end{lemma}

\begin{proof}
 For the local integrability around the origin, we rely on the precise asymptotic behavior established in Proposition~\ref{prop:kernel_asymptotics}. Fix a sufficiently small radius $R_0 \in (0, 1/2)$, we have
\[
\int_{B_{R_0}(0)} \big| K_{s+\ln}^\lambda(x) \big|^r \,dx \le C \int_{|x| \le R_0} \frac{|x|^{(2s-n)r}}{(-\ln|x|)^r} \,dx.
\]
Substituting the critical exponent $r = \frac{n}{n-2s}>1$, the integral becomes
\[
C |\mathbb{S}|^{n-1} \int_0^{R_0} \frac{\rho^{-n}}{(-\ln \rho)^r} \rho^{n-1} \,d\rho =C|\mathbb{S}|^{n-1}  \int_{-\ln R_0}^\infty u^{-r} \,du<\infty.
\]

For the behavior at infinity, by Proposition \ref{lem:Klog-asymp-infty-final}, the kernel $K_{s+\ln}^\lambda(x)$ exhibits an exponential decay $O(|x|^{\frac{1-n}{2}}e^{-\sqrt{\lambda-1}|x|})$ for large $|x|$, it is absolutely integrable to any power $r \ge 1$ outside the ball $B_{R_0}(0)$. Combining both regions yields $K_{s+\ln}^\lambda \in L^r(\R^n)$.

Finally, the strong mapping property follows immediately from Young's inequality  (c.f. \cite[Theorem 1.2.12]{grafakos2008classical})  for convolutions since
\[
1 + \frac{1}{p^*} = 1 + \left( \frac{1}{p} - \frac{2s}{n} \right) = \frac{1}{p} + \left( 1 - \frac{2s}{n} \right) = \frac{1}{p} + \frac{1}{r},
\]
which complete the proof.
\end{proof}

    
\subsection{Logarithmic Riesz Potential}

We recall that
\[
\cS_\infty(\R^n):=\Bigl\{\varphi\in\cS(\R^n):\ \int_{\R^n}x^\gamma\varphi(x)\,dx=0,\ \forall\gamma\in\N_0^n\Bigr\},
\]
and
\[
\cS_0(\R^n):=\Bigl\{\psi\in\cS(\R^n):\ \partial^\gamma\psi(0)=0,\ \forall\gamma\in\N_0^n\Bigr\}.
\]
Then $\cF:\cS_\infty'\to\cS_0'$ is an isomorphism, and  $\cF^{-1}:\cS_0'\to\cS_\infty'$ is an isomorphism as well.

Fix $s>0$ and set
\[
m_{s+\ln}(\xi):=(4\pi^2|\xi|^2)^s\ln(4\pi^2|\xi|^2),\qquad \xi\in\R^n\setminus\{0\},
\qquad m_{s+\ln}(0):=0.
\]
Then $m_{s+\ln}\in C^\infty(\R^n\setminus\{0\})$ and has at most polynomial growth at infinity, while it has a
logarithmic singularity at $\xi=0$. Since every $\psi\in\cS_0(\R^n)$ vanishes to infinite order at $\xi=0$,
the product $m_{s+\ln}\psi$ extends (by setting it equal to $0$ at $\xi=0$) to a function in $\cS_0(\R^n)$.
Consequently, the multiplication operator
\[
M_{s+\ln}:\cS_0(\R^n)\to\cS_0(\R^n),\qquad \psi\mapsto m_{s+\ln}\psi,
\]
is well-defined and continuous.

We begin with a Bernstein-type estimate on dyadic annuli, which allows us to
control the $L^1$-norm of an inverse Fourier transform in terms of uniform bounds on
finitely many derivatives.

\begin{lemma}
\label{lem:bernstein_annulus}
Let $n\ge1$ and let $N>n$ be an integer. Then there exists a constant
$C=C(n,N)>0$ such that for every $j\in\Z$ and every $h\in C^N(\R^n)$ satisfying
\[
\operatorname{supp} h \subset \{\xi\in\R^n:2^{j-1}\le |\xi|\le 2^{j+1}\},
\]
one has
\begin{equation}\label{eq:bernstein_annulus_lemma}
\|\cF^{-1}h\|_{L^1(\R^n)}
\le C\sum_{|\alpha|\le N} 2^{j|\alpha|}\,\|\partial^\alpha h\|_{L^\infty(\R^n)}.
\end{equation}
\end{lemma}

\begin{proof}
Fix $j\in\Z$ and such an $h$. Define the rescaled function
\[
H(\eta):=h(2^j\eta),\qquad \eta\in\R^n.
\]
Then $\operatorname{supp} H\subset\{\eta:2^{-1}\le|\eta|\le 2\}$, and by the change of
variables $\xi=2^j\eta$, we obtain 
\[\cF^{-1}h(x)
=\int_{\R^n} e^{2\pi i x\cdot\xi}h(\xi)\,d\xi
=2^{jn}\int_{\R^n} e^{2\pi i (2^j x)\cdot\eta}H(\eta)\,d\eta
=2^{jn}\,(\cF^{-1}H)(2^j x).\]
Taking $L^1$-norms and using the change of variables $y=2^j x$ yields
\begin{equation}\label{eq:L1_scaling}
\|\cF^{-1}h\|_{L^1}
=\int_{\R^n} 2^{jn}\,|(\cF^{-1}H)(2^j x)|\,dx
=\int_{\R^n} |(\cF^{-1}H)(y)|\,dy
=\|\cF^{-1}H\|_{L^1}.
\end{equation}

\medskip
\medskip
\noindent\textbf{Step 1: A pointwise decay estimate for \(\mathcal F^{-1}H\).}
Let \(g:=\mathcal F^{-1}H\). For \(x\neq0\), we use the identity
\[
e^{2\pi i x\cdot\eta}
=
\frac{1}{2\pi i |x|^2}
(x\cdot\nabla_\eta)e^{2\pi i x\cdot\eta}.
\]
Iterating this identity \(N\) times and integrating by parts, we get
\[
g(x)
=
\frac{(-1)^N}{(2\pi i |x|^2)^N}
\int_{\mathbb R^n}
e^{2\pi i x\cdot\eta}
(x\cdot\nabla_\eta)^N H(\eta)\,d\eta .
\]
Since
\[
(x\cdot\nabla_\eta)^N H
=
\sum_{|\alpha|=N} c_\alpha x^\alpha \partial^\alpha H,
\]
we have
\[
|(x\cdot\nabla_\eta)^N H(\eta)|
\le
C_N |x|^N
\sum_{|\alpha|=N}|\partial^\alpha H(\eta)|.
\]
Therefore
\[
|g(x)|
\le
C_N |x|^{-N}
\sum_{|\alpha|=N}\|\partial^\alpha H\|_{L^1}.
\]
Since
\[
\operatorname{supp}H\subset\{\eta:2^{-1}\le |\eta|\le2\},
\]
this support has finite measure, and hence
\[
\|\partial^\alpha H\|_{L^1}
\le
C\|\partial^\alpha H\|_{L^\infty}.
\]
Thus
\begin{equation}\label{eq:pointwise_decay}
    |g(x)|
\le
C |x|^{-N}
\sum_{|\alpha|=N}\|\partial^\alpha H\|_{L^\infty},
\qquad x\neq0.
\end{equation}
Also trivially $|g(x)|\le \|H\|_{L^1}\le C\|H\|_{L^\infty}$.

\medskip
\noindent\textbf{Step 2: Integrate the pointwise bounds.}
Split $\R^n$ into $\{|x|\le1\}$ and $\{|x|>1\}$. Using the trivial bound on the unit
ball,
\[
\int_{|x|\le1}|g(x)|\,dx
\le |B(0,1)|\,\|H\|_{L^1}
\le C\,\|H\|_{L^\infty}.
\]
For $|x|>1$, use \eqref{eq:pointwise_decay} and the assumption $N>n$:
\[
\int_{|x|>1}|g(x)|\,dx
\le C\sum_{|\alpha|=N}\|\partial^\alpha H\|_{L^\infty}\int_{|x|>1}|x|^{-N}\,dx
\le C\,\sum_{|\alpha|=N}\|\partial^\alpha H\|_{L^\infty}.
\]
Combining these gives
\begin{equation}\label{eq:L1_unit_annulus}
\|\cF^{-1}H\|_{L^1}
\le C\sum_{|\alpha|\le N}\|\partial^\alpha H\|_{L^\infty}.
\end{equation}

\medskip
\noindent\textbf{Step 3: Scale back to general $j$.}
By definition $H(\eta)=h(2^j\eta)$, so for any multi-index $\alpha$,
\[
\partial^\alpha H(\eta)=2^{j|\alpha|}\,(\partial^\alpha h)(2^j\eta),
\qquad\text{hence}\qquad
\|\partial^\alpha H\|_{L^\infty}=2^{j|\alpha|}\,\|\partial^\alpha h\|_{L^\infty}.
\]
Insert this into \eqref{eq:L1_unit_annulus} and use \eqref{eq:L1_scaling}:
\[
\|\cF^{-1}h\|_{L^1}
=\|\cF^{-1}H\|_{L^1}
\le C\sum_{|\alpha|\le N} 2^{j|\alpha|}\,\|\partial^\alpha h\|_{L^\infty}.
\]
This is exactly \eqref{eq:bernstein_annulus_lemma}.
\end{proof}

\begin{proof}[\textbf{Proof of Proposition \ref{prop:bridge-measures}.}]

\textbf{(i).} Set $f(\xi):=m(\xi)-1\in L^{\infty}(\R^n)$.

\medskip
\noindent\textbf{Step 1. Dyadic partition.}
Fix a radial function $\varphi\in C_c^\infty(\R^n)$ such that
\[
\operatorname{supp}\varphi\subset\Big\{\tfrac12\le|\xi|\le2\Big\},\qquad
\sum_{j\in\Z}\varphi(2^{-j}\xi)=1\quad\text{for }\xi\neq0.
\]
Set $\varphi_j(\xi):=\varphi(2^{-j}\xi)$. Then
$\operatorname{supp}\varphi_j\subset\{\xi:2^{j-1}\le|\xi|\le2^{j+1}\}$, $\varphi_j(0)=0$, and for every multi-index
$\alpha$,
$\|\partial^\alpha\varphi_j\|_{L^\infty}\lesssim 2^{-j|\alpha|}$.
Define the low-frequency cut-off
\[
\chi(\xi):=\sum_{j\le0}\varphi_j(\xi),\,\xi\ne 0;\quad \chi(0)=0.
\]
Then $\chi\in C_c^\infty(\R^n)$, and one has
$\chi(\xi)=1$ for $|\xi|\le \frac12$ and $\operatorname{supp}\chi\subset\{|\xi|\le2\}$.
Consequently,
\[
f(\xi)=f(\xi)\chi(\xi)+\sum_{j\ge1} f_j(\xi),
\quad f_j(\xi):=\varphi_j(\xi)f(\xi)\,\quad \xi\ne 0.
\]

\noindent\textbf{Step 2. Symbol bounds.}
Write $r:=4\pi^2|\xi|^2$ and
\[
M(r):=\frac{r^s\ln r}{(\lambda+r)^s\ln(\lambda+r)}\quad r>0,
\]
so that $m(\xi)=M(4\pi^2|\xi|^2)$ and $f(\xi)=M(4\pi^2|\xi|^2)-1$.

\smallskip
\noindent\textbf{(i) Low frequencies $j\le 0$, then $|\xi|\le 2$.}
For $0<r\le 1$ one has $M(r)=\frac{1}{\lambda^s\ln\lambda}\,r^s\ln r+O\bigl(r^{s+1}|\ln r|\bigr)$.
Differentiating yields: for each integer $k\ge1$,
\[|M^{(k)}(r)|\lesssim r^{s-k}\,(1+|\ln r|),\qquad 0<r\le1.\]
A standard chain-rule estimate for compositions with $|\xi|^2$ implies that for any multi-index $\alpha$ with $|\alpha|\ge 1,$
there exists $C_\alpha>0$ such that for $0<|\xi|\le 2$,
\[|\partial^\alpha m(\xi)|\le C_\alpha\,|\xi|^{2s-|\alpha|}\,(1+|\ln|\xi||).\]
Note that $m(0)=0$, so this bound is valid for $\alpha=0$ as well.
Defining $m_j(\xi):=\varphi_j(\xi)m(\xi)$, and using the product rule with $\|\partial^\beta \varphi_j\|_\infty\lesssim 2^{-j|\beta|}$,
for every $|\alpha|\le N$ and all $j\le 0$ we obtain
\begin{equation}\label{eq:low_derivative_mj}
\|\partial^\alpha m_j\|_\infty
\lesssim 2^{j(2s-|\alpha|)}(1+|j|).
\end{equation}

\smallskip
\noindent\textbf{(ii) High frequencies $j\ge 1$, then $|\xi|\ge 1$.}
As $r\to\infty$, $M(r)=1+O(r^{-1})$, and differentiating gives for each $k\ge1$,
$|M^{(k)}(r)|\lesssim r^{-k-1}$ for $r\ge1$.
Arguing again by the chain rule, for every $|\alpha|\le N$ and all $|\xi|\ge1$,
\[|\partial^\alpha f(\xi)|\lesssim |\xi|^{-2-|\alpha|}.\]
Hence, using $f_j(\xi)=\varphi_j(\xi)f(\xi)$, for $j\ge 1$ we obtain
\begin{equation}\label{eq:high_derivative_fj}
\|\partial^\alpha f_j\|_\infty \lesssim 2^{-j(2+|\alpha|)}.
\end{equation}

\medskip
\noindent\textbf{Step 3. Summation in $L^1(\R^n)$.}
Fix an integer $N>n$. Let $\cF^{-1}$ denote the inverse Fourier transform. 

For the high-frequency blocks $j\ge 1$, let $g_j := \cF^{-1}(f_j)$. By Lemma \ref{lem:bernstein_annulus} and \eqref{eq:high_derivative_fj},
\[
\|g_j\|_{L^1(\R^n)}
\lesssim \sum_{|\alpha|\le N} 2^{j|\alpha|}\,\|\partial^\alpha f_j\|_\infty
\lesssim \sum_{|\alpha|\le N} 2^{j|\alpha|}\,2^{-j(2+|\alpha|)}
\lesssim 2^{-2j}.
\]
Thus, the sum $\sum_{j\ge 1} g_j$ converges absolutely in $L^1(\R^n)$.

For the low-frequency part $f(\xi)\chi(\xi)$, since $f = m - 1$, we can write
\[
f(\xi)\chi(\xi) = m(\xi)\chi(\xi) - \chi(\xi) = \sum_{j\le 0} m_j(\xi) - \chi(\xi).
\]
Since $\chi \in C_c^\infty(\R^n)$, its inverse Fourier transform $\mathcal{F}^{-1}({\chi})$ is in the Schwartz class $\mathcal{S}(\R^n)$, which implies $\mathcal{F}^{-1}({\chi}) \in L^1(\R^n)$. 
For the terms $m_j$, let $h_j := \cF^{-1}(m_j)$. Using Lemma \ref{lem:bernstein_annulus} and \eqref{eq:low_derivative_mj} for $j\le 0$, we have
\[
\|h_j\|_{L^1(\R^n)}
\lesssim \sum_{|\alpha|\le N} 2^{j|\alpha|}\,\|\partial^\alpha m_j\|_\infty
\lesssim (1+|j|)\sum_{|\alpha|\le N} 2^{j|\alpha|}\,2^{j(2s-|\alpha|)}
\lesssim (1+|j|)\,2^{2sj}.
\]
Since $s>0$, the series $\sum_{j\le 0}(1+|j|)\,2^{2sj}$ is convergent, implying $\sum_{j\le 0} h_j$ converges absolutely in $L^1(\R^n)$.

Finally, we define
\[
g_{s,\lambda} := \sum_{j\le 0} h_j - \mathcal{F}^{-1}({\chi}) + \sum_{j\ge 1} g_j.
\]
By the absolute convergence of each component, $g_{s,\lambda} \in L^1(\R^n)$. 
By continuity of the Fourier transform $L^1\to L^\infty$,
we may interchange the Fourier transform and the summation, obtaining
\[
\widehat{g_{s,\lambda}}(\xi) = \sum_{j\le 0} m_j(\xi) - \chi(\xi) + \sum_{j\ge 1} f_j(\xi) = m(\xi)\chi(\xi) - \chi(\xi) + f(\xi)(1-\chi(\xi)) = f(\xi) = m(\xi)-1,
\]
which completes the proof.

\textbf{(ii).} Let $c_0:=\frac{2}{es}$, define
\[
  f(x):=g_{s,\lambda}(x)
          +c_0\,K_{s+\ln}^{\lambda}(x),
\]
where $g_{s,\lambda}$ is given in Proposition \ref{prop:bridge-measures} and $K_{s+\ln}^{\lambda}$ is the logarithmic Bessel kernel introduced before.  
Since $g\in L^1(\R^n)$ and $K_{s,\ln}^{\lambda}$ belong to
$L^1(\R^n)$, we have $f\in L^1(\R^n)$ and
its Fourier transform satisfies
\[
  \widehat f(\xi)+1
  =\frac{(4\pi^2|\xi|^2)^s\,\ln (4\pi^2|\xi|^2)+c_0}{(\lambda+4\pi^2|\xi|^2)^s\,\ln(\lambda+4\pi^2|\xi|^2)},
  \qquad \xi\in \R^n.
\]

Next we check that $\widehat f(\xi)+1$ vanishes nowhere on $\R^n$.
For $t>0$ and $s>0$, set $\phi(t):=t^{s}\ln t$. A direct computation gives
\[
  \phi'(t)=t^{s-1}(s\ln t+1).
\]
Hence $\phi'(t)=0$ if and only if $s\ln t+1=0$, i.e.
\[
  t_* = e^{-1/s}.
\]
Moreover, $\phi'(t)<0$ for $t\in(0,t_*)$ and $\phi'(t)>0$ for $t\in(t_*,\infty)$, so
$\phi$ decreases on $(0,t_*)$ and increases on $(t_*,\infty)$. Therefore $t_*$ is the
unique global minimizer of $\phi$ on $(0,\infty)$, and
\[
  \phi(t_*) = (e^{-1/s})^{s}\,\ln(e^{-1/s})
  = -\frac{1}{es}.
\]
Consequently,
\[
  t^{s}\ln t \ge -\frac{1}{es}\qquad \text{for all }t>0,
\]
and therefore
\[
 (4\pi^2|\xi|^2)^s\,\ln (4\pi^2|\xi|^2) + c_0
  \ge -\frac1{es} + \frac{2}{es}
  =\frac1{es}>0
\]
for every $\xi\in \R^n$. Consequently
\[
  \widehat f(\xi)+1
  >0\qquad\text{for all }\xi\in\R^n,
\]
so $\widehat f+1$ has no zeros.

We may now apply the $n$–dimensional version of Wiener's theorem: since $f\in L^1(\R^n)$ and
$\widehat f+1$ never vanishes, there exists a function
$\Phi_{s,\lambda}\in L^1(\R^n)$ such that (see \cite[Theorem F.1]{abatangelo2025gentle})
\begin{equation}\label{eq:Wiener-inverse}
  (\widehat f(\xi)+1)^{-1}
  = \widehat{\Phi_{s,\lambda}}(\xi)+1,
  \qquad \xi\in\R^n.
\end{equation}
Combining this identity with the explicit expression for $\widehat f+1$
we obtain
\[
  \frac{(\lambda+4\pi^2|\xi|^2)^s\,\ln(\lambda+4\pi^2|\xi|^2)}
       {(4\pi^2|\xi|^2)^s\,\ln (4\pi^2|\xi|^2)+c_0}
  = \widehat{\Phi_{s,\lambda}}(\xi)+1,
\]
which gives
\[
(\lambda+4\pi^2|\xi|^2)^s\,\ln(\lambda+4\pi^2|\xi|^2)
  = (4\pi^2|\xi|^2)^s\,\ln (4\pi^2|\xi|^2)\,\bigl(\widehat{\Phi_{s,\lambda}}(\xi)+1\bigr)
    + c_0\,\bigl(\widehat{\Phi_{s,\lambda}}(\xi)+1\bigr).
\]

Finally we set
\[
  w_{s,\lambda}:=\delta_0 + \Phi_{s,\lambda}(x)\,dx,
  \qquad
  \mu_{s,\lambda}:=c_0\,w_{s,\lambda}.
\]
Both $w_{s,\lambda}$ and $\mu_{s,\lambda}$ are finite Borel measures on
$\R^n$, and using $\widehat{\delta_0}\equiv1$ we have
\[
  \widehat{w_{s,\lambda}}(\xi)
  = \widehat{\Phi_{s,\lambda}}(\xi)+1,
  \qquad
  \widehat{\mu_{s,\lambda}}(\xi)
  = c_0\,\widehat{w_{s,\lambda}}(\xi).
\]
Thus the last identity can be rewritten as
\[
(\lambda+4\pi^2|\xi|^2)^s\,\ln(\lambda+4\pi^2|\xi|^2)
  = (4\pi^2|\xi|^2)^s\,\ln (4\pi^2|\xi|^2)\,\widehat{w_{s,\lambda}}(\xi)
    + \widehat{\mu_{s,\lambda}}(\xi),
\]
which is exactly the desired factorisation.
\end{proof}

The dyadic argument used in the proof of Proposition~\ref{prop:bridge-measures} in fact
yields the following more general $L^1$-kernel criterion, which we now establish.

\begin{proof}[\textbf{Proof of Proposition \ref{erjinzhi}.}]
Fix a smooth radial dyadic partition $\{\varphi_j\}_{j\in\Z}\subset C_c^\infty(\R^n)$ supported on annuli
$\{2^{j-1}\le|\xi|\le 2^{j+1}\}$ and satisfying $\sum_{j\in\Z}\varphi_j(\xi)=1$ for $\xi\neq0$.
Set $\chi(\xi):=\sum_{j\le 0}\varphi_j(\xi)$.
Define dyadic pieces
\[
f_j(\xi):=\varphi_j(\xi)f(\xi),\qquad j\in\Z.
\]
Let $g_j:=\cF^{-1}(f_j)$.

\medskip\noindent
\textbf{(i) High frequencies ($j\ge 1$).}
 By \eqref{eq:HF_assump} and the product rule,
for every $|\alpha|\le N$,
\[
\|\partial^\alpha f_j\|_{L^\infty}
\lesssim \sum_{\beta\le \alpha}\|\partial^\beta\varphi_j\|_\infty
\sup_{\xi\in\operatorname{supp}\varphi_j}|\partial^{\alpha-\beta}f(\xi)|
\lesssim \sum_{\beta\le \alpha}2^{-j|\beta|}\,2^{-j(\delta_1+|\alpha-\beta|)}
\lesssim 2^{-j(\delta_1+|\alpha|)}.
\]
By the standard annulus Bernstein estimate (Lemma~\ref{lem:bernstein_annulus}),
\[
\|g_j\|_{L^1(\R^n)}
\lesssim \sum_{|\alpha|\le N}2^{j|\alpha|}\,\|\partial^\alpha f_j\|_\infty
\lesssim \sum_{|\alpha|\le N}2^{j|\alpha|}\,2^{-j(\delta_1+|\alpha|)}
\lesssim 2^{-j\delta_1}.
\]
Hence $\sum_{j\ge 1}\|g_j\|_1<\infty$ since $\delta_1>0$.

\medskip\noindent
\textbf{(ii) Low frequencies ($j\le 0$).}
By \eqref{eq:LF_assump} and the product rule,
for every $|\alpha|\le N$,
\[
\|\partial^\alpha f_j\|_{L^\infty}
\lesssim \sum_{\beta\le \alpha}\|\partial^\beta\varphi_j\|_\infty
\sup_{\xi\in\operatorname{supp}\varphi_j}|\partial^{\alpha-\beta}f(\xi)|
\lesssim \sum_{\beta\le \alpha}2^{-j|\beta|}\,2^{j(\delta_2-|\alpha-\beta|)}(1+|j|)
\lesssim (1+|j|)\,2^{j(\delta_2-|\alpha|)}.
\]
Again by Lemma~\ref{lem:bernstein_annulus},
\[
\|g_j\|_{L^1(\R^n)}
\lesssim \sum_{|\alpha|\le N}2^{j|\alpha|}\,\|\partial^\alpha f_j\|_\infty
\lesssim (1+|j|)\sum_{|\alpha|\le N}2^{j|\alpha|}\,2^{j(\delta_2-|\alpha|)}
\lesssim (1+|j|)\,2^{j\delta_2}.
\]
Thus $\sum_{j\le 0}\|g_j\|_1<\infty$ since $\delta_2>0$.

\medskip
Define
\[
g:=\sum_{j\in\Z}g_j.
\]
The above bounds show that the series converges absolutely in $L^1(\R^n)$, so $g\in L^1(\R^n)$.
By continuity of the Fourier transform $L^1\to L^\infty$, we may interchange Fourier transform and
summation to obtain
\[
\widehat g(\xi)=\sum_{j\in\Z}\widehat{g_j}(\xi)=\sum_{j\in\Z}f_j(\xi)=f(\xi)
\quad\text{for }\xi\neq0,
\]
and the identity extends to $\xi=0$ after modifying $f$ on a null set if needed.
\end{proof}

\begin{remark}
The assumptions in Proposition \ref{erjinzhi} cannot be weakened by requiring
\eqref{eq:HF_assump}--\eqref{eq:LF_assump} only for $|\alpha|>0$.
Indeed, take $f(\xi)\equiv 1$. Then for every multi-index $\alpha$ with $|\alpha|>0$ we have
$\partial^\alpha f\equiv 0$, so the derivative bounds \eqref{eq:HF_assump}--\eqref{eq:LF_assump}
are trivially satisfied for any choice of $\delta_1,\delta_2>0$.
However, there is no $g\in L^1(\R^n)$ such that $\widehat g\equiv 1$.
In fact, the Fourier transform of an $L^1$ function belongs to $C_0(\R^n)$, hence it is continuous
and vanishes at infinity, whereas the constant function $1$ does not vanish at infinity.
Equivalently, the inverse Fourier transform of the constant function $1$ is the Dirac mass
$\delta_0$, which is not in $L^1(\R^n)$.
Therefore, controlling only derivatives does not exclude a nonzero constant component, and the
assumption with $\alpha=0$ is essential.
\end{remark}


\section{Regularity Theory and Space Embedding}

\subsection{Regularity Theory for Global Solutions}

In this subsection we develop an $L^p$ regularity theory for whole-space 
solutions driven by fractional--logarithmic operators. The main results are the
two a priori estimates stated in Theorems~\ref{thm:Lp-log-Bessel}
and~\ref{thm:Lp-log-Riesz}: the first establishes $L^p$ well-posedness for the
inhomogeneous equation $(\lambda I-\Delta)^{s+\ln}u=f$ and identifies the unique
solution as the logarithmic Bessel potential $u=K_{s+\ln}^{\lambda}*f$; the second
treats the homogeneous equation $(-\Delta)^{s+\ln}u=f$ and shows that solutions
belong to the same space $\mathcal L^{p}_{s+\ln,\lambda}$, with
a quantitative estimate obtained via the bridge between logarithmic Bessel and
logarithmic Riesz potentials.
\begin{proof}[\textbf{Proof of Theorem \ref{thm:Lp-log-Bessel}.}]
By definition of $K^{\lambda}_{s+\ln}$ and Proposition~\ref{lem:log-bessel-kernel}
we know that $K^{\lambda}_{s+\ln}\in L^1(\R^n)$ and that the Fourier
multiplier
\[
  m_{s+\ln}^\lambda(\xi)
  :=\frac{1}{(\lambda+4\pi^2|\xi|^2)^s\,\ln(\lambda+4\pi^2|\xi|^2)}.
\]
By Lemma \ref{lem:Fourier-conv-L1-OM},
\[
  \cF\bigl(K^{\lambda}_{s,\ln}*f\bigr)
  = m_{s+\ln}^\lambda(\xi)\,\widehat f(\xi)
  \quad \text{in}\quad\cS'(\R^n).
\]

Since $u$ is a distributional solution of \eqref{eq:lambda-log-eq},  we have
\[
(\lambda+4\pi^2|\xi|^2)^s\,\ln(\lambda+4\pi^2|\xi|^2)\,\widehat u(\xi)
  =\widehat f(\xi)
  \quad\text{in }\cS'(\R^n).
\]
Hence
\[
  \widehat u(\xi)
  = m_{s+\ln}^\lambda(\xi)\,\widehat f(\xi)
  = \cF\bigl(J^{\lambda}_{s,\ln}f\bigr)(\xi)\quad\text{in }\cS'(\R^n).
\]
so $u = J^{\lambda}_{s,\ln}f$ in $\cS'(\R^n)$,
and $u$ is unique.

Moreover,
\[
  \cF^{-1}\bigl((\lambda+4\pi^2|\xi|^2)^s\ln(\lambda+4\pi^2|\xi|^2)\,\widehat u\bigr)
  = \cF^{-1}\widehat f = f\in L^p(\R^n),
\]
so $u\in \mathcal{L}^{p}_{s+\ln,\lambda}(\R^n)$ by definition. Thus,
\[
  \|u\|_{\mathcal{L}^{p}_{s+\ln,\lambda}(\R^n)}
  = \|f\|_{L^p(\R^n)}.
\]
 This completes the proof.
\end{proof}

We now turn to the homogeneous equation \eqref{eq:log-Riesz-eq} and prove the corresponding
$L^p$ regularity estimate.

\begin{proof}[\textbf{Proof of Theorem \ref{thm:Lp-log-Riesz}.}]
By Proposition~\ref{prop:bridge-measures},
there exist finite Borel measures $\mu_{s,\lambda}$ and $w_{s,\lambda}$
on $\R^n$ such that
\[
  (\lambda + 4\pi^2|\xi|^2)^s \ln (\lambda + 4\pi^2|\xi|^2)
      =\widehat{\mu_{s,\lambda}}(\xi)+\widehat{w_{s,\lambda}}(\xi)\,(4\pi^2|\xi|^2)^s\,\ln (4\pi^2|\xi|^2),
      \qquad \xi\in\R^n.
\]
By Definition \ref{def:frac-log-lap} and  \eqref{eq:Fourier-measure-conv},  we obtain
\begin{equation} \label{eq:log-op-factor}
     (\lambda I-\Delta)^{s+\ln}u
  = \mu_{s,\lambda}*u + w_{s,\lambda}* (-\Delta)^{s+\ln}u
  \qquad\text{in }\cS_{\infty}'(\R^n),
\end{equation}
for all $u\in L^p(\mathbb{R}^n)\subset\cS_{\infty}'(\R^n)$.

Now suppose that $u,f\in L^p(\R^n)$ satisfy
\eqref{eq:log-Riesz-eq}. Inserting this into
\eqref{eq:log-op-factor} we find
\[
  (\lambda I-\Delta)^{s+\ln}u
  = \mu_{s,\lambda}*u + w_{s,\lambda}* f
  \qquad\text{in }\cS'(\R^n).
\]
Since $\mu_{s,\lambda}$ and $w_{s,\lambda}$ are finite measures,
convolution with them is bounded on $L^p(\R^n)$ for $1\le p\le\infty$.
Hence the right–hand side belongs to $L^p(\R^n)$ and we obtain
\[
  (\lambda I-\Delta)^{s+\ln}u \in L^p(\R^n),
\]
with
\[
  \|(\lambda I-\Delta)^{s+\ln}u\|_{L^p(\R^n)}
  \le C_1\bigl(\|f\|_{L^p(\R^n)}+\|u\|_{L^p(\R^n)}\bigr)
\]
for some constant $C_1=C_1(n,s,p,\lambda)>0$.

We may now apply Theorem~\ref{thm:Lp-log-Bessel} to the equation
\eqref{eq:lambda-log-eq} with the right–hand side
\[
  g := \mu_{s,\lambda}*u + w_{s,\lambda}* f \in L^p(\R^n).
\]
It follows that $u\in \mathcal{L}^{p}_{s+\ln,\lambda}(\R^n)$ and
\[
  \|u\|_{\mathcal{L}^{p}_{s+\ln,\lambda}(\R^n)}
 =\,\|g\|_{L^p(\R^n)}
  \le C\,\bigl(\|f\|_{L^p(\R^n)}+\|u\|_{L^p(\R^n)}\bigr),
\]
for some $C=C(n,s,p,\lambda)>0$, which is the desired estimate.
\end{proof}


\subsection{Comparison of Logarithmic Bessel Spaces}

In this subsection we compare several logarithmic function space scales. We first
use the dyadic $L^1$-kernel criterion in Proposition~\ref{erjinzhi} to analyze the
dependence of $\mathcal L^{p}_{s+\ln,\lambda}(\mathbb R^n)$ on the parameter $\lambda$
and to relate this logarithmic Bessel scale to the Bessel potential
spaces. We begin with the $\lambda$-equivalence result.

\begin{proof}[\textbf{Proof of Proposition \ref{prop:lambda-equivalence-wiener}.}]
Define
\[
m_\lambda(\xi)
:=
(\lambda+4\pi^2|\xi|^2)^s
\ln(\lambda+4\pi^2|\xi|^2),
\qquad \xi\in\R^n,
\]
and let $A_\lambda=(\lambda I-\Delta)^{s+\ln}$, so that
$\widehat{A_\lambda u}=m_\lambda\,\widehat u$ for $u\in\mathcal S(\R^n)$.

Fix $1<\lambda_1<\lambda_2$ and set
\[
\theta(\xi):=\frac{m_{\lambda_1}(\xi)}{m_{\lambda_2}(\xi)}.
\]
Since $\lambda_j>1$, one has $m_{\lambda_j}(\xi)>0$ for all $\xi$, hence
$\theta$ is smooth on $\R^n$ and $\theta(\xi)>0$. Moreover,
as $|\xi|\to\infty$,
\begin{equation}\label{eq:theta_infty}
\theta(\xi)=1+O(|\xi|^{-2}),
\end{equation}
while as $|\xi|\to0$,
writing $r:=4\pi^2|\xi|^2$ and using a Taylor expansion at $r=0$ yields
\begin{equation}\label{eq:theta_zero}
\theta(\xi)=\theta(0)+O(|\xi|^{2}),
\qquad
\theta(0)=:\,C_0=\frac{\lambda_1^s\ln\lambda_1}{\lambda_2^s\ln\lambda_2}.
\end{equation}

\medskip
\noindent\textbf{Step 1. Reduction to a symbol vanishing at $0$ and $\infty$.}
Let $\chi\in C_c^\infty(\R^n)$ be radial such that $\chi(\xi)=1$ for $|\xi|\le1$ and
$\operatorname{supp}\chi\subset\{|\xi|\le2\}$. Define
\[
f(\xi):=\theta(\xi)-1-(C_0-1)\chi(\xi).
\]
Then $f(\xi)=\theta(\xi)-C_0$ for $|\xi|\le1$, hence by \eqref{eq:theta_zero},
\begin{equation}\label{eq:f_low}
f(\xi)=O(|\xi|^{2})\qquad |\xi|\to0.
\end{equation}
On the other hand, for $|\xi|\ge2$ one has $\chi(\xi)=0$, hence by \eqref{eq:theta_infty},
\begin{equation}\label{eq:f_high}
f(\xi)=\theta(\xi)-1=O(|\xi|^{-2})\qquad |\xi|\to\infty.
\end{equation}

\medskip
\noindent\textbf{Step 2. $f\in \mathcal F(L^1)$.}
We claim that $f\in\mathcal F(L^1(\R^n))$.
Indeed, $f$ is smooth on $\R^n$ and, by differentiating \eqref{eq:f_low}--\eqref{eq:f_high}
(and using that $\theta$ is a smooth function of $|\xi|^2$),
for every multi-index $\alpha$ with $|\alpha|\le N$ (for some fixed $N>n$) one has the bounds
\[
|\partial^\alpha f(\xi)|\lesssim |\xi|^{2-|\alpha|},\quad 0<|\xi|\le1,
\qquad
|\partial^\alpha f(\xi)|\lesssim |\xi|^{-2-|\alpha|},\quad |\xi|\ge1.
\]
Thus $f$ satisfies the hypotheses of Proposition ~\ref{erjinzhi} with
$\delta_1=\delta_2=2$, and we obtain the existence of $K\in L^1(\R^n)$ such that
\begin{equation}\label{eq:f_is_Fourier_L1}
\widehat K(\xi)=f(\xi)
=\theta(\xi)-1-(C_0-1)\chi(\xi).
\end{equation}

Since $\mathcal{F}^{-1}(\chi)\in\mathcal S(\R^n)\subset L^1(\R^n)$, define
\[
H:=K+(C_0-1)\mathcal{F}^{-1}(\chi)\in L^1(\R^n).
\]
Taking inverse Fourier transforms in \eqref{eq:f_is_Fourier_L1} gives
\[
\mathcal F^{-1}\theta=\delta_0+H,
\qquad H\in L^1(\R^n).
\]
Consequently, the Fourier multiplier
\[
Tf:=\mathcal F^{-1}(\theta\,\widehat f)
\]
can be written as
\[Tf=(\delta_0+H)*f=f+H*f.\]
By Young's inequality,
\begin{equation}\label{eq:T_bound}
\|Tf\|_{L^p}\le (1+\|H\|_{L^1})\|f\|_{L^p},
\qquad 1\le p\le\infty.
\end{equation}

\medskip
\noindent\textbf{Step 3. Invertibility of $T$ on $L^p$.}
Since $\theta(\xi)>0$ and $\theta(\xi)\to1$ as $|\xi|\to\infty$, it follows that
$\inf_{\R^n}\theta>0$, hence $\theta^{-1}$ is smooth and bounded.
Repeating Steps~1--2 with $\theta^{-1}$ in place of $\theta$, we find
$H_1\in L^1(\R^n)$ such that
\[
\mathcal F^{-1}(\theta^{-1})=\delta_0+H_1.
\]
Equivalently, the inverse multiplier $T^{-1}$ satisfies
\[
T^{-1}f=(\delta_0+H_1)*f=f+H_1*f,
\]
and therefore
\begin{equation}\label{eq:Tinv_bound}
\|T^{-1}f\|_{L^p}\le (1+\|H_1\|_{L^1})\|f\|_{L^p},
\qquad 1\le p\le\infty.
\end{equation}
In particular, $T$ is a bounded isomorphism on $L^p(\R^n)$ for every $1\le p\le\infty$.

\medskip
\noindent\textbf{Step 4. Equivalence of norms.}
For $u\in\mathcal S(\R^n)$ we have
\[
\widehat{A_{\lambda_1}u}
=m_{\lambda_1}\widehat u
=\theta\,m_{\lambda_2}\widehat u
\quad\Longrightarrow\quad
A_{\lambda_1}u=T(A_{\lambda_2}u).
\]
Hence, by \eqref{eq:T_bound},
\[
\|u\|_{\mathcal L^p_{s+\ln,\lambda_1}}
=\|A_{\lambda_1}u\|_{L^p}
=\|T(A_{\lambda_2}u)\|_{L^p}
\le (1+\|H\|_{L^1})\,\|A_{\lambda_2}u\|_{L^p}
=(1+\|H\|_{L^1})\,\|u\|_{\mathcal L^p_{s+\ln,\lambda_2}}.
\]
The reverse inequality follows by symmetry, or by applying \eqref{eq:Tinv_bound} to
$A_{\lambda_2}u=T^{-1}(A_{\lambda_1}u)$:
\[
\|u\|_{\mathcal L^p_{s+\ln,\lambda_2}}
=\|A_{\lambda_2}u\|_{L^p}
\le (1+\|H_1\|_{L^1})\,\|A_{\lambda_1}u\|_{L^p}
=(1+\|H_1\|_{L^1})\,\|u\|_{\mathcal L^p_{s+\ln,\lambda_1}}.
\]
Taking $C:=\max\{1+\|H\|_{L^1},\,1+\|H_1\|_{L^1}\}$ completes the proof.
\end{proof}

We next show that, although different parameters $\lambda>1$ yield the same space up to
equivalent norms, the logarithmic scale is strictly smaller than the borderline case
$\lambda=1$.

\begin{proof}[\textbf{Proof of Proposition \ref{prop:lambda-inclusion}.}]
Set
\[
m_\lambda(\xi)
:=
(\lambda+4\pi^2|\xi|^2)^s\,
\ln(\lambda+4\pi^2|\xi|^2),
\qquad
m_1(\xi)
:=
(1+4\pi^2|\xi|^2)^s\,
\ln(1+4\pi^2|\xi|^2),
\]
and denote $A_\lambda=(\lambda I-\Delta)^{s+\ln}$, $A_1=(I-\Delta)^{s+\ln}$.
Consider the ratio
\[
\psi(\xi):=\frac{m_1(\xi)}{m_\lambda(\xi)}.
\]

Note that $\psi(\xi)\to 0$ as $|\xi|\to 0$ and $\psi(\xi)\to 1$ as $|\xi|\to\infty$.
Fix a radial cutoff $\chi\in C_c^\infty(\R^n)$ such that $\chi(\xi)=1$ for $|\xi|\le 1$ and
$\operatorname{supp}\chi\subset\{|\xi|\le 2\}$, and define
\[
f(\xi):=\psi(\xi)-1+\chi(\xi).
\]
Then $f(\xi)=\psi(\xi)$ for $|\xi|\le 1$, hence $f(\xi)=O(|\xi|^2)$ as $|\xi|\to 0$,
while $f(\xi)=\psi(\xi)-1=O(|\xi|^{-2})$ for $|\xi|\ge 2$.
Moreover, by differentiating these expansions (and using that $\psi$ is a smooth function of
$|\xi|^2$), one obtains for some integer $N>n$ the bounds
\[
|\partial^\alpha f(\xi)|\lesssim |\xi|^{2-|\alpha|},
\quad 0<|\xi|\le 1,\qquad
|\partial^\alpha f(\xi)|\lesssim |\xi|^{-2-|\alpha|},
\quad |\xi|\ge 1,\qquad |\alpha|\le N.
\]
Therefore, Proposition ~\ref{erjinzhi} applies (with $\delta_1=\delta_2=2$) and yields
a function $K\in L^1(\R^n)$ such that
\[
\widehat K(\xi)=f(\xi)=\psi(\xi)-1+\chi(\xi).
\]
Since $\mathcal{F}^{-1}(\chi)\in\mathcal S(\R^n)\subset L^1(\R^n)$, we may rewrite
\[
\psi(\xi)=1+\widehat K(\xi)-\chi(\xi)=1+\widehat{K}(\xi)-\chi(\xi).
\]
Setting $H:=K-\mathcal{F}^{-1}(\chi)\in L^1(\R^n)$, we obtain
\[
\mathcal F^{-1}\psi=\delta_0+H.
\]

In particular, the Fourier multiplier
\[
Tf:=\mathcal F^{-1}(\psi\,\widehat f)
\]
can be written as a convolution operator
\[
Tf=(\delta_0+H)*f=f+H*f.
\]
By Young's inequality, $T$ is bounded on $L^p(\R^n)$ for every $1\le p\le\infty$, and
\[
\|Tf\|_{L^p}\le (1+\|H\|_{L^1})\,\|f\|_{L^p}.
\]

Now let $u\in \mathcal{L}^{p}_{s+\ln,\lambda}(\R^n)$, i.e.\ $A_\lambda u\in L^p(\R^n)$.
Since
\[
\widehat{A_1 u}
=m_1\,\widehat u
=\psi\, m_\lambda\,\widehat u
=\psi\,\widehat{A_\lambda u},
\]
we have
\[
A_1 u=T(A_\lambda u)=(\delta_0+H)*(A_\lambda u)\in L^p(\R^n).
\]
Therefore $u\in \mathcal{L}^{p}_{s+\ln,1}(\R^n)$, which proves the inclusion.

To see that the inclusion is strict for $1\le p<\infty$, take the constant function $u\equiv1$.
Then $\widehat u=\delta_0$, so
\[
m_1\,\widehat u=0,
\]
hence
\[
A_1 u\equiv0\in L^p(\R^n).
\]
Thus $u\in \mathcal{L}^{p}_{s+\ln,1}(\R^n)$.

On the other hand,
\[
A_\lambda u
=
m_\lambda(0)\cdot 1
=
\lambda^s\ln\lambda,
\]
which is a nonzero constant.
Since constant functions do not belong to $L^p(\R^n)$ for $1\le p<\infty$,
we conclude that
\[
u\notin \mathcal{L}^{p}_{s+\ln,\lambda}(\R^n).
\]

For the case $p=\infty$, we consider the quadratic polynomial $u(x) = |x|^2$. Clearly, $u$ belongs to the space of tempered distributions $\cS'(\R^n)$, and the Fourier transform of $u$ in the sense of distributions is given by$$\widehat{u}(\xi) = -\frac{1}{4\pi^2} \Delta \delta_0,$$where $\delta_0$ is the Dirac delta distribution at the origin. We first show that $u \in \mathcal{L}^{\infty}_{s+\ln,1}(\R^n)$. By definition, we need to compute $A_1 u = \cF^{-1}(m_1 \widehat{u})$. For any  $\phi \in \cS(\R^n)$, we have
$$\langle m_1 \Delta \delta_0, \phi \rangle 
= \langle \Delta \delta_0, m_1 \phi \rangle 
= \Delta(m_1 \phi)(0).$$
Applying the product rule for the Laplacian, we get
$$\Delta(m_1 \phi)(0) 
= (\Delta m_1)(0)\phi(0) + 2\nabla m_1(0)\cdot\nabla\phi(0) + m_1(0)\Delta\phi(0).$$
Recall that $m_1(\xi) = (1+4\pi^2|\xi|^2)^s \ln(1+4\pi^2|\xi|^2)$. It is immediate that $m_1(0)= 0$. Since $m_1$ is a smooth radial function depending on $|\xi|^2$, its gradient vanishes at the origin, i.e., $\nabla m_1(0) = 0$. Consequently, the expression simplifies to
$$\Delta(m_1 \phi)(0) = (\Delta m_1)(0)\phi(0).$$
In the sense of distributions, this means $m_1 \widehat{u} = -\frac{1}{4\pi^2}(\Delta m_1)(0)\delta_0$. Taking the inverse Fourier transform yields
$$A_1 u(x) = -\frac{1}{4\pi^2}(\Delta m_1)(0),$$
which is a constant. Since constant functions are bounded, $A_1 u \in L^\infty(\R^n)$, and thus $u \in \mathcal{L}^{\infty}_{s+\ln,1}(\R^n)$.

Next, we show that $u \notin \mathcal{L}^{\infty}_{s+\ln,\lambda}(\R^n)$. We repeat the same procedure for the symbol $m_\lambda(\xi) = (\lambda+4\pi^2|\xi|^2)^s \ln(\lambda+4\pi^2|\xi|^2)$. Evaluating at the origin yields
$$m_\lambda(0) = \lambda^s \ln \lambda.$$
Since $\lambda > 1$ and $s > 0$, we have $m_\lambda(0) > 0$. As before, $\nabla m_\lambda(0) = 0$. For a test function $\phi \in \cS(\R^n)$, we now obtain
$$\Delta(m_\lambda \phi)(0) 
= (\Delta m_\lambda)(0)\phi(0) + m_\lambda(0)\Delta\phi(0).$$
This implies that in $\cS'(\R^n)$,$$m_\lambda \widehat{u} 
= -\frac{1}{4\pi^2}(\Delta m_\lambda)(0)\delta_0 - \frac{1}{4\pi^2} m_\lambda(0) \Delta \delta_0.$$
Applying the inverse Fourier transform, we find
$$A_\lambda u(x) = -\frac{1}{4\pi^2}(\Delta m_\lambda)(0) + m_\lambda(0) |x|^2.$$
Because $m_\lambda(0) = \lambda^s \ln \lambda \neq 0$, $A_\lambda u(x)$ exhibits quadratic growth as $|x| \to \infty$. Therefore, $A_\lambda u \notin L^\infty(\R^n)$, which implies $u \notin \mathcal{L}^{\infty}_{s+\ln,\lambda}(\R^n)$. This completes the proof that the inclusion is strict.

This proves that the inclusion is strict for $1\le p\le \infty$.
\end{proof}

We now compare the logarithmic Bessel  space $\mathcal{L}^{p}_{s+\ln,\lambda}(\R^n)$ with the
classical Bessel potential space $\mathcal{L}^{p}_{2s}(\R^n)$.
A first idea is to relate their defining multipliers via the ratio
\[
\theta_1(\xi):=\frac{(1+4\pi^2|\xi|^2)^s}{(\lambda+4\pi^2|\xi|^2)^s\ln(\lambda+4\pi^2|\xi|^2)},
\]
so that $(I-\Delta)^s = \cF^{-1}(\theta_1\,\widehat{A_\lambda(\,\cdot\,)})$ at the formal level.
If one could show $\theta_1\in\mathcal{F}(L^1(\R^n))$, then $\cF^{-1}\theta_1$ would be an $L^1$ kernel
and the embedding $\mathcal{L}^{p}_{s+\ln,\lambda}(\R^n)\hookrightarrow \mathcal{L}^{p}_{2s}(\R^n)$
would follow immediately from Young's inequality.
However, Proposition ~\ref{erjinzhi} cannot be applied directly to $\theta_1$:
indeed, $\theta_1$ decays only logarithmically at infinity,
\[
\theta_1(\xi)\sim \frac{1}{\ln|\xi|^2}\qquad |\xi|\to\infty,
\]
and the dyadic summation argument underlying Proposition ~\ref{erjinzhi} would lead to a
non-summable harmonic-type series at high frequencies.
To overcome this obstacle, we exploit a structural representation of
\(
(\ln(\lambda+4\pi^2|\xi|^2))^{-1}
\)
as a superposition of Bessel multipliers, which yields an $L^1$ kernel by positivity and
Tonelli's theorem. This is the key step in the proof below.

\begin{proof}[\textbf{Proof of Proposition \ref{prop:Hsln-vs-Hs}.}]
Set
\[
M(\xi):=(\lambda+4\pi^2|\xi|^2)^s\ln(\lambda+4\pi^2|\xi|^2),
\qquad
Au:=\cF^{-1}(M(\xi)\widehat u).
\]
For large $|\xi|$ one has $M(\xi)\sim (4\pi^2|\xi|^2)^s\ln(4\pi^2|\xi|^2)$.
Using $\ln(1+t)\le C_\varepsilon\,t^{\varepsilon/2}$ for $t\ge1$, we obtain for each $\varepsilon>0$
constants $c_1,c_2>0$ such that
\begin{equation}\label{eq:M_two_sided}
c_1(1+|\xi|^2)^s \le M(\xi)\le c_2(1+|\xi|^2)^{s+\varepsilon/2},
\qquad \xi\in\R^n.
\end{equation}
Since $\lambda>1$, we also have $M(\xi)\to \lambda^s\ln\lambda>0$ as $\xi\to0$.

Define the symbol ratios associated with the embeddings:
\[
\theta_1(\xi):=\frac{(1+4\pi^2|\xi|^2)^s}{M(\xi)},
\qquad
\theta_2(\xi):=\frac{M(\xi)}{(1+4\pi^2|\xi|^2)^{s+\varepsilon/2}}.
\]
Both $\theta_1$ and $\theta_2$ are smooth bounded functions on $\R^n$. To prove the embeddings, it suffices to show that the inverse Fourier transforms $\cF^{-1}(\theta_1)$ and $\cF^{-1}(\theta_2)$ belong to $L^1(\R^n)$, so that the corresponding convolution operators are bounded on $L^p(\R^n)$ by Young's inequality.

\medskip
\noindent\textbf{Step 1. The $L^1$ integrability of $\cF^{-1}(\theta_2)$.}
Notice that as $|\xi|\to\infty$, $\theta_2(\xi) = O(|\xi|^{-\varepsilon}\ln|\xi|)$.
Let $\chi\in C_c^\infty(\R^n)$ be a smooth radial cutoff such that $\chi(\xi)=1$ for $|\xi|\le1$ and $\operatorname{supp}\chi\subset\{|\xi|\le2\}$. Define
\[
f_2(\xi) := \theta_2(\xi) - \theta_2(0)\chi(\xi).
\]
By construction, $f_2(0)=0$ and $f_2(\xi)$ inherits the smoothness of $\theta_2$. Differentiating the asymptotic expansion shows that for any integer $N>n$ and $|\alpha|\le N$,
\[
|\partial^\alpha f_2(\xi)|\lesssim |\xi|^{2-|\alpha|}\bigl(1+|\ln|\xi||\bigr) \quad \text{for } 0<|\xi|\le 2,
\]
and
\[
|\partial^\alpha f_2(\xi)|\lesssim |\xi|^{-\varepsilon-|\alpha|}\,(1+\ln|\xi|) \quad \text{for } |\xi|\ge 1.
\]
This precisely satisfies the conditions of Proposition ~\ref{erjinzhi} (with $\delta_1 = \varepsilon/2 > 0$). Therefore, $\cF^{-1}(f_2) \in L^1(\R^n)$. Since $\mathcal{F}^{-1}(\chi)\in \mathcal{S}(\R^n) \subset L^1(\R^n)$, it follows that $\cF^{-1}(\theta_2) = \cF^{-1}(f_2) + \theta_2(0)\mathcal{F}^{-1}(\chi)\in L^1(\R^n)$.

\medskip
\noindent\textbf{Step 2. The $L^1$ integrability of $\cF^{-1}(\theta_1)$.}
The symbol $\theta_1$ only decays logarithmically, $\theta_1(\xi) = O((\ln|\xi|)^{-1})$ as $|\xi|\to\infty$. We split $\theta_1$ into a main part and an error term:
\[
\theta_1(\xi) = \frac{1}{\ln(\lambda+4\pi^2|\xi|^2)} + \frac{(1+4\pi^2|\xi|^2)^s - (\lambda+4\pi^2|\xi|^2)^s}{(\lambda+4\pi^2|\xi|^2)^s \ln(\lambda+4\pi^2|\xi|^2)}=: m_{\text{main}}(\xi)+m_{\text{err}}(\xi).
\]

For the main term, we use the Laplace transform identity $z^{-1} = \int_0^\infty e^{-pz} \,dp$ for $z>0$ to write
\[
m_{\text{main}}(\xi) = \int_0^\infty e^{-p \ln(\lambda+4\pi^2|\xi|^2)} \, dp = \int_0^\infty (\lambda+4\pi^2|\xi|^2)^{-p} \, dp=  \int_0^\infty \widehat{G^\lambda_{p}}(\xi) \, dp,
\]
where $ \widehat{G^\lambda_{p}}(\xi)$ is defined in $\eqref{eq:def-shifted-bessel}$. In particular, $G^\lambda_{p}$ is a positive, radial,
integrable function with
\[
\|G^\lambda_{p}\|_{L^1(\R^n)}  = \lambda^{-p}.
\]
Since $\lambda>1$, we can integrate these $L^1$ norms over $p \in (0,\infty)$:
\[
\int_0^\infty \|G^\lambda_{p}\|_{L^1(\R^n)} \, dp = \int_0^\infty \lambda^{-p} \, dp = \frac{1}{\ln\lambda} < \infty.
\]
By Bochner's theory for Lebesgue integration, the finiteness of this integral guarantees that the superposition
\[
K_{\text{main}}(x) := \int_0^\infty G^\lambda_{p}(x) \, dp
\]
converges absolutely in $L^1(\R^n)$, which rigorously implies $\cF^{-1}(m_{\text{main}}) = K_{\text{main}} \in L^1(\R^n)$.

For the error term $m_{\text{err}}(\xi)$, setting $r = 4\pi^2|\xi|^2$, a Taylor expansion for $r \to \infty$ gives
\[
(1+r)^s - (\lambda+r)^s = s(1-\lambda)r^{s-1} + O(r^{s-2}).
\]
Dividing by $(\lambda+r)^s \ln(\lambda+r) \sim r^s \ln r$, we obtain
\[
m_{\text{err}}(\xi) = O\Big(\frac{r^{s-1}}{r^s\ln r}\Big) = O\Big(\frac{1}{r\ln r}\Big) = O\big(|\xi|^{-2}(\ln|\xi|)^{-1}\big) \quad \text{as } |\xi| \to \infty.
\]

Defining $f_1(\xi) := m_{\text{err}}(\xi) - m_{\text{err}}(0)\chi(\xi)$, the decay implies $|\partial^\alpha f_1(\xi)| \lesssim |\xi|^{-2-|\alpha|}$ for $|\xi|\ge 1$. Moreover, since $m_{\mathrm{err}}$ is a smooth radial function of $r=4\pi^2|\xi|^2$,
we have the Taylor expansion at $\xi=0$
\[
m_{\mathrm{err}}(\xi)=m_{\mathrm{err}}(0)+O(|\xi|^2)\qquad |\xi|\to0,
\]
and hence, for every multi-index $\alpha$ with $|\alpha|\le N$,
\[
|\partial^\alpha f_1(\xi)|
=|\partial^\alpha\big(m_{\mathrm{err}}(\xi)-m_{\mathrm{err}}(0)\chi(\xi)\big)|
\lesssim |\xi|^{2-|\alpha|}\bigl(1+|\ln|\xi||\bigr),
\qquad 0<|\xi|\le 1.
\]
Applying Proposition ~\ref{erjinzhi} (with $\delta_1 = 2$), we conclude that $\cF^{-1}(f_1) \in L^1(\R^n)$, and thus $\cF^{-1}(m_{\text{err}}) \in L^1(\R^n)$. Combining both parts yields $\cF^{-1}(\theta_1) \in L^1(\R^n)$.

\medskip
\noindent\textbf{Step 3. Norm comparison.}
Let $T_1 u := \cF^{-1}(\theta_1 \widehat{u})$ and $T_2 u := \cF^{-1}(\theta_2 \widehat{u})$. By Steps 1 and 2, $T_1$ and $T_2$ are convolution operators with $L^1(\R^n)$ kernels, thus bounded on $L^p(\R^n)$ for all $1\le p\le\infty$. 

For $u\in\mathcal S(\R^n)$ we have
\[
\widehat{(I-\Delta)^s u}=(1+4\pi^2|\xi|^2)^s\widehat u=\theta_1(\xi)\,M(\xi)\widehat u
\quad\Longrightarrow\quad
(I-\Delta)^s u=T_1(Au),
\]
and similarly
\[
\widehat{Au}=M(\xi)\widehat u=\theta_2(\xi)\,(1+4\pi^2|\xi|^2)^{s+\varepsilon/2}\widehat u
\quad\Longrightarrow\quad
Au=T_2\big((I-\Delta)^{s+\varepsilon/2}u\big).
\]
Taking $L^p$ norms gives
\[
\|(I-\Delta)^s u\|_{L^p}\lesssim \|Au\|_{L^p}\lesssim \|(I-\Delta)^{s+\varepsilon/2}u\|_{L^p},
\]
which implies
\[
\|u\|_{\mathcal L^p_{2s}}\le C\,\|u\|_{\mathcal L^p_{s+\ln,\lambda}}
\le C\,\|u\|_{\mathcal L^p_{2s+\varepsilon}}.
\]
Finally, since $(1+4\pi^2|\xi|^2)^{-s}\in\mathcal F(L^1)$, the map
$u\mapsto \cF^{-1}((1+4\pi^2|\xi|^2)^{-s}\widehat u)$ is a convolution with an $L^1$ kernel, yielding
$\|u\|_{L^p}\lesssim \|(I-\Delta)^s u\|_{L^p}$, meaning $\mathcal L^p_{2s}\hookrightarrow L^p$.
This completes the proof.
\end{proof}

We finally identify our logarithmic Bessel scale and a
classical family of logarithmic Bessel potential spaces introduced in the
literature; see \cite{opic2000bessel}.

\begin{proof}[\textbf{Proof of Proposition \ref{prop:space-equivalence-literature}.}]
Let the symbols of the respective operators be
\[
M(\xi) := (\lambda+4\pi^2|\xi|^2)^s\ln(\lambda+4\pi^2|\xi|^2), \qquad N(\xi) := (1+4\pi^2|\xi|^2)^s\bigl(1+\ln(1+4\pi^2|\xi|^2)\bigr).
\]
To prove the norm equivalence, it suffices to show that the multiplier ratios
\[
\theta_1(\xi) := \frac{M(\xi)}{N(\xi)} \quad \text{and} \quad \theta_2(\xi) := \frac{N(\xi)}{M(\xi)}
\]
are of the form $c_j + \widehat{K_j}(\xi)$, where $c_j$ are constants and $K_j \in L^1(\R^n)$. If so, the operators $u \mapsto \cF^{-1}(\theta_j \widehat{u})$ are bounded on $L^p(\R^n)$ for all $1 \le p \le \infty$ by Young's convolution inequality.

We observe that both $\theta_1$ and $\theta_2$ tend to $1$ as $|\xi| \to \infty$. Writing $r = 4\pi^2|\xi|^2$, a careful asymptotic expansion as $r \to \infty$ reveals that the difference between the symbols is dominated by logarithmic decay:
\[
\theta_1(\xi) - 1 \sim O(\frac{1}{\ln r}), \qquad \theta_2(\xi) - 1 \sim O(\frac{1}{\ln r}).
\]
Because the purely logarithmic decay $O((\ln|\xi|)^{-1})$ is too slow to apply Proposition ~\ref{erjinzhi} directly, we systematically decompose each ratio into a completely monotonic slow part and an accelerated fast-decaying error part.

\medskip
\noindent\textbf{Step 1. Analysis of $\theta_2(\xi)$.}
We split $\theta_2(\xi)$ as follows to match the exact fractions over $M(\xi)$:
\[
\theta_2(\xi) = 1 + \frac{1}{\ln(\lambda+4\pi^2|\xi|^2)} + \frac{N(\xi) - M(\xi) - (\lambda+4\pi^2|\xi|^2)^s}{M(\xi)}=:1+ m_{\text{slow},2}(\xi)+m_{\text{err},2}(\xi).
\]
For the slow part $m_{\text{slow},2}(\xi)$, by the proof of Proposition~\ref{prop:Hsln-vs-Hs}, we know that $\mathcal{F}^{-1}( m_{\text{slow},2})\in L^1(\R^n)$.
For the error part $m_{\text{err},2}(\xi)$, we perform a Taylor expansion for large $r = 4\pi^2|\xi|^2$:
\begin{align*}
N(r) - M(r) - (\lambda+r)^s &= (1+r)^s \bigl(1+\ln(1+r)\bigr) - (\lambda+r)^s \ln(\lambda+r) - (\lambda+r)^s \\
&=s(1-\lambda) r^{s-1}\ln r + o(r^{s-1}\ln r).
\end{align*}
Dividing this by $M(r) \sim r^s \ln r$, we obtain for large $|\xi|$:
\[
m_{\text{err},2}(\xi)= O(r^{-1}) = O(|\xi|^{-2}).
\]
To apply Proposition ~\ref{erjinzhi}, we define $f_2(\xi) := m_{\text{err},2}(\xi) - m_{\text{err},2}(0)\chi(\xi)$ with a smooth low-frequency cutoff $\chi(\xi)$ ($\chi=1$ for $|\xi|\le 1$, supported in $|\xi|\le 2$). 
For the low frequencies ($|\xi| \to 0$), since $\lambda>1$, $M(\xi)$ and $N(\xi)$ are strictly positive and smooth near the origin. Thus, $m_{\text{err},2}(\xi)$ is smooth. Because $f_2(0)=0$ by construction, a standard expansion yields $f_2(\xi) = O(|\xi|^2)$, which guarantees $|\partial^\alpha f_2(\xi)| \lesssim |\xi|^{2-|\alpha|}$ for $0 < |\xi| \le 2$. This perfectly satisfies the low-frequency condition \eqref{eq:LF_assump} with $\delta_2 = 2$.
Simultaneously, the high-frequency bound gives $|\partial^\alpha f_2(\xi)| \lesssim |\xi|^{-2-|\alpha|}$ for $|\xi|\ge 1$ (\eqref{eq:HF_assump} with $\delta_1 = 2$). Proposition ~\ref{erjinzhi} thus yields $\cF^{-1}(m_{\text{err},2}) \in L^1(\R^n)$.

\medskip
\noindent\textbf{Step 2. Analysis of $\theta_1(\xi)$.}
Similarly, taking $N(\xi)$ as the common denominator, we decompose $\theta_1(\xi)$:
\[
\theta_1(\xi) =  1 - \frac{1}{1+\ln(1+4\pi^2|\xi|^2)} + \frac{M(\xi) - N(\xi) + (1+4\pi^2|\xi|^2)^s}{N(\xi)}=:1- m_{\text{slow},1}(\xi)+m_{\text{err},1}(\xi).
\]

For the slow part, we have the Laplace representation
\[
\frac{1}{1+\ln(1+r)}
=
\frac{1}{\ln(e(1+r))}
=
\int_0^\infty (e(1+r))^{-p}\,dp
=
\int_0^\infty e^{-p}(1+r)^{-p}\,dp,\quad r=4\pi^2 |\xi|^2.
\]
For each $p>0$, we recall that $G_p=\cF^{-1}((1+r)^{-p})$, the Bessel kernel. It is well known that $G_p>0$
and $G_p\in L^1(\R^n)$ with $\|G_p\|_{L^1}=1$.
Therefore
\[
\int_0^\infty \|e^{-p}G_p\|_{L^1}\,dp=\int_0^\infty e^{-p}\,dp<\infty,
\]
so the Bochner integral $K_{\rm slow}:=\int_0^\infty e^{-p}G_p\,dp$ converges in $L^1(\R^n)$.
By continuity of the Fourier transform $L^1\to C_0$, we may interchange Fourier transform and
integration to obtain
\[
\widehat{K_{\rm slow}}(\xi)=\int_0^\infty e^{-p}(1+r)^{-p}\,dp=\frac{1}{1+\ln(1+r)}.
\]

For the error part $m_{\text{err},1}(\xi)$, the same Taylor expansion gives 
\[
m_{\text{err},1}(\xi) = O(r^{-1}) = O(|\xi|^{-2}).
\]
By the same reasoning at the origin (since $N(0)=1>0$), the shifted function $f_1(\xi) := m_{\text{err},1}(\xi) - m_{\text{err},1}(0)\chi(\xi)$ satisfies the low-frequency condition with $\delta_2 = 2$ and the high-frequency condition with $\delta_1 = 2$. Applying Proposition ~\ref{erjinzhi} guarantees $\cF^{-1}(m_{\text{err},1}) \in L^1(\R^n)$. 

Combining these steps, both operators $\cF^{-1}(\theta_1 \widehat{\cdot})$ and $\cF^{-1}(\theta_2 \widehat{\cdot})$ are convolution operators with $L^1$ kernels (plus an identity operator). Young's inequality then secures the uniform bound on $L^p(\R^n)$, completing the proof.
\end{proof}


\subsection{Critical Embedding and Compactness Theory}

As a direct consequence of Proposition~\ref{prop:space-equivalence-literature}, we can now address the critical Sobolev embedding. It is a well-known that the classical Bessel potential space $\mathcal{L}^{p}_{2s}(\R^n)$ fails to embed into $L^\infty(\R^n)$ in the critical case $s = \frac{n}{2p}$ for $1< p<\infty$. However, the logarithmic correction intrinsic to our space recovers this embedding, and in fact, it yields continuity. Moreover, it naturally extends to higher-order critical cases.

\begin{proof}[\textbf{Proof of Theorem \ref{cor:critical_embeddings}.}]
By Proposition~\ref{prop:space-equivalence-literature}, we have the space identification $\mathcal{L}^{p}_{s+\ln,\lambda}(\R^n) = L^p_{2s,1}(\R^n)$ with equivalent norms. Throughout the proof, the intrinsic logarithmic exponent of our space is $\alpha = 1$, which strictly satisfies $\alpha > \frac{1}{p'}$ since $\frac{1}{p'} = 1 - \frac{1}{p} < 1$ for all $1 < p < \infty$.

(i) For $s = \frac{n}{2p}$, the smoothness index corresponds precisely to the critical Sobolev case $2s = \frac{n}{p}$. According to the embedding theorem by Opic and Trebels \cite[Theorem 1.2(a)]{opic2000bessel}, the continuous embedding $L^p_{n/p,\alpha}(\R^n) \hookrightarrow C_0(\R^n)$ holds provided $\alpha > \frac{1}{p'}$. Since $\alpha = 1$, this condition is unconditionally satisfied, yielding the first embedding.

(ii) For $s = \frac{1}{2}\bigl(\frac{n}{p} + m\bigr)$, the smoothness index is $2s = \frac{n}{p} + m$. By the standard properties of logarithmic Bessel potential spaces \cite[Eq.(1.9)]{opic2000bessel}, $u \in L^p_{n/p+m, 1}(\R^n)$ implies that for any multi-index $\gamma$ with $|\gamma| = m$, its weak derivatives satisfy $\partial^\gamma u \in L^p_{n/p, 1}(\R^n)$. Applying part (i) to these derivatives ensures $\partial^\gamma u \in C_0(\R^n)$, which directly yields $u \in C^m_0(\R^n)$. 

Furthermore, \cite{opic2000bessel} provides the precise modulus of continuity for the derivatives in this critical regime:
\[
\|\partial^\gamma u(\cdot + h) - \partial^\gamma u(\cdot)\|_{C_0} \lesssim \|\partial^\gamma u\|_{L^p_{n/p,1}} |\ln|h||^{\frac{1}{p'} - \alpha} \quad \text{as } |h| \to 0.
\]
Substituting our specific exponent $\alpha = 1$ and using the relation $\frac{1}{p'} - 1 = -\frac{1}{p}$ alongside the following fact
\[
\|\partial^\gamma u\|_{L^p_{n/p,1}} \lesssim \|u\|_{L^p_{n/p+m,1}} \simeq \|u\|_{\mathcal{L}^p_{s+\ln,\lambda}},
\]
we obtain the desired logarithmic modulus of continuity.
\end{proof}

\begin{remark}\label{rmk:blowup-sequence-p1}
The endpoint $p=1$ cannot be included in the critical $C_0$-embedding statement.
In fact, in the critical relation $2s=\frac{n}{p}$, taking $p=1$ forces $2s=n$,
i.e.\ $s=\frac{n}{2}$. By Proposition~\ref{prop:kernel_asymptotics}\textnormal{(ii)},
the logarithmic Bessel kernel $K_{\frac n2+\ln}^{\lambda}$ has a double-logarithmic
blow-up at the origin and, in particular, is unbounded in every neighborhood of $0$.

Let $\{\eta_k\}_{k\ge1}\subset C_c^\infty(\mathbb R^n)$ be a standard approximation of the identity.
That is, $\eta_k\ge0$, $\int_{\mathbb R^n}\eta_k(x)\,dx=1$ for all $k$, and for every $\varepsilon>0$,
\[
\int_{B_\varepsilon}\eta_k(x)\,dx \longrightarrow 1 \qquad \text{as }k\to\infty,
\]
equivalently, $\eta_k\rightharpoonup \delta_0$ in the sense of distributions. Define
\[
u_k:=K_{\frac n2+\ln}^{\lambda}*\eta_k\in L^1(\mathbb{R}^n).
\]
Then $\|\eta_k\|_{L^1}=1$ for all $k$, hence the data are uniformly bounded in $L^1$.
However, since $K_{\frac n2+\ln}^{\lambda}$ is nonnegative and unbounded at $0$,
one can choose $k\to\infty$ so that the mass of $\eta_k$ concentrates in an
arbitrarily small neighborhood of the origin, and consequently
\[
u_k(0)=\int_{\mathbb R^n}K_{\frac n2+\ln}^{\lambda}(y)\,\eta_k(y)\,dy \longrightarrow +\infty.
\]
In particular, $\|u_k\|_{L^\infty(\mathbb R^n)}\ge u_k(0)\to\infty$, which shows that
the  operator $f\mapsto K_{\frac n2+\ln}^{\lambda}*f$ cannot be bounded from
$L^1(\mathbb R^n)$ into $L^\infty(\mathbb R^n)$, and hence no continuous embedding into
$C_0(\mathbb R^n)$ can hold at $p=1$.
\end{remark}

We prove the compact embeddings at the critical line $2s=\frac{n}{p}$, treating first
the local compactness on bounded domains and then the global compactness under radial
symmetry.

\begin{proof}[\textbf{Proof of Theorem \ref{thm:compact_critical}.}]
\textbf{(i).} Assume $s = \frac{n}{2p}$. Let $\{u_k\}_{k=1}^\infty$ be a bounded sequence in $\mathcal{L}^{p}_{s+\ln,\lambda}(\R^n)$, meaning there exists a constant $M > 0$ such that $\|u_k\|_{\mathcal{L}^{p}_{s+\ln,\lambda}} \le M$ for all $k \ge 1$. 

By Theorem \ref{cor:critical_embeddings} (i), the space continuously embeds into $C_0(\R^n)$. Thus, the sequence is uniformly bounded on $\R^n$:
\[
\sup_{k \ge 1} \|u_k\|_{C(\overline{\Omega})} \le \sup_{k \ge 1} \|u_k\|_{L^\infty(\R^n)} \le C \sup_{k \ge 1} \|u_k\|_{\mathcal{L}^{p}_{s+\ln,\lambda}} \le CM.
\]
By Theorem \ref{cor:critical_embeddings}, for any $x, y \in \overline{\Omega}$ with $|x - y|$ sufficiently small,
\[
|u_k(x) - u_k(y)| \le C \|u_k\|_{L^p_{n/p, 1}} |\ln|x - y||^{\frac{1}{p'} - 1} \le C' M |\ln|x - y||^{-\frac{1}{p}}.
\]
Since the right-hand side converges to $0$ as $|x - y| \to 0$ independently of $k$, the sequence $\{u_k\}$ is uniformly equicontinuous on the compact set $\overline{\Omega}$. 

By the Arzelà-Ascoli theorem, there exists a subsequence $\{u_{k_j}\}$ that converges uniformly in $C(\overline{\Omega})$. This proves the compact embedding $\mathcal{L}^{p}_{s+\ln,\lambda}(\R^n) \hookrightarrow\hookrightarrow C(\overline{\Omega})$. 
Furthermore, since $\Omega$ is bounded, the embedding $C(\overline{\Omega}) \hookrightarrow L^q(\Omega)$ is continuous for any $1 \le q \le \infty$. The composition of a compact operator with a continuous operator is compact, yielding $\mathcal{L}^{p}_{s+\ln,\lambda}(\R^n) \hookrightarrow\hookrightarrow L^q(\Omega)$.

\textbf{(ii).}
Let $\{u_k\}_{k=1}^\infty$ be a bounded sequence in $\mathcal{L}^{p}_{s+\ln,\lambda,\mathrm{rad}}(\mathbb{R}^n)$, i.e.
\[
\sup_{k\ge 1}\|u_k\|_{\mathcal{L}^{p}_{s+\ln,\lambda}}\le M.
\]
Since $\mathcal{L}^{p}_{s+\ln,\lambda}(\mathbb{R}^n)\hookrightarrow L^p(\mathbb{R}^n)$ continuously, there exists $C_p>0$ such that
\[
\|u_k\|_{L^p(\mathbb{R}^n)}\le C_p M,\qquad k\ge 1.
\]
By Theorem \ref{cor:critical_embeddings}(i), $\{u_k\}$ is uniformly bounded in $C_0(\mathbb{R}^n)$ and uniformly equicontinuous on $\mathbb{R}^n$, with the logarithmic modulus of continuity: there exist constants $C>0$ and $\rho_0\in(0,1)$ such that for all $k$,
\[
|u_k(x)-u_k(y)|\le C M\,|\ln|x-y||^{-1/p}
\quad\text{whenever }0<|x-y|<\rho_0.
\]

\medskip
\noindent\textbf{Step 1: Local uniform convergence.}
Fix $R>0$. The restrictions $\{u_k|_{\overline{B_R(0)}}\}$ are uniformly bounded and equicontinuous on the compact set $\overline{B_R(0)}$, hence relatively compact in $C(\overline{B_R(0)})$ by the Arzel\`a--Ascoli theorem. Applying a Cantor diagonal argument over $R=1,2,\dots$, we extract a subsequence (still denoted $\{u_k\}$) and a function $u$ such that
\[
u_k\to u\quad\text{uniformly on each }\overline{B_R(0)}.
\]
In particular, $u\in C(\mathbb{R}^n)$. Moreover $u$ is radial: for any rotation $Q\in SO(n)$ and any $x\in\mathbb{R}^n$, since $u_k(Qx)=u_k(x)$ for all $k$ and the convergence is uniform on $\overline{B_{|x|}(0)}$, we obtain
\[
u(Qx)=\lim_{k\to\infty}u_k(Qx)=\lim_{k\to\infty}u_k(x)=u(x).
\]

\medskip
\noindent\textbf{Step 2: Uniform decay at infinity.}
We claim that $\{u_k\}$ decays uniformly at infinity: for every $\eta>0$ there exists $R>0$ such that
\[
|u_k(x)|<\eta\qquad\text{for all }|x|\ge R\text{ and all }k.
\]
Suppose not, then there exist $\varepsilon>0$, a subsequence $\{u_{k_j}\}$, and radii $r_j\to\infty$ such that
\[
|u_{k_j}(r_j)|\ge \varepsilon\qquad\text{for all }j.
\]
Choose $\delta\in(0,\rho_0)$ so that
\[
CM\,|\ln\delta|^{-1/p}\le \varepsilon/2.
\]
Fix $j$ and choose any unit vector $e\in\mathbb{S}^{n-1}$. For $r\in[r_j,r_j+\delta]$ we have
$|(re)-(r_je)|=|r-r_j|\le\delta$, hence by the modulus of continuity and radiality,
\[
|u_{k_j}(r)-u_{k_j}(r_j)|
=
|u_{k_j}(re)-u_{k_j}(r_je)|
\le CM\,|\ln|r-r_j||^{-1/p}
\le CM\,|\ln\delta|^{-1/p}\le \varepsilon/2.
\]
Therefore,
\[
|u_{k_j}(r)|\ge \varepsilon/2\qquad\text{for all }r\in[r_j,r_j+\delta].
\]
Let $A_j:=\{x\in\mathbb{R}^n:\ r_j\le |x|\le r_j+\delta\}$. For $j$ large we may assume $r_j\ge 1$, and then
\begin{align*}
\|u_{k_j}\|_{L^p(\mathbb{R}^n)}^p
&\ge \int_{A_j}|u_{k_j}(x)|^p\,dx
=|\mathbb{S}^{n-1}|\int_{r_j}^{r_j+\delta}|u_{k_j}(r)|^p\,r^{n-1}\,dr \\
&\ge |\mathbb{S}^{n-1}|\left(\frac{\varepsilon}{2}\right)^p\int_{r_j}^{r_j+\delta} r^{n-1}\,dr
\ge |\mathbb{S}^{n-1}|\left(\frac{\varepsilon}{2}\right)^p r_j^{\,n-1}\delta.
\end{align*}
Since $n\ge 2$ and $r_j\to\infty$, the right-hand side tends to $+\infty$, contradicting the uniform bound
$\|u_{k}\|_{L^p(\mathbb{R}^n)}\le C_pM$. This proves the claimed uniform decay.

\medskip
\noindent\textbf{Step 3: Global uniform convergence and compactness in $C_0$.}
Fix $\varepsilon>0$. By Step~2 choose $R>0$ such that
\[
\sup_{k\ge 1}\sup_{|x|\ge R}|u_k(x)|\le \varepsilon/3.
\]
Since $u_k\to u$ uniformly on $\overline{B_R(0)}$, there exists $K$ such that
\[
\sup_{|x|\le R}|u_k(x)-u(x)|\le \varepsilon/3,\qquad k\ge K.
\]
Moreover, for every $|x|\ge R$ we have
\[
|u(x)|=\lim_{m\to\infty}|u_{k_m}(x)|
\le \sup_{m}\sup_{|y|\ge R}|u_{k_m}(y)|
\le \varepsilon/3,
\]
hence $\sup_{|x|\ge R}|u(x)|\le \varepsilon/3$. Therefore, for $k\ge K$,
\[
\sup_{|x|\ge R}|u_k(x)-u(x)|
\le \sup_{|x|\ge R}|u_k(x)|+\sup_{|x|\ge R}|u(x)|
\le \frac{2\varepsilon}{3}.
\]
Combining the estimates on $|x|\le R$ and $|x|\ge R$ yields
\[
\sup_{x\in\mathbb{R}^n}|u_k(x)-u(x)|\le \varepsilon,\qquad k\ge K,
\]
so $u_k\to u$ in $L^\infty(\mathbb{R}^n)$ and $u\in C_0(\mathbb{R}^n)$. This proves the compact
embedding $\mathcal{L}^{p}_{s+\ln,\lambda,\mathrm{rad}}(\mathbb{R}^n)\hookrightarrow\hookrightarrow C_0(\mathbb{R}^n)$.

\medskip
\noindent\textbf{Step 4: Compactness in $L^q$.}
If $p<q<\infty$, the interpolation inequality gives
\[
\|u_k-u\|_{L^q(\mathbb{R}^n)}
\le \|u_k-u\|_{L^\infty(\mathbb{R}^n)}^{\,1-\frac{p}{q}}
\|u_k-u\|_{L^p(\mathbb{R}^n)}^{\,\frac{p}{q}}.
\]
The sequence $\{u_k-u\}$ is bounded in $L^p(\mathbb{R}^n)$ and converges to $0$ in $L^\infty(\mathbb{R}^n)$,
hence the right-hand side tends to $0$, and thus $u_k\to u$ strongly in $L^q(\mathbb{R}^n)$.
The case $q=\infty$ is exactly the convergence in $C_0(\mathbb{R}^n)$ proved above.
\end{proof}

At last, we prove the compact embeddings at the critical line $2s<\frac{n}{p}$.

\begin{proof}[\textbf{Proof of Theorem \ref{thm:compact_subcritical}.}]
\noindent\textbf{(i) Local Compactness.}
By definition, for any $u \in \mathcal{L}^{p}_{s+\ln,\lambda}(\R^n)$, there exists $f \in L^p(\R^n)$ such that $u = K_{s+\ln}^\lambda * f$ with $$\|u\|_{\mathcal{L}^p_{s+\ln,\lambda}} = \|f\|_{L^p(\R^n)}.$$
Let $Tf := (K_{s+\ln}^\lambda * f)\big|_\Omega$. We show that $T: L^p(\R^n) \to L^{p^*}(\Omega)$ is a compact operator.

Notice that our subcritical assumption $s < \frac{n}{2p}$, combined with $p > 1$, strictly implies $2s < \frac{n}{p} < n$. According to the  Proposition~\ref{prop:kernel_asymptotics}, the kernel behaves as 
\[
K_{s+\ln}^{\lambda}(x) \sim  \frac{\Gamma(\frac{n}{2}-s)}{\pi^{n/2} 2^{2s} \Gamma(s)}\frac{|x|^{2s-n}}{-2\ln|x|} \quad \text{as } |x| \to 0.
\]
Thus, there exists a radius $r_0 \in (0, 1/2)$ and a constant $C > 0$ such that for all $0 < |x| < r_0$, we have $$0 \le K_{s+\ln}^{\lambda}(x) \le C |x|^{2s-n} (-\ln|x|)^{-1}.$$

For any small $\epsilon \in (0, r_0)$, we decompose the kernel into a highly singular local part and a bounded tail part:
\[
K_{s+\ln}^\lambda = K_\epsilon + K^\epsilon, \quad \text{where } K_\epsilon = K_{s+\ln}^\lambda \chi_{\{|x| < \epsilon\}} \text{ and } K^\epsilon = K_{s+\ln}^\lambda \chi_{\{|x| \ge \epsilon\}}.
\]
This splits the operator as $T = T_\epsilon + T^\epsilon$, where $T_\epsilon f = (K_\epsilon * f)|_\Omega$ and $T^\epsilon f = (K^\epsilon * f)|_\Omega$.

\medskip
\noindent\textbf{Estimate of $T_\epsilon$:} 
For $|x| < \epsilon < 1/2$, the logarithmic term implies $(-\ln|x|)^{-1} < (-\ln\epsilon)^{-1}$, thus
\[
K_\epsilon(x) \le \frac{C}{-\ln\epsilon} |x|^{2s-n}.
\]
Since $0 < 2s < n$ and the exponents satisfy $\frac{1}{p} - \frac{1}{p^*} = \frac{2s}{n}$, the classical Hardy-Littlewood-Sobolev (HLS) inequality (\cite[Chapter 5, Theorem 1]{stein1970singular}) guarantees that convolution with $|x|^{2s-n}$ is bounded from $L^p(\R^n)$ to $L^{p^*}(\R^n)$ with a universal constant $C_{HLS}$. Thus,
\[
\|T_\epsilon f\|_{L^{p^*}(\Omega)} \le \|K_\epsilon * f\|_{L^{p^*}(\R^n)} \le \frac{C \cdot C_{HLS}}{-\ln\epsilon} \|f\|_{L^p(\R^n)}.
\]
Crucially, this implies the operator norm $\|T_\epsilon\|_{L^p \to L^{p^*}} \le \frac{\tilde{C}}{-\ln\epsilon}$. As $\epsilon \to 0^+$, the norm of $T_\epsilon$ converges to $0$.

\medskip
\noindent\textbf{Compactness of $T^\epsilon$:}
For any fixed $\epsilon > 0$, the kernel $K^\epsilon(x)$ is bounded, continuous, and decays exponentially at infinity. Thus, $K^\epsilon$ and its gradient $\nabla K^\epsilon$ belong to $L^{p'}(\R^n)$. 
By Hölder's inequality, for any bounded sequence $\{f_k\}$ in $L^p(\R^n)$ bounded by $M$, the sequence $u_k^\epsilon = T^\epsilon f_k$ satisfies
\begin{align*}
\|u_k^\epsilon\|_{L^\infty(\Omega)} &\le M \|K^\epsilon\|_{L^{p'}(\R^n)}, \\
\|\nabla u_k^\epsilon\|_{L^\infty(\Omega)} &\le M \|\nabla K^\epsilon\|_{L^{p'}(\R^n)}.
\end{align*}
This implies that $\{u_k^\epsilon\}$ is uniformly bounded and equicontinuous on $\overline{\Omega}$. By the Arzelà-Ascoli theorem, $\{u_k^\epsilon\}$ is relatively compact in $C(\overline{\Omega})$, and since $\Omega$ is bounded, it is relatively compact in $L^{p^*}(\Omega)$. Thus, $T^\epsilon$ is a compact operator. Since $T$ is the uniform limit of the compact operators $T^\epsilon$ as $\|T - T^\epsilon\| \to 0$, $T$ is also compact. This proves the local compact embedding.

\medskip
\noindent\textbf{(ii) Global radial compactness.}
Let $\{u_k\}$ be a bounded sequence in $\mathcal L^{p}_{s+\ln,\lambda,\mathrm{rad}}(\mathbb R^n)$ with $$\sup_{k \ge 1} \|u_k\|_{\mathcal L^{p}_{s+\ln,\lambda}} \le M.$$ 
We aim to show that $\{u_k\}$ admits a strongly convergent subsequence in $L^{p^*}(\mathbb R^n)$.

\textbf{Step 1: Local strong convergence.} 
By part (i), for any fixed radius $R > 0$, the restriction operator from $\mathcal L^{p}_{s+\ln,\lambda}(\mathbb R^n)$ to $L^{p^*}(B_R(0))$ is strictly compact. Thus, utilizing a standard Cantor diagonal extraction argument over a sequence of expanding balls $B_1(0) \subset B_2(0) \subset \cdots$, we can extract a common subsequence, which we still denote by $\{u_k\}$, and find a function $u \in L^{p^*}(\mathbb R^n)$, such that 
\[
u_k \to u \quad \text{strongly in } L^{p^*}(B_R(0)) \text{ for every } R > 0.
\]

\textbf{Step 2: Uniform control of the tails.}
By Lemma~\ref{lem:kernel_Lr}, the kernel $K := K_{s+\ln}^\lambda$ belongs to $L^r(\R^n)$ where $\frac{1}{p^*} = \frac{1}{p} + \frac{1}{r} - 1$. For any fixed small $\varepsilon > 0$, there exists a radius $\delta > 0$ such that the local truncated singular kernel $K_\delta := K \chi_{B_\delta(0)}$ satisfies
\[
\|K_\delta\|_{L^r(\R^n)} \le \frac{\varepsilon}{2M}.
\]
There exists radial functions $\{f_k\}\subset L^p(\mathbb{R}^n)$ such that 
\[
u_k = K_\delta * f_k + K^\delta * f_k =: u_k^\delta + v_k^\delta, \quad \text{where } K^\delta := K \chi_{\R^n \setminus B_\delta(0)}.
\]

By Young's convolution inequality, we have a uniform global bound:
\[
\|u_k^\delta\|_{L^{p^*}(\R^n)} \le \|K_\delta\|_{L^r(\R^n)} \|f_k\|_{L^p(\R^n)} \le \frac{\varepsilon}{2M} M = \frac{\varepsilon}{2}.
\]

By the properties of the heat kernel representation, $K^\delta$ and its weak derivatives decay exponentially at infinity and are globally bounded. In particular, $K^\delta, \nabla K^\delta \in L^1(\R^n)$. 
Consequently, by Young's inequality, $v_k^\delta \in W^{1,p}(\R^n)$ with a uniformly bounded norm:
\[
\|v_k^\delta\|_{W^{1,p}(\R^n)} \le \big( \|K^\delta\|_{L^1} + \|\nabla K^\delta\|_{L^1} \big) \|f_k\|_{L^p} \le C_\delta M.
\]
Crucially, since both $K^\delta$ and $f_k$ are radially symmetric, $v_k^\delta$ is a radial function in the classical Sobolev space $W^{1,p}(\R^n)$. For this highly regularized part, the classical Strauss Radial Lemma (see \cite[Proposition 1.1]{yuan2013radial}) becomes completely legitimate and rigorously applicable. For $n \ge 2$ and $|x| \ge 1$, we have:
\[
|v_k^\delta(x)| \le C \|v_k^\delta\|_{W^{1,p}(\R^n)} |x|^{-\frac{n-1}{p}} \le \tilde{C}_\delta M |x|^{-\frac{n-1}{p}}.
\]
Using this pointwise decay, we estimate the $L^{p^*}$-tail of $v_k^\delta$ outside a large ball $B_R(0)$:
\[
\int_{|x| > R} |v_k^\delta(x)|^{p^*} \,dx \le \left( \sup_{|x| > R} |v_k^\delta(x)| \right)^{p^* - p} \int_{|x| > R} |v_k^\delta(x)|^p \,dx \le \left( \tilde{C}_\delta M R^{-\frac{n-1}{p}} \right)^{p^* - p} \|v_k^\delta\|_{L^p(\R^n)}^p.
\]
Because $p^* > p$ and $n \ge 2$, the exponent $-\frac{n-1}{p}(p^* - p)$ is strictly negative. Thus, the right-hand side vanishes as $R \to \infty$. We can choose $R_0 \ge 1$ large enough such that for all $R \ge R_0$ and all $k$, $$\|v_k^\delta\|_{L^{p^*}(\R^n \setminus B_R(0))} \le \frac{\varepsilon}{2}.$$

\textbf{Combining for the uniform tail bound:}
For $R \ge R_0$ and all $k \ge 1$, we have
\[
\|u_k\|_{L^{p^*}(\R^n \setminus B_R(0))} \le \|u_k^\delta\|_{L^{p^*}(\R^n \setminus B_R(0))} + \|v_k^\delta\|_{L^{p^*}(\R^n \setminus B_R(0))} \le \frac{\varepsilon}{2} + \frac{\varepsilon}{2} = \varepsilon.
\]
\textbf{Step 3: Global strong convergence.}
Given any $\varepsilon > 0$, we can choose $R$ sufficiently large such that the uniform tail estimate yields
\[
\|u_k\|_{L^{p^*}(\mathbb R^n \setminus B_R(0))} \le \frac{\varepsilon}{3} \quad \text{for all } k \ge 1.
\]
By Fatou's lemma or weak lower semicontinuity, the limit function $u$ inherits the same tail bound:
\[
\|u\|_{L^{p^*}(\mathbb R^n \setminus B_R(0))} \le \frac{\varepsilon}{3}.
\]
Now, fixing this $R$, the local strong convergence from Step 1 ensures that there exists an integer $K$ such that for all $k \ge K$,
\[
\|u_k - u\|_{L^{p^*}(B_R(0))} \le \frac{\varepsilon}{3}.
\]
Combining these estimates via the triangle inequality, we conclude that for all $k \ge K$,
\begin{align*}
\|u_k - u\|_{L^{p^*}(\mathbb R^n)} &\le \|u_k - u\|_{L^{p^*}(B_R(0))} + \|u_k\|_{L^{p^*}(\mathbb R^n \setminus B_R(0))} + \|u\|_{L^{p^*}(\mathbb R^n \setminus B_R(0))}\le \varepsilon.
\end{align*}
This proves that $u_k \to u$ strongly in $L^{p^*}(\mathbb R^n)$, thereby establishing the global compact embedding 
\[
\mathcal L^{p}_{s+\ln,\lambda,\mathrm{rad}}(\mathbb R^n) \hookrightarrow\hookrightarrow L^{p^*}(\mathbb R^n),
\]
thus, we complete the proof.
\end{proof}

\section{Acknowledgements}

The author is deeply grateful to Prof.~Petr Gurka and Prof.~Bobo Hua
for their constant encouragement and valuable suggestions throughout the
preparation of this work.



		\printbibliography

@article{brooks1969representations,
  title     = {Representations of Weak and Strong Integrals in Banach Spaces},
  author    = {Brooks, James K.},
  journal   = {Proceedings of the National Academy of Sciences},
  volume    = {63},
  number    = {2},
  pages     = {266--270},
  year      = {1969}
}

@book{stein1970singular,
  title     = {Singular Integrals and Differentiability Properties of Functions},
  author    = {Stein, Elias M.},
  series    = {Princeton Mathematical Series},
  volume    = {30},
  publisher = {Princeton University Press},
  year      = {1970}
}

@book{Triebel1983Theory,
  title     = {Theory of Function Spaces},
  author    = {Triebel, Hans},
  series    = {Monographs in Mathematics},
  volume    = {78},
  publisher = {Birkh{\"a}user Verlag},
  address   = {Basel},
  year      = {1983},
  pages     = {284}
}

@book{talagrand1984pettis,
  title     = {Pettis Integral and Measure Theory},
  author    = {Talagrand, Michel},
  volume    = {307},
  publisher = {American Mathematical Society},
  address   = {Providence, RI},
  year      = {1984}
}

@article{1996sharpness,
  title     = {Sharpness of Embeddings in Logarithmic Bessel-Potential Spaces},
  author    = {Edmunds, David E. and Gurka, Petr and Opic, Bohum{\'\i}r},
  journal   = {Proceedings of the Royal Society of Edinburgh Section A: Mathematics},
  volume    = {126},
  number    = {5},
  pages     = {995--1009},
  year      = {1996},
  publisher = {The Royal Society of Edinburgh}
}

@article{1997embeddings,
  title     = {On Embeddings of Logarithmic Bessel Potential Spaces},
  author    = {Edmunds, David E. and Gurka, Petr and Opic, Bohum{\'\i}r},
  journal   = {Journal of Functional Analysis},
  volume    = {146},
  number    = {1},
  pages     = {116--150},
  year      = {1997},
  publisher = {Elsevier}
}

@article{1998norms,
  title     = {Norms of Embeddings of Logarithmic Bessel Potential Spaces},
  author    = {Edmunds, David E. and Gurka, Petr and Opic, Bohum{\'\i}r},
  journal   = {Proceedings of the American Mathematical Society},
  volume    = {126},
  number    = {8},
  pages     = {2417--2425},
  year      = {1998}
}

@article{opic2000bessel,
  title     = {Bessel Potentials with Logarithmic Components and Sobolev-Type Embeddings},
  author    = {Opic, Bohum{\'\i}r and Trebels, Walter},
  journal   = {Analysis Mathematica},
  volume    = {26},
  number    = {4},
  pages     = {299--319},
  year      = {2000},
  publisher = {Springer}
}

@article{sickel2000radial,
  title     = {Radial Subspaces of Besov and Lizorkin--Triebel Classes: Extended Strauss Lemma and Compactness of Embeddings},
  author    = {Sickel, Winfried and Skrzypczak, Leszek},
  journal   = {Journal of Fourier Analysis and Applications},
  volume    = {6},
  number    = {6},
  pages     = {639--662},
  year      = {2000},
  publisher = {Springer}
}

@article{edmunds2000optimality,
  title     = {Optimality of Embeddings of Logarithmic Bessel Potential Spaces},
  author    = {Edmunds, David E. and Gurka, Petr and Opic, Bohum{\'\i}r},
  journal   = {The Quarterly Journal of Mathematics},
  volume    = {51},
  number    = {2},
  pages     = {185--209},
  year      = {2000},
  publisher = {Oxford University Press}
}

@article{edmunds2005compact,
  title     = {Compact and Continuous Embeddings of Logarithmic Bessel Potential Spaces},
  author    = {Edmunds, David E. and Gurka, Petr and Opic, Bohum{\'\i}r},
  journal   = {Studia Mathematica},
  volume    = {168},
  number    = {3},
  pages     = {229--250},
  year      = {2005}
}

@book{triebel2006theory,
  title     = {Theory of Function Spaces {III}},
  author    = {Triebel, Hans},
  publisher = {Birkh{\"a}user},
  year      = {2006}
}

@article{edmunds2006non,
  title     = {Non-compact and Sharp Embeddings of Logarithmic Bessel Potential Spaces into H{\"o}lder-type Spaces},
  author    = {Edmunds, David E. and Gurka, Petr and Opic, Bohum{\'\i}r},
  journal   = {Zeitschrift f{\"u}r Analysis und ihre Anwendungen},
  volume    = {25},
  number    = {1},
  pages     = {73--80},
  year      = {2006}
}

@article{silvestre2007regularity,
  title     = {Regularity of the Obstacle Problem for a Fractional Power of the Laplace Operator},
  author    = {Silvestre, Luis},
  journal   = {Communications on Pure and Applied Mathematics},
  volume    = {60},
  number    = {1},
  pages     = {67--112},
  year      = {2007},
  publisher = {Wiley}
}

@article{caffarelli2007extension,
  title     = {An Extension Problem Related to the Fractional Laplacian},
  author    = {Caffarelli, Luis and Silvestre, Luis},
  journal   = {Communications in Partial Differential Equations},
  volume    = {32},
  number    = {8},
  pages     = {1245--1260},
  year      = {2007},
  publisher = {Taylor \& Francis}
}

@article{caffarelli2008regularity,
  title     = {Regularity Estimates for the Solution and the Free Boundary of the Obstacle Problem for the Fractional Laplacian},
  author    = {Caffarelli, Luis A. and Salsa, Sandro and Silvestre, Luis},
  journal   = {Inventiones Mathematicae},
  volume    = {171},
  number    = {2},
  pages     = {425--461},
  year      = {2008},
  publisher = {Springer}
}

@book{grafakos2008classical,
  title     = {Classical Fourier Analysis},
  author    = {Grafakos, Loukas},
  publisher = {Springer},
  year      = {2008},
  volume    = {2}
}

@book{grafakos2009modern,
  title     = {Modern Fourier Analysis},
  author    = {Grafakos, Loukas and Ross, Kenneth A.},
  publisher = {Springer},
  year      = {2009},
  volume    = {250}
}

@article{caffarelli2010nonlocal,
  title     = {Nonlocal Minimal Surfaces},
  author    = {Caffarelli, Luis A. and Roquejoffre, Jean-Michel and Savin, Ovidiu},
  journal   = {Communications on Pure and Applied Mathematics},
  volume    = {63},
  number    = {9},
  pages     = {1111--1144},
  year      = {2010}
}

@article{chang2011fractional,
  title     = {Fractional Laplacian in Conformal Geometry},
  author    = {Chang, Sun-Yung Alice and Gonzalez, Maria del Mar},
  journal   = {Advances in Mathematics},
  volume    = {226},
  number    = {2},
  pages     = {1410--1432},
  year      = {2011},
  publisher = {Elsevier}
}

@book{murray2012asymptotic,
  title     = {Asymptotic Analysis},
  author    = {Murray, J. D.},
  publisher = {Springer Science \& Business Media},
  year      = {2012}
}

@article{yuan2013radial,
  title     = {The Radial Lemma of Strauss in the Context of Morrey Spaces},
  author    = {Sickel, Winfried and Yang, Dachun and Yuan, Wen},
  journal   = {Annales Fennici Mathematici},
  volume    = {39},
  pages     = {417--442},
  year      = {2014}
}

@article{RosOtonSerra2014,
  title     = {The Dirichlet Problem for the Fractional Laplacian: Regularity up to the Boundary},
  author    = {Ros-Oton, Xavier and Serra, Joaquim},
  journal   = {Journal de Math{\'e}matiques Pures et Appliqu{\'e}es},
  volume    = {101},
  number    = {3},
  pages     = {275--302},
  year      = {2014}
}

@article{RosOtonSerra2015,
  title     = {Nonexistence Results for Nonlocal Equations with Critical and Supercritical Nonlinearities},
  author    = {Ros-Oton, Xavier and Serra, Joaquim},
  journal   = {Communications in Partial Differential Equations},
  volume    = {40},
  number    = {1},
  pages     = {115--133},
  year      = {2015}
}

@article{cobos2015besov,
  title     = {On Besov Spaces of Logarithmic Smoothness and Lipschitz Spaces},
  author    = {Cobos, Fernando and Dom{\'\i}nguez, {\'O}scar},
  journal   = {Journal of Mathematical Analysis and Applications},
  volume    = {425},
  number    = {1},
  pages     = {71--84},
  year      = {2015},
  publisher = {Elsevier}
}

@article{CaffarelliRosOtonSerra2017,
  title     = {Obstacle Problems for Integro-Differential Operators: Regularity of Solutions and Free Boundaries},
  author    = {Caffarelli, Luis and Ros-Oton, Xavier and Serra, Joaquim},
  journal   = {Inventiones Mathematicae},
  volume    = {208},
  number    = {3},
  pages     = {1155--1211},
  year      = {2017}
}

@article{chen2019dirichlet,
  title     = {The Dirichlet Problem for the Logarithmic Laplacian},
  author    = {Chen, Huyuan and Weth, Tobias},
  journal   = {Communications in Partial Differential Equations},
  volume    = {44},
  number    = {11},
  pages     = {1100--1139},
  year      = {2019},
  publisher = {Taylor \& Francis}
}

@article{jarohs2020poisson,
  title     = {A New Look at the Fractional Poisson Problem via the Logarithmic Laplacian},
  author    = {Jarohs, Sven and Salda{\~n}a, Alberto and Weth, Tobias},
  journal   = {Journal of Functional Analysis},
  volume    = {279},
  number    = {11},
  pages     = {108732},
  year      = {2020},
  publisher = {Elsevier}
}

@article{frank2020classification,
  title={Classification of solutions of an equation related to a conformal log Sobolev inequality},
  author={Frank, Rupert L and K{\"o}nig, Tobias and Tang, Hanli},
  journal={Advances in Mathematics},
  volume={375},
  pages={107395},
  year={2020},
  publisher={Elsevier}
}

@book{dominguez2023function,
  title     = {Function Spaces of Logarithmic Smoothness: Embeddings and Characterizations},
  author    = {Dom{\'\i}nguez, {\'O}scar and Tikhonov, Sergey},
  publisher = {American Mathematical Society},
  volume    = {282},
  number    = {1393},
  year      = {2023}
}

@article{chenveron2024cauchy,
  title     = {The Cauchy Problem Associated to the Logarithmic Laplacian with an Application to the Fundamental Solution},
  author    = {Chen, Huyuan and V{\'e}ron, Laurent},
  journal   = {Journal of Functional Analysis},
  volume    = {287},
  pages     = {110470},
  year      = {2024},
  publisher = {Elsevier}
}

@book{abatangelo2025gentle,
  title     = {A Gentle Invitation to the Fractional World},
  author    = {Abatangelo, Nicola and Dipierro, Serena and Valdinoci, Enrico and others},
  publisher = {Springer},
  year      = {2025}
}

@article{gerontogiannis2025ahlfors,
  title     = {The Logarithmic Dirichlet Laplacian on {A}hlfors Regular Spaces},
  author    = {Gerontogiannis, D. M.},
  journal   = {Transactions of the American Mathematical Society},
  volume    = {378},
  number    = {1},
  pages     = {1--42},
  year      = {2025},
  publisher = {American Mathematical Society}
}

@article{ch2025logarithmic,
  title     = {Logarithmic Laplacian on General Riemannian Manifolds},
  author    = {Chen, Rui},
  journal   = {},
  volume    = {arXiv:2506.19311},
  year      = {2025}
}

@article{chen2025logarithmic,
  title     = {The Logarithmic Laplacian on General Graphs},
  author    = {Chen, Rui and Xu, Wendi},
  journal   = {arXiv preprint},
  volume    = {arXiv:2507.05936},
  year      = {2025}
}

@article{bellido2025compact,
  title     = {Compact Embeddings of Bessel Potential Spaces},
  author    = {Bellido, Jos{\'e} C. and Cueto, Javier and Garc{\'\i}a-S{\'a}ez, Guillermo},
  journal   = {arXiv preprint},
  volume    = {arXiv:2506.01677},
  year      = {2025}
}

@article{chen2026fractional,
  title     = {The Fractional-Logarithmic Laplacian: Fundamental Properties and Eigenvalues},
  author    = {Chen, Huyuan and Chen, Rui and Hauer, Daniel},
  journal   = {arXiv preprint},
  volume    = {arXiv:2602.06581},
  year      = {2026}
}

@article{chen2026poincar,
  title     = {On Poincar{\'e}-Sobolev Level Involving Fractional {GJMS} Operators on Hyperbolic Space},
  author    = {Chen, Huyuan and Chen, Rui},
  journal   = {arXiv preprint},
  volume    = {arXiv:2602.03299},
  year      = {2026}
}

@article{RosOtonWeidner2026,
  title     = {Regularity for Nonlocal Equations with Local Neumann Boundary Conditions},
  author    = {Ros-Oton, Xavier and Weidner, Moritz},
  journal   = {Analysis \& PDE},
  volume    = {19},
  number    = {2},
  pages     = {353--411},
  year      = {2026}
}

		
		\noindent\textit{Rui Chen}: School of Mathematical Sciences, Fudan University,\\[1mm]
		Shanghai 200433,  China\\[1mm]
		Brandenburg University of Technology Cottbus--Senftenberg,\\[1mm]
		Cottbus 03046, Germany\\[1mm]
		\noindent\emph{Email:} \texttt{chenrui23@m.fudan.edu.cn}
	\end{document}